\documentclass[11pt]{article}
\usepackage{eurosym}
\usepackage{color}
\usepackage{fancyhdr}
\usepackage{graphicx}
\usepackage{mathrsfs}
\usepackage[margin=20mm]{geometry}
\usepackage{amsmath,amsfonts,amsthm,amsfonts,euscript,euscript,bbm,yfonts}
\usepackage{indentfirst}
\usepackage{url}
\usepackage{colonequals}
\usepackage{upgreek}
\usepackage{dsfont}
\usepackage{enumerate}
\usepackage{graphicx}
\usepackage{amssymb}

\providecommand{\U}[1]{\protect\rule{.1in}{.1in}}
\DeclareMathOperator*{\inset}{\mathbb{I}}

\newtheorem{proposition}{Proposition}[section]
\newtheorem{theorem}[proposition]{Theorem}

\newtheorem{lemma}[proposition]{Lemma}

\newtheorem{open problem}[proposition]{Open problem}

\newtheorem{definition}[proposition]{Definition}
\newtheorem{remark}[proposition]{Remark}
\newtheorem{example}[proposition]{Example}

\newtheorem{assumption}{Assumption}

\numberwithin{equation}{section}
\begin{document}

\title{Large deviations for configurations generated by Gibbs distributions
	with energy functionals consisting of singular interaction and weakly
	confining potentials}
\author{Paul Dupuis\thanks{%
		Research supported in part by AFOSR FA9550-12-1-0399}, Vaios Laschos\thanks{%
		Research supported in part by AFOSR FA9550-12-1-0399}, and Kavita Ramanan%
	\thanks{%
		Research supported in part by AFOSR FA9550-12-1-0399 and NSF DMS-1407504}}
\date{Division of Applied Mathematics \\
	Brown University \\
	Providence, RI\\
	[\baselineskip] \today}
\maketitle

\begin{abstract}
	We establish large deviation principles (LDPs) for empirical measures
	associated with a sequence of Gibbs distributions on $n$-particle
	configurations, each of which is defined in terms of an inverse temperature $%
	\beta_n$ and an energy functional consisting of a (possibly singular)
	interaction potential and a (possibly weakly) confining potential. Under
	fairly general assumptions on the potentials, we use a common framework to
	establish LDPs both with speeds $\beta_n/n \rightarrow \infty$, in which
	case the rate function is expressed in terms of a functional involving the
	potentials, and with speed $\beta_n =n$, when the rate function contains an
	additional entropic term. Such LDPs are motivated by questions arising in
	random matrix theory, sampling, simulated annealing and asymptotic convex
	geometry. Our approach, which uses the weak convergence method developed by
	Dupuis and Ellis,
	establishes LDPs with respect to stronger Wasserstein-type topologies. 
	Our results address several interesting examples not covered by previous works, 
	including the case of a weakly confining potential, which allows for rate
	functions with minimizers that do not have compact support, thus resolving several open
	questions raised in a work of Chafa\"{\i} et al.
\end{abstract}
\vspace{32pt}

\newpage

{ 
 \tableofcontents  } 
\newpage
\section{Introduction}

\subsection{Description of problem and contributions}

We consider configurations of a finite number of $\mathbb{R}^{d}$-valued
particles that are subject to an external force consisting of a confining
potential $V:\mathbb{R}^{d}\rightarrow (-\infty ,+\infty ]$ that acts on
each particle and a pairwise interaction potential $W:\mathbb{R}^{d}\times 
\mathbb{R}^{d}\rightarrow (-\infty ,+\infty ]$. For every $n\in \mathbb{N},$
we define a Hamiltonian or energy functional $H_{n}:\mathbb{R}%
^{dn}\rightarrow (-\infty ,+\infty ]$ that assigns to every
$\mathbb{R}^{dn}$-valued 
configuration $\mathbf{x}^{n}=(\mathbf{x}_{1},\mathbf{x}_{2},\ldots ,\mathbf{x}_{n})$ of $n$ particles, the energy 
\begin{equation}
H_{n}(\mathbf{x}^{n})=H_{n}\left( \mathbf{x}_{1},...,\mathbf{x}_{n}\right) :=%
\frac{1}{n}\sum_{i=1}^{n}V\left( \mathbf{x}_{i}\right) +\frac{1}{2n^{2}}%
\sum_{i,j=1,i\neq j}^{n}W\left( \mathbf{x}_{i},\mathbf{x}_{j}\right) .
\label{def-Hn}
\end{equation}%
Also, for any $n$-particle configuration $\mathbf{x}^{n}=(\mathbf{x}_{1},%
\mathbf{x}_{2},\ldots ,\mathbf{x}_{n})$, let $L_{n}(\mathbf{x}^{n},\cdot )$
be the associated empirical measure: 
\begin{equation}
L_{n}\left( \mathbf{x}^{n};\cdot \right) :=\frac{1}{n}\sum_{i=1}^{n}\delta _{%
	\mathbf{x}_{i}}(\cdot ),  \label{Ln}
\end{equation}%
where $\delta _{\mathbf{y}}$ represents the Dirac delta mass at $\mathbf{y}%
\in \mathbb{R}^{d}$. Given a separable metric space $S$, let $\mathcal{B}(S)$
denote the collection of Borel subsets of $S$, and let $\mathcal{P}(S)$
denote the space of probability measures on $(S,\mathcal{B}(S))$. Note that
for every $\mathbf{x}^{n}\in \mathbb{R}^{d}$, $L_{n}(\mathbf{x}^{n};\cdot )$
lies in $\mathcal{P}(\mathbb{R}^{d})$, where $\mathbb{R}^{d}$ is equipped
with the usual Euclidean metric. If $\mathbf{x}^{n}$ is random and each
component of $\mathbf{x}^{n}$ has a density that is absolutely continuous
with respect to a measure with no atoms (which will be true in this
article), $H_{n}$ can be rewritten in terms of $L_{n}$ as follows: 
\begin{equation}
\begin{split}
H_{n}(\mathbf{x}^{n})& =\int_{\mathbb{R}^{d}}V\left( \mathbf{x}\right)
L_{n}\left( \mathbf{x}^{n};d\mathbf{x}\right) +\frac{1}{2}\int_{(\mathbb{R}%
	^{d}\times \mathbb{R}^{d})_{\neq }}\!\!\!\!\!\!\!\!\!\!\!\!W\left( \mathbf{x}%
,\mathbf{y}\right) L_{n}\left( \mathbf{x}^{n};d\mathbf{x}\right) L_{n}\left( 
\mathbf{x}^{n};d\mathbf{y}\right) \\
& =\frac{1}{n}\int_{\mathbb{R}^{d}}V\left( \mathbf{x}\right) L_{n}\left( 
\mathbf{x}^{n};d\mathbf{x}\right) +\frac{1}{2}\int_{(\mathbb{R}^{d}\times 
	\mathbb{R}^{d})_{\neq }}\!\!\!\!\!\!\!\!\!\!\!\!\!\!(W\left( \mathbf{x},%
\mathbf{y}\right) +V(\mathbf{x})+V(\mathbf{y}))L_{n}\left( \mathbf{x}^{n};d%
\mathbf{x}\right) L_{n}\left( \mathbf{x}^{n};d\mathbf{y}\right) ,
\end{split}
\label{def-Hnequiv}
\end{equation}%
where for  { $k \in \mathbb{N}$} and a set $A\subset \mathbb{R}^{kd},$ the symbol $A_{\neq }$ denotes
the set of points in $A$ whose $d$-dimensional components are all distinct:

	\begin{equation}\label{neqdef}%
{	A_{\neq }:=A\setminus \left\{ (\mathbf{x}_{1},\ldots ,\mathbf{x}_{k})\in 
	\mathbb{R}^{kd}:\mathbf{x}_{i}=\mathbf{x}_{j}\text{ for some }1\leq i\leq
	j\leq k\right\}. } \end{equation}

Let $\{\beta _{n}\}$ be a sequence of positive numbers diverging to
infinity, which can be interpreted as a sequence of inverse temperatures,
and for each $n\in \mathbb{N}$, let $P_{n}\in \mathcal{P}(\mathbb{R}^{dn})$
be the probability measure given by 
\begin{equation}
P_{n}\left( d\mathbf{x}_{1},...,d\mathbf{x}_{n}\right) :=\frac{\exp \left(
	-\beta _{n}H_{n}\left( \mathbf{x}_{1},...,\mathbf{x}_{n}\right) \right) }{%
	Z_{n}}\ell (d\mathbf{x}_{1})\cdots \ell (d\mathbf{x}_{n}),  \label{tometro}
\end{equation}%
where $\ell ${\ is a $\sigma $-finite measure on $\mathbb{R}^{d}$ that has
	no atoms and acts as a reference measure}, and $Z_{n}$ is the normalization
constant (which is also referred to as the partition function) given by 
\begin{equation}
Z_{n}:=\int_{\mathbb{R}^{d}}\cdots \int_{\mathbb{R}^{d}}\exp \left( -\beta
_{n}H_{n}\left( \mathbf{x}_{1},...,\mathbf{x}_{n}\right) \right) \ell (d%
\mathbf{x}_{1})\cdots \ell (d\mathbf{x}_{n}).  \label{Zndef}
\end{equation}%
Measures of the form \eqref{tometro} arise in a variety of contexts. For the
case when $\ell $ is Lebesgue measure on $\mathbb{R}^{d}$, it is well known
that if $W$ and $V$ are sufficiently smooth, then $P_{n}$ is the invariant
distribution of a reversible Markov diffusion on $\mathbb{R}^{dn}$ (with
identity diffusion matrix and drift proportional to $\nabla H_{n}$), which
can be viewed as describing the dynamics of $n$ interacting Brownian
particles in $\mathbb{R}^{d}$ \cite[Chapter 5]{Gardiner2010}. On the other
hand, for particular choices of $d$, $V$ and $W$, $P_{n}$ arises as the law
of the spectrum of various random matrix ensembles, including the so-called $%
\beta $-ensemble as well as certain random normal matrices (see Section
1.5.7 of \cite{Chafai2014} for details). 

Given $P_n \in \mathcal{P}(\mathbb{R}^d)$ as in \eqref{tometro}, let $%
Q_{n}=(L_n)_{\#}P_n$ be the measure induced on $\mathcal{P}\left( \mathbb{R}%
^{d}\right)$ by pushing $P_{n}$ forward under the mapping $L_{n}:\mathbb{R}%
^{dn}\rightarrow\mathcal{P}(\mathbb{R}^{d})$ defined in \eqref{Ln} (see
Definition \ref{push} for the definition of $(L_n)_{\#}$). The aim of this
paper is to establish large deviation principles (LDPs) for sequences $%
\{Q_{n}\}$ under general conditions on $V$ and $W$ that allow $V$ and $W$ to
be not only unbounded, but also highly irregular. We apply the weak
convergence methods developed in \cite{Dupuis} to provide results for both
cases where $\beta_{n}=n$ (see Theorem \ref{nspeed}), and $%
\lim_{n\rightarrow\infty} \beta_{n}/n =\infty$ (see Theorem \ref%
{biggernspeed}).   {  For the reader unfamiliar with the 
	weak convergence approach to large deviations,
	we provide a brief outline in Appendix \ref{sec:weakconv}. }

To the best of our knowledge, the most general result in the direction of
Theorem \ref{biggernspeed} is \cite[Theorem 1.1]{Chafai2014}. The latter
seems to be the first paper to present a general approach to proving LDPs
for empirical measures generated by Gibbs distributions, when the inverse
temperatures $\beta _{n}$ diverge faster than $n$, the number of particles
(the particular case of $\beta _{n}=n^{2}$ was considered earlier in \cite%
{Arous1998}). Our theorems recover existing results (see Example \ref{pfffff}%
) and also extend prior results 
in the following multiple directions, in particular resolving several open questions
raised in \cite{Chafai2014}:
\begin{enumerate}
	\item 
	First, whereas the result in \cite[Theorem 1.1]%
	{Chafai2014} considers only speeds $\{\beta _{n}\}$ that satisfy $%
	\lim_{n\rightarrow \infty }\beta _{n}/n{\log n}=\infty $, we allow for any
	speed diverging faster than $n$, { including the speed $n^2$ considered in \cite{Arous1998}}, thus showing that the growth rate condition
	of \cite{Chafai2014} is a technical one related to the combinatorial
	approach used in the proofs therein.  
	\item 
	Second,  { we consider more general confining potentials $V$.  In particular, 
	the assumptions imposed in \cite{Chafai2014} only allow for
	large deviation rate functions whose minimizers  have compact support.
	The minimizer of the rate function in the LDP (when unique) identifies the
	limiting equilibrium measure of the particles.} On the other hand, LDPs with{%
		\ rate functions that have minimizers that are not compactly supported }%
	arise in the context of simulated annealing algorithms that are designed to
	sample from the minimizer of the rate function (in which case the sequence $%
	\{\beta _{n}\}$ is referred to as a cooling schedule); see, e.g., \cite[%
	Section 1.5.8]{Chafai2014}. In this article, we impose weaker assumptions
	that, unlike in \cite{Chafai2014}, allow for{\ confining} potentials{\ $V$}
	that can be weak { (that is, with the  limit of the corresponding
	sequence of empirical measures being non-compactly supported)}, 
	discontinuous (see Examples \ref{Vnoinfinity} and \ref%
	{egWdisc}), unbounded, and possibly not even locally integrable. In
	particular, the potential $V$ is allowed to even be zero in a non-trivial
	unbounded domain, {provided that the volume of the domain outside a ball
		of radius $R$ around the origin decreases \textquotedblleft sufficiently fast" as $R$
		goes to infinity}. This allows us to consider cases where the particles are
	not confined in a bounded set, and in particular leads to examples with
	minimizers that do not have compact support (see Example \ref{Vnoinfinity}),
	thus addressing the open question raised in \cite[Section 1.5.1]{Chafai2014}.
{	 It appears that this has only been previously studied in the case of $\mathbb{R}^{2}$ with the logarithmic Coulomb interaction potential (see Remark \ref{coulos}). }
	\item Third, the freedom of choice for the reference measure $\ell $ allows the
	  study of Gibbs distributions defined on sets that have zero $d$-dimensional
            Lebesgue measure such as, for example, a non-smooth surface on 
	$\mathbb{R}^{3}$ or a fractal set like the Cantor dust in $\mathbb{R}^{2}.$
	A more specific example that often appears in complex potential theory is
	the case where $W$ is the Coulomb potential, and $\ell $ is Lebesgue measure
	on some 1-dimensional subset of { the complex plane} $\mathbb{C}$ such as the unit circle.
	\item  Furthermore, we establish these LDPs not 
	only with respect to the weak topology, but also with respect to a family of
	stronger topologies that include the $p$\emph{-Wasserstein topologies} for
	$p\geq 1$, thus resolving another open question raised in \cite[Section 1.5.6]%
	{Chafai2014}. { The LDPs with respect to stronger topologies are used in
		Lemma 3.4   of \cite{KimRam19} (see also \cite{Kim17})
		to show that, for a large class of Hamiltonians, 
		the sequence $\{\mathbf{x}^n\}_{n \in \mathbb{N}}$ of Gibbs configurations satisfies 
		the so-called ``asymptotic
		thin-shell condition'',  
		which is of relevance in asymptotic geometric analysis
		and high-dimensional probability.
		Specifically, this condition  stipulates that 
		the  sequence of
		scaled Euclidean norms of the random vectors
		satisfies an LDP, and was shown in \cite{KimRam19} (see
		Theorems 2.4, 2.6 and 2.8 therein) to imply that then the  corresponding sequences of
		multi-dimensional random projections of
		the random vectors also satisfy an LDP. 
		This  can be viewed as a non-universal large deviation counterpart of the universality
		result of 
		\cite{AntBalPer03} that 
		random projections of a high-dimensional measure whose
		Euclidean norms satisfy a certain
		concentration property called the ``thin shell condition'' have
		Gaussian fluctuations. 
		As first observed in \cite{GanKimRam16,GanKimRam17}, 
		LDPs are useful for describing non-universal features  lower-dimensional
		projections that encode information about the original high-dimensional measures,
		which is of relevance in numerous applications.
		Since the mapping that takes a vector to its Euclidean norm is continuous
		in the Wasserstein-$2$ topology, 
		the fact that  $\{\mathbf{x}^n\}_{n \in \mathbb{N}}$ satisfies the asymptotic
		thin shell condition can be deduced 
		from the contraction principle and the LDPs 
		(with respect to the $2$-Wasserstein topology)  for   $\{\mathbf{x}^n\}_{n \in \mathbb{N}}$
		established in this article. }
\end{enumerate}
It is also worthwhile to mention that, in
contrast to prior works, in this work the LDPs for all speeds and topologies are
established using a common methodology.

{ 
\subsection{Discussion of related recent results} 
	\label{subs-priorwork}

        This article is a substantial  generalization of
        the   first version of this article \cite{DupLasRam15v1} and  also
  resolves a  minor technical issue therein.
  It contains significant extensions, 
  the most important of which is to allow  weakly confining
  potentials $V$.   In the case
  $\lim_{n\rightarrow\infty} \beta_{n}/n =\infty$ we show that
  although the entropic term disappears
  in the limit, its appearance in the pre-limit can guarantee the validity of the LDP
  in some cases when $V$ does not satisfy $\lim_{\Vert x\Vert \rightarrow \infty}V(x)=\infty$
  (see Section \ref{snd_rep}). This result also highlights the intuitive nature
  of weak convergence methods, and more specifically the use of representations that are connected to the method, like the one in \eqref{snd_rep}.
  In addition, compared to the original version \cite{DupLasRam15v1},
  in the present article the illustrative examples
  have been significantly extended (see Section \ref{sec:2}), 
  to  include  cases where $V$ and $W$ are not continuous,
  and heuristic arguments in \cite{DupLasRam15v1}
  related to the examples 
  have been replaced here with rigorous proofs.
  Finally,  some open problems have  also been
  added, which can generate new directions for research.

  Since the first version \cite{DupLasRam15v1} appeared, several authors have extended our work or used some of the arguments. In \cite{GarciaZelada2018} an LDP for a sequence of point processes defined by Gibbs measures on a compact orientable two-dimensional Riemannian manifold is studied. In \cite{Berman2018,Garcia-Zelada2019}, the connection between LDPs and $\Gamma$-convergence that was first highlighted in \cite{Mariani2012} and subsequently implicitly exploited in our work, was furthered explored. We believe that this is a very natural connection and hypothesize that it can also lead to some new insights in the case where Assumption \ref{C}\ref{C2} is not satisfied.   More recent work
  that appeared after the present version of this article  was posted includes
    \cite{Chafai2019}, where the particular case of the Coulomb potential in dimension 
		$d=2$ is studied. 

}

\section{Assumptions and main results}

\label{subs:2}

{ 
This section is devoted to stating and discussion our main assumptions and results.
Section \ref{subs-not} introduced 
basic definitions and notation used throughout the article,  
and  Sections \ref{subs-mainres1} and \ref{subs-mainres2},
present the main results for the case $\beta_n = n$ and $\beta_n \rightarrow \infty$,
respectively. Corollaries of the main results 
and illustrative examples are presented in Section \ref{sec:2},  and
 an outline of the rest of the article is given in Section \ref{subs-out}.

        First, we start by stating the  assumptions on the potentials $V$ and $W$ that
        will hold throughout.}
        
\renewcommand\theassumption{A}

\begin{assumption}
	\label{A} The functions $W:\mathbb{R}^{d}\times\mathbb{R}^{d}\rightarrow
	(-\infty,\infty]$ and $V :\mathbb{R}^{d}\rightarrow(-\infty,+\infty]$ are
	lsc on their respective domains. There exist $1 > a \geq 0$ and $c\in\mathbb{%
		R}$ such that 
	\begin{equation}
	\begin{split}
	\int_{\mathbb{R}^{d}}\exp\left( -(1-a) V\left( \mathbf{x}\right) \right)
	\ell(d\mathbf{x})<\infty,\hspace{15pt} \inf_{\mathbf{x}\in\mathbb{R}^{d}}
	V\left( \mathbf{x}\right), \inf_{(\mathbf{x},\mathbf{y})\in\mathbb{R}%
		^{d}\times \mathbb{R}^{d}} \left[ W\left( \mathbf{x},\mathbf{y}\right)
	+a\left( V\left( \mathbf{x}\right) +V\left( \mathbf{y}\right) \right) \right]
	>c .  \label{d}
	\end{split}%
	\end{equation}
	In addition, there exists a set $A\in\mathcal{B}(\mathbb{R}^{d})$ with $\ell
	(A) > 0$ 
	such that 
	\begin{equation}  \label{nontrivial}
	\int_{A\times A}(V(\mathbf{x}) +V(\mathbf{y}) + W(\mathbf{x},\mathbf{y}%
	))\ell(d\mathbf{x})\ell(d\mathbf{y})<\infty.
	\end{equation}
\end{assumption}

Assumption \ref{A} guarantees that the Gibbs distribution given in %
\eqref{tometro} is well defined. More precisely, \eqref{d} guarantees that
the measure is well defined and finite, and \eqref{nontrivial} guarantees
that the measure is not trivial.

\begin{remark}
	\label{rem-A} \emph{Under Assumption \ref{A}, without loss of generality, we
		can assume that 
		\begin{equation*}
		\int_{\mathbb{R}^{d}}\exp\left( -(1-a) V\left( \mathbf{x}\right) \right)
		\ell(d\mathbf{x})=1,
		\end{equation*}
		because any constant added to $V$ can be absorbed into $Z_{n}$; see %
		\eqref{tometro}. With some abuse of notation, we use $e^{-(1-a)V}\ell$ to
		denote the probability measure $e^{-(1-a)V(\mathbf{x})}\ell(d\mathbf{x}).$ }
\end{remark}

	\subsection{Notation and definitions}
	
	\label{subs-not}
	In this section, we provide some necessary definitions. Although some of the definitions will be given in their more general form, we would like to clarify that for our results, the underlying space $S$ would be a separable metric space without any assumptions about its completeness. 
We recall the standard definition of the push forward operator $\#$.

\begin{definition}
	\label{push} Given measurable spaces $(S,\mathcal{F})$ and $(\tilde{S} ,%
	\tilde{\mathcal{F}})$, a measurable mapping $f:S\rightarrow\tilde{S}$ and a
	measure $\mu:\mathcal{F}\rightarrow\lbrack0,\infty]$, the pushforward of $%
	\mu $ is the measure induced on $(\tilde{S},\tilde{\mathcal{F}})$ by $\mu$
	under $f$, that is, the measure $f_{\#}\mu:\tilde{\mathcal{F}}\rightarrow
	\lbrack0,\infty]$ is given by 
	\begin{equation*}
	(f_{\#}\mu)(B)=\mu\left( f^{-1}(B)\right) \mbox{ for }B\in\tilde {\mathcal{F}%
	}.
	\end{equation*}
	In other words, $f_{\#}\mu$ is the image measure of $\mu$ under $f$.
\end{definition}

We next recall the definition of a rate function on a separable metric space 
$S$.

\begin{definition}
	\label{ratefn} Given a topological space $S$, a function $\mathcal{H}%
	:S\rightarrow\lbrack0,\infty]$ is said to be a \emph{rate function} if each level set $\{x:\mathcal{H}(x)\leq
	M\},\,M\in\lbrack0,\infty)$, is compact. 
\end{definition}

Note that a function that satisfies the properties in Definition \ref{ratefn}
is sometimes referred to as a good rate function in the literature, { as a way
	to highlight the compactness of the level sets and to distinguish it from lower semi-continuous functions that can be defined by the property of having closed level sets}, but which can in some cases provide large
deviation rates of decay. When not in the context of LDPs, a function that
has the properties stated in Definition \ref{ratefn} is also called a \emph{%
	tightness function}; a term that will be used extensively in the sequel. In
contrast to much of the previous application of weak convergence methods in
large deviations, here we do not assume $S$ is complete. This will be
convenient when dealing with topologies other than the weak topology.

We now recall the definition of an LDP for a sequence of probability
measures on $(S,\mathcal{B}(S))$.

\begin{definition}
	\label{rate} $L$et $\{R_{n}\}\subset\mathcal{P}(S),$ let $\{\alpha_{n}\}$ be
	a sequence of positive real numbers such that $\lim_{n\rightarrow\infty}%
	\alpha_{n}=\infty,$ and let $\mathcal{H}:S\rightarrow\lbrack0,\infty]$ be a
	rate function. The sequence $\{R_{n}\}$ is said to satisfy a large deviation
	principle with speed $\{\alpha_{n}\}$ and rate function $\mathcal{H}$ if for
	each $E\in\mathcal{B}(S),$ 
	\begin{equation*}
	-\inf_{x\in E^{\circ}}\mathcal{H}(x)\leq\liminf_{n\rightarrow\infty}\alpha
	_{n}^{-1}\log(R_{n}(E))\leq\limsup_{n\rightarrow\infty}\alpha_{n}^{-1}%
	\log(R_{n}(E))\leq-\inf_{x\in\bar{E}}\mathcal{H}(x),
	\end{equation*}
	where $E^{\circ}$ and $\bar{E}$ denote the interior and closure of $E$,
	respectively.
\end{definition}

Let $C(\mathbb{R}^d)$ be the space of continuous functions on $\mathbb{R}^d$%
, and let $C_b (\mathbb{R}^d)$ denote the subspace of bounded functions in $%
C(\mathbb{R}^d)$. We endow $\mathcal{P}(\mathbb{R}^{d})$ with the topology
of weak convergence and use $\xrightarrow{w}$ to denote convergence with
respect to this topology; recall that $\mu_{n}\xrightarrow{w}\mu$ if and
only if for all $f \in C_{b}(\mathbb{R}^{d})$, $\int_{\mathbb{R}^{d}}f(%
\mathbf{x})\mu_{n}(d\mathbf{x})\rightarrow\int_{\mathbb{R}^{d}}f(\mathbf{x}%
)\mu(d\mathbf{x})$. 
The L\'{e}vy-Prohorov metric $d_{w}$ metrizes the weak topology on $\mathcal{%
	P}(\mathbb{R}^{d})$, and the space $(\mathcal{P}(\mathbb{R}^{d}),d_{w})$ is
Polish { (see, e.g., Theorem 5 of Appendix III of \cite{BilBook}).}
We also consider stronger topologies, parameterized by functions
belonging to the following set: 
\begin{equation}  \label{cond-psi}
\Psi := \{ \psi \in C(\mathbb{R}^d); \psi \geq 0,
\lim_{c\rightarrow\infty}\inf_{\mathbf{x}:\Vert\mathbf{x}\Vert=c} \psi(%
\mathbf{x})=\infty \}.
\end{equation}
Given $\psi \in \Psi$, let 
\begin{equation}  \label{def-ppsi}
\mathcal{P}_{\psi}( \mathbb{R}^{d}) := \left\{\mu \in\mathcal{P}(\mathbb{R}%
^{d}):\int_{\mathbb{R}^{d}}\psi\left( \mathbf{x}\right) \mu\left( d\mathbf{x}%
\right) <+\infty\right\}.
\end{equation}
We endow $\mathcal{P}_{\psi}\left( \mathbb{R}^{d}\right)$ with the metric 
\begin{equation}  \label{psidef}
d_{\psi}(\mu,\nu):= d_{w}(\mu,\nu)+\left\vert \int_{\mathbb{R}^{d}}\psi(%
\mathbf{x})\mu(d\mathbf{x})-\int_{\mathbb{R}^{d}}\psi(\mathbf{x})\nu(d%
\mathbf{x})\right\vert .
\end{equation}
The space $(\mathcal{P}_{\psi}(\mathbb{R}^{d}), d_{\psi})$ is a separable
metric space (see Lemma \ref{Polish} for a proof).

\begin{remark}({Alternative metrizations of the Wasserstein topology})
	\emph{When $\psi (\mathbf{x})=\|\mathbf{x}\|^{p}$ for $p\in\lbrack1,\infty),$ 
		with  $\mathbf{x} \in \mathbb{R}%
		^d,$ { and $\|\cdot\|$ denoting the Euclidean norm},
		$d_\psi$ induces the $p$-Wasserstein
		topology (see \cite[Remark 7.1.11]{Ambrosio2008}). Another metric that is
		commonly used to induce the $p$-Wasserstein topology on ${\mathcal{P}} (%
		\mathbb{R}^d)$ is $d_p(\mu, \nu) := \inf_{\zeta\in\Pi(\mu,\nu)} \int_{%
			\mathbb{R}^d \times \mathbb{R}^d} ||\mathbf{x}-\mathbf{y}||^p \zeta (d%
		\mathbf{x},d\mathbf{y})$, where $\Pi(\mu,\nu)$ is the set of all measures in 
		$\mathbb{R}^{2d}$ with first marginal $\mu$ and second marginal $\nu.$
		Although $\mathcal{P}_{\psi}(\mathbb{R}^{d})$ endowed with $d_p$ is complete
		and separable, we use the somewhat simpler metric $d_{\psi}$ defined for any 
		$\psi$ satisfying (\ref{cond-psi}), under which $\mathcal{P}_{\psi}(\mathbb{R%
		}^{d})$ is only separable, and not complete. { For more information on the Wasserstein distance and its topological properties, the reader is referred to \cite{Villani2003}; specifically, see Theorem 6.8 therein.} }
\end{remark}

\subsection{Results in the case $\protect\beta_n = n$}
\label{subs-mainres1}

Our first result concerns the LDP for $\{Q_{n}\}$ with speed $\alpha
_{n}=\beta _{n}=n.$ The rate function is expressed in terms of the following
functionals. Given $a\in \lbrack 0,1]$, for $\zeta \in \mathcal{P}(\mathbb{R}%
^{d}\times \mathbb{R}^{d})$, let 
\begin{equation}
\mathfrak{J}_{a}(\zeta ):=\frac{1}{2}\int_{\mathbb{R}^{d}\times \mathbb{R}%
	^{d}}\left( W(\mathbf{x},\mathbf{y})+aV(\mathbf{x})+aV(\mathbf{y})\right)
\zeta (d\mathbf{x}d\mathbf{y}).  \label{scr}
\end{equation}%
Also, for a measure $\nu \in \mathcal{P}(\mathbb{R}^{d})$, as usual we
define the relative entropy functional by 
\begin{equation*}
\mathcal{R}(\mu |\nu ):=%
\begin{cases}
\displaystyle\int_{\mathbb{R}^{d}}\frac{d\mu }{d\nu }(\mathbf{x})\log \left( 
\frac{d\mu }{d\nu }(\mathbf{x})\right) \nu (d\mathbf{x}) & \text{if}\,\mu
\ll \nu , \\ 
\infty & \displaystyle\text{otherwise,}%
\end{cases}%
\end{equation*}%
where $\mu \ll \nu $ denotes that $\mu $ is absolutely continuous with
respect to $\nu $. Then, for $\mu \in \mathcal{P}(\mathbb{R}^{d})$ let 
\begin{equation}
\mathcal{J}_{a}\left( \mu \right) :=\mathfrak{J}_{a}(\mu \otimes \mu )=\frac{%
	1}{2}\int_{\mathbb{R}^{d}\times \mathbb{R}^{d}}\left( W(\mathbf{x},\mathbf{y}%
)+aV(\mathbf{x})+aV(\mathbf{y})\right) \mu \left( d\mathbf{x}\right) \mu
\left( d\mathbf{y}\right) ,  \label{fnal-J}
\end{equation}%
and 
\begin{equation}
\mathcal{I}\left( \mu \right) :=\mathcal{R}\left( \mu |e^{-(1-a)V}\ell
\right) +\mathcal{J}_{a}(\mu ),  \label{def:I}
\end{equation}%
Note that the lack of subscript on $\mathcal{I}$ in \eqref{def:I} is
justified because, as the following easily verifiable relation shows, $%
\mathcal{I}$ does not depend on the constant $a$: 
\begin{equation*}
\mathcal{I}(\mu )=\mathcal{S}(\mu )+\mathcal{V}(\mu )+\mathcal{W}(\mu ),
\end{equation*}%
where $\mathcal{V},\mathcal{S},\mathcal{W}:\mathcal{P}(\mathbb{R}%
^{d})\rightarrow (-\infty ,\infty ],$ are given by 
\begin{equation}
\mathcal{V}\left( \mu \right) :=\int_{\mathbb{R}^{d}}V\left( \mathbf{x}%
\right) \mu \left( d\mathbf{x}\right) ,\hspace{16pt}\mathcal{S}\left( \mu
\right) :=\int_{\mathbb{R}^{d}}\frac{d\mu }{d\ell }(\mathbf{x})\log \left( 
\frac{d\mu }{d\ell }(\mathbf{x})\right) \ell \left( d\mathbf{x}\right) ,
\label{Vdef}
\end{equation}%
and 
\begin{equation}
\mathcal{W}(\mu ):=\frac{1}{2}\int_{\mathbb{R}^{d}\times \mathbb{R}^{d}}W(%
\mathbf{x},\mathbf{y})\mu (d\mathbf{x})\mu (d\mathbf{y}).  \label{Wdef}
\end{equation}

To establish the LDP with respect to stronger topologies $d_{\psi }$, $\psi
\in \Psi $, we will need an additional condition, which we now state. Let $%
\Phi $ be the class of functions defined by 
\begin{equation}
\Phi :=\left\{ \phi :\mathbb{R}_{+}\mapsto \mathbb{R};\phi 
\mbox{ is lsc and
}\lim_{s\rightarrow \infty }\frac{\phi (s)}{s}=\infty \right\} .
\label{def-Phi}
\end{equation}

\renewcommand\theassumption{B}

\begin{assumption}
	\label{B} There exists a lsc function $\gamma:\mathbb{R}^{d}\rightarrow%
	\mathbb{R},$ of the form $\gamma(\mathbf{x})=\phi\left( \psi\left( \mathbf{x}%
	\right) \right),$ for some $\phi \in \Phi$, such that for the constant $a$
	in Assumption \ref{A}, and every $\mu\in\mathcal{P}(\mathbb{R}^{d}),$ we
	have 
	\begin{equation}  \label{m}
	\int_{\mathbb{R}^{d}}\gamma\left( \mathbf{x} \right) \mu\left( d\mathbf{x}%
	\right) \leq\inf_{\zeta\in\Pi(\mu,\mu)}\left\{ \mathfrak{J}_{a}\left(
	\zeta\right) +\mathcal{R}\left( \zeta|e^{-(1-a)V}\ell\otimes
	e^{-(1-a)V}\ell\right) \right\} .
	\end{equation}
\end{assumption}

The following lemma provides a more easily verifiable sufficient condition
under which Assumption \ref{B} holds; its proof is deferred to Appendix \ref%
{sec-apA}. Recall the set $\Psi$ defined in \eqref{cond-psi}.

\begin{lemma}
	\label{phi3} 
	Let $V$ and $W$ satisfy Assumptions \ref{A}, and let $\psi \in \Psi $ 
	satisfy 
	\begin{equation}
	\int_{\mathbb{R}^{d}\times \mathbb{R}^{d}}e^{\lambda \left( \psi \left( 
		\mathbf{x}\right) +\psi \left( \mathbf{y}\right) \right) }e^{-\left( V\left( 
		\mathbf{x}\right) +V\left( \mathbf{y}\right) +W\left( \mathbf{x},\mathbf{y}%
		\right) \right) }d\mathbf{x}d\mathbf{y}<\infty  \label{alalal}
	\end{equation}%
	for all $\lambda \in \mathbb{R}$. Then Assumption \ref{B} is satisfied for
	that $\psi $ and some $\phi \in \Phi $.
\end{lemma}

We now state our first main result, whose proof is given in Section \ref%
{sec:4}.

\begin{theorem}
	\label{nspeed} Let $V$ and $W$ satisfy Assumption \ref{A}, and for $n \in 
	\mathbb{N}$, let $\beta_{n}=n$, let $P_{n}$ be defined as in \eqref{tometro}
	and let $Q_{n}=(L_n)_{\#}P_n.$ Then $\{Q_{n}\}$ satisfies an LDP on $(%
	\mathcal{P} \left( \mathbb{R}^{d}\right), d_w)$ with speed $%
	\alpha_{n}=\beta_{n}=n$ and rate function 
	\begin{equation}  \label{def:Istar}
	\mathcal{I}_{\star}\left( \mu\right) :=\mathcal{I}\left( \mu\right)
	-\inf_{\mu\in\mathcal{P}\left( \mathbb{R}^{d}\right) }\{\mathcal{I}\left(
	\mu\right) \},
	\end{equation}
	where $\mathcal{I}$ is defined by \eqref{def:I}. Furthermore, if there
	exists $\psi \in \Psi$ for which Assumption \ref{B} holds, then $\{Q_{n}\}$
	satisfies an LDP on $(\mathcal{P}_{\psi}(\mathbb{R}^{d}), d_{\psi})$ with
	rate function%
	\begin{equation}  \label{def:Istarpsi}
	\mathcal{I}_{\star}^\psi \left( \mu\right) :=\mathcal{I}\left( \mu\right)
	-\inf_{\mu\in\mathcal{P}_{\psi}\left( \mathbb{R}^{d}\right) }\{\mathcal{I}%
	\left( \mu\right) \}.
	\end{equation}
\end{theorem}

In the case when $W\equiv 0$, Theorem \ref{nspeed} recovers the well known
Sanov's Theorem (see \cite[Theorem 6.2.10]{Dembo1998} or \cite[Theorem 2.2.1]%
{Dupuis} for the LDP with respect to the weak topology and \cite[Theorem 1.1]%
{Wang2010} for the LDP with respect to the $p$-Wasserstein topology). In 
\cite[Theorem 1.1]{Wang2010}{\ the authors prove that for any $p\in \lbrack
	1,\infty )$ and $\psi (x)=||x||^{p}$, if 
	the following condition holds:} 
\begin{equation}
\int_{\mathbb{R}^{d}\times \mathbb{R}^{d}}e^{\lambda \psi \left( \mathbf{x}%
	\right) -V\left( \mathbf{x}\right) }\ell (d\mathbf{x})<+\infty \hspace{8pt}%
\forall \lambda >0,  \label{ellaplo}
\end{equation}%
then $\{Q_{n}\}$ satisfies an LDP in $\mathcal{P}_{\psi }(\mathbb{R}^{d})$
with rate function $\mathcal{R}(\mu |e^{-V}\ell )$, which corresponds to the
case $a=0$ in our setting. The bound (\ref{alalal}) implies Assumption \ref%
{B}, and can be viewed as a generalization of condition (1.3) of \cite%
{Wang2010} ({which was shown in \cite{Wang2010} to be equivalent to %
	\eqref{ellaplo}}) to the case when $W\neq 0$.

Next, if $W$ is continuous and satisfies certain growth conditions on $%
\mathbb{R}^{d}\times\mathbb{R}^{d},$ then the result can be obtained from
the $W=0$ (or Sanov's theorem) case by a simple application of Varadhan's
lemma (see \cite[Theorem 4.3.1]{Dembo1998} or \cite[Theorem 1.2.1]{Dupuis}).
To the best of our knowledge, there are no general results in the literature
that cover the case when $W$ is both unbounded and discontinuous, and
therefore Theorem \ref{nspeed} is the first in that direction. Furthermore,
Assumption \ref{B} 
provides a sufficient condition for the LDP to hold with respect to a rather
large class of stronger topologies, which was useful for the verification of
the asymptotic thin shell condition in \cite{Kim17,KimRam19}.

\subsection{Results in the case $\protect\beta_n/n \rightarrow \infty$}
\label{subs-mainres2}

Motivated by questions arising in random matrix theory, sampling and
simulated annealing, several authors \cite%
{Arous1997a,Arous1998,Chafai2014,Hardy2012,Petz1998,Serfaty2015} have
considered LDPs for $\{Q_{n}\}$ at specific speeds that are faster than $n$,
such as $\beta_n/n\log n \rightarrow \infty$ and $\beta_n = n^2$. Our second
theorem presents a general result for speeds faster than $n$, that is, when $%
\beta_n/n \rightarrow \infty$, under Assumption \ref{A} and certain modified
assumptions on $V$ and $W$ stated in Assumption \ref{C} below. In what
follows, we use $\mathcal{J}$ to denote the functional ${\mathcal{J}}_1:%
\mathcal{P}(\mathbb{R}^{d})\rightarrow (-\infty,\infty]$ defined in %
\eqref{fnal-J}, with $a = 1$. Recall the set $\Phi$ defined in %
\eqref{def-Phi}, for a set $A \subset \mathbb{R}^d$, { $A_{\neq}$ is defined in \eqref{neqdef}},  and as usual let 
$I_A$ denote the indicator function, which assigns $1$ to points in $A$ and $%
0$ otherwise. 

\renewcommand\theassumption{C}

\begin{assumption}
	\label{C}
	
	\begin{enumerate}
		\item \label{C1} There exist a lsc function $\gamma :\mathbb{R}%
		^{d}\rightarrow \mathbb{R}$ of the form $\gamma (\mathbf{x})=\phi (\Vert 
		\mathbf{x}\Vert )$ for some $\phi :\mathbb{R}_{+}\mapsto \mathbb{R}$ with $%
		\lim_{s\rightarrow \infty }\phi (s)=\infty $, a set $A\in {\mathcal{B}}(%
		\mathbb{R}^{d}),$ a sequence $\{r_{n}\}\subset (0,\infty )$ with $%
		r_{n}\rightarrow \infty $ and a constant $C\in \mathbb{R},$ such that 
		\begin{equation}
		\gamma (\mathbf{x})I_{A^{c}}(\mathbf{x})+\gamma (\mathbf{y})I_{A^{c}}(%
		\mathbf{y})\leq V\left( \mathbf{x}\right) +V\left( \mathbf{y}\right)
		+W\left( \mathbf{x},\mathbf{y}\right) +C,  \label{C11}
		\end{equation}%
		for every $n\in \mathbb{N}$ and $\mathbf{x}^{n}\in (\mathbb{R}^{nd})_{\neq }$
		\begin{equation}
		\int_{A_{n}^{1}}\gamma (\mathbf{x})L_{n}\left( \mathbf{x}^{n};d\mathbf{x}%
		\right) \leq \int_{(A_{n}^{1}\times A_{n}^{1})_{\neq }}\left( V\left( 
		\mathbf{x}\right) +V\left( \mathbf{y}\right) +W\left( \mathbf{x},\mathbf{y}%
		\right) \right) L_{n}\left( \mathbf{x}^{n};d\mathbf{x}\right) L_{n}\left( 
		\mathbf{x}^{n};d\mathbf{y}\right) +C,  \label{C12}
		\end{equation}%
		and 
		\begin{equation}
		\sup_{n\in \mathbb{N}}\frac{n}{\beta _{n}}\log \left( \int_{A_{n}^{2}}e^{%
			\frac{\beta _{n}}{n}\gamma (\mathbf{x})-(1-a)V\left( \mathbf{x}\right) }\ell
		\left( d\mathbf{x}\right) \right) <C,  \label{C13}
		\end{equation}%
		where $A_{n}^{1}:=A\cap B\left( 0,r_{n}\right) $ and $A_{n}^{2}:=A\setminus
		B\left( 0,r_{n}\right) .$
		
		\item \label{C2} For each $\mu\in\mathcal{P}\left( \mathbb{R}^{d}\right) $
		such that $\mathcal{J}\left( \mu\right) <+\infty,$ there is a sequence $%
		\{\mu_{n}\}\subset {\mathcal{P}}(\mathbb{R}^d)$, with each $\mu_n \ll \ell$ 
		such that $\mu_{n} \overset{w}{\rightarrow} \mu$, and $\mathcal{J}\left(
		\mu_{n}\right) \rightarrow\mathcal{J}\left( \mu\right) $ as $n \rightarrow
		\infty$.
		
		\item \label{C3} The function $\gamma$ in part 1. above is of the form $%
		\gamma(\mathbf{x})=\phi(\psi(\mathbf{x})),$ for some $\phi \in \Phi$. 
	\end{enumerate}
\end{assumption}

When $A=\emptyset,$ then Assumption \ref{C}\ref{C1} collapses to the
following more easily verifiable condition:

\renewcommand\theassumption{C'1}

\begin{assumption}
	\label{C'1} There exists a lsc function $\gamma:\mathbb{R}^{d}\rightarrow%
	\mathbb{R}$ of the form $\gamma(\mathbf{x})=\phi(\|\mathbf{x}\|),$ where $%
	\phi:\mathbb{R}_{+}\rightarrow\mathbb{R}$ satisfies $\lim_{s
		\rightarrow+\infty}\phi\left(s\right) =+\infty,$ such that 
	\begin{equation}
	V\left( \mathbf{x}\right) +V\left( \mathbf{y}\right) +W\left( \mathbf{x},%
	\mathbf{y}\right)\geq \gamma(\mathbf{x})+ \gamma(\mathbf{y}).
	\end{equation}
\end{assumption}

Assumption \ref{C'1} covers all of the well known examples in the literature
and the majority of generalizations that we provide in this paper. The main
reason we introduce the more complicated Assumption \ref{C}\ref{C1}, { is that
	we want to include cases involving higher dimensional spaces ($d\geq 3$),
	where the confining potential $V$ does not necessarily satisfy
	$\lim_{n \rightarrow \infty} V(x_n) = \infty$ when
	$\left\Vert x_{n}\right\Vert \rightarrow \infty $, an assumption that 
	is used in both \cite{Chafai2014} and \cite{Serfaty2015},
	which  cover cases
	where the rate function has minimizers with compact support. 
	For  special cases in two dimensions where minimizers do not have compact
	support, see \cite{Hardy2012} and Remark \ref{coulos}}.

Assumption \ref{C}\ref{C2} is used in Section \ref{Upn2} to establish the
Laplace principle upper bound. For examples of pairs $(V,W)$ that satisfy
Assumption \ref{C}\ref{C2}, the reader is directed to \cite[Proposition 2.8]%
{Chafai2014}. Finally, Assumption \ref{C}\ref{C3} is used to obtain the
result for the stronger topologies. 

\begin{remark}({On the entropic term in the case  $\beta _{n}/n\rightarrow \infty .$}  )
	\label{rem-relent} \emph{\ It is worthwhile to emphasize that, in the case
		where $\beta _{n}/n\rightarrow \infty ,$ although the entropic term
		disappears from the rate function and is ignored in almost all other proofs
		we found in the literature, in our proof approach (see in particular, Lemma %
		\ref{controlingcontrols4}) the relative entropy functional still plays an
		important role in the intermediate steps. In particular, it helps us
		establish tightness and prove the LDP under weaker conditions on $V$ than
		those assumed in the literature. }
\end{remark}

We now state our second main result, whose proof is deferred to Section \ref%
{sec:5}.

\begin{theorem}
	\label{biggernspeed} Consider a sequence $\{\beta_{n}\}$ such that $%
	\lim_{n\rightarrow\infty}\beta _{n}/n=\infty,$ and let $V$ and $W$ satisfy
	Assumptions \ref{A}, \ref{C}\ref{C1} and \ref{C}\ref{C2}. For $n\in\mathbb{N}%
	,$ let $P_{n}$ be as in \eqref{tometro} and $Q_{n}=(L_n)_{\#}P_n.$ Then $%
	\{Q_{n}\}$ satisfies an LDP on $\mathcal{P}(\mathbb{R}^{d})$ with speed $%
	\alpha_{n}=\beta_{n}$ and rate function%
	\begin{equation}  \label{def:Jstar}
	\mathcal{J}_{\star}\left( \mu\right) =\mathcal{J}\left( \mu\right)
	-\inf_{\mu\in\mathcal{P}(\mathbb{R}^{d})}\{\mathcal{J}\left( \mu\right) \},
	\end{equation}
	where $\mathcal{J}$ is the functional $\mathcal{J}_1$ given in \eqref{fnal-J}%
	. Furthermore, if there exists $\psi \in \Psi$ for which Assumption $\ref{C}%
	\ref{C3}$ holds, then $\{Q_{n}\}$ satisfies an LDP on $(\mathcal{P}_{\psi}(%
	\mathbb{R}^{d}), d_{\psi})$ 
	with the rate function 
	\begin{equation}  \label{def:Jstarpsi}
	\mathcal{J}_{\star}^\psi \left( \mu\right) := \mathcal{J}\left( \mu\right)
	-\inf_{\mu\in\mathcal{P}_{\psi}\left( \mathbb{R}^{d}\right) }\{\mathcal{J}%
	\left( \mu\right) \}.
	\end{equation}
\end{theorem}

A direct consequence of Theorems \ref{nspeed} and \ref{biggernspeed} is the
following.

\begin{remark}({ LDPs for empirical moments})  
	\emph{Suppose $V$ and $W$ satisfy Assumption \ref{A}, and Assumption \ref{B}
		holds with $\psi (\mathbf{x}) := ||\mathbf{x}||^p$ for some $p \geq 1$. Let $%
		(X_1^n, \ldots, X_n^n)$ be distributed according to $P^n$ and for any $q
		\leq p$, let $Y^{n}_q := \frac{1}{n}\sum_{i=1}^{n}|X^{n}_{i}|^{q},$ $n \in 
		\mathbb{N}$. Then Theorem \ref{nspeed}, the continuity of the map $\mu
		\mapsto \int ||\mathbf{x}||^q \mu(d\mathbf{x})$ in the Wasserstein-$p$
		topology and the contraction principle \cite[Theorem 4.2.1]{Dembo1998}
		together show that $\{Y_q^n\}$ satisfies an LDP with speed $\beta_{n} = n$
		and rate function 
		\begin{equation*}
		H({y}) := \inf_{\mu \in {\mathcal{P}}(\mathbb{R}^d)}\left\{\mathcal{I}%
		_{\star}^\psi(\mu):{y}=\int_{\mathbb{R}^{d}}||\mathbf{x}||^{q}\mu(dx)\right%
		\}.
		\end{equation*}
		Likewise, if $V$ and $W$ satisfy Assumptions \ref{A}, \ref{C}\ref{C1} and %
		\ref{C}\ref{C2}, and Assumption \ref{C}\ref{C3} holds with $\psi (\mathbf{x}%
		) := ||\mathbf{x}||^p$ and $\beta_n/n \rightarrow \infty$ as $n \rightarrow
		\infty$, then Theorem \ref{biggernspeed} shows that $\{Y_q^n\}$ satisfies an
		LDP with speed $\beta_{n}$ and rate function $\tilde{H}({y})=\inf_{\mu \in {%
				\mathcal{P}} (\mathbb{R}^d)}\{\mathcal{J}_{\star}^\psi(\mu):{y}=\int_{%
			\mathbb{R}^{d}}||\mathbf{x}||^{q}\mu(dx)\}.$ }
\end{remark}

\subsection{Corollaries of the main results}
\label{sec:2}

In this section, we provide several illustrative examples for which the
assumptions of Theorems \ref{nspeed} and \ref{biggernspeed} can be verified.

\subsubsection{Known examples covered by our assumptions}

We start by considering potentials that have already been investigated in
the literature for the weak topology, and showing that they satisfy our
assumptions. In what follows, let $K_{\Delta}:\mathbb{R}^{d}\mapsto\mathbb{R}
$ be the Coulomb potential given by $K_{\Delta}(x)=-|x|$ when $d=1$, $%
K_{\Delta }(x)=-\log(||x||)$ when $d=2$ and $K_{\Delta}(x)=1/||x||^{d-2}$
when $d>2$.

\begin{example}
	\label{pfffff} Let $\ell$ be Lebesgue measure. The pair $(V,W)$ given by $V(%
	\mathbf{x})=||\mathbf{x}||^{p}$ for some $p>1$ and $W(\mathbf{x},\mathbf{y}%
	)=K_{\Delta}(\mathbf{x}-\mathbf{y})$ satisfies Assumptions \ref{A}, \ref{C'1}
	and \ref{C}\ref{C3}, and Assumptions \ref{B} and \ref{C}\ref{C2} also hold
	with $\psi(\mathbf{x})=||\mathbf{x}||^{q}$, $q<p$.
\end{example}

\begin{proof}[Proof of Example \ref{pfffff}.] 
	For $d\geq 3,$ it is trivial to verify that Assumption \ref{A} is satisfied with $a = 0$. 
	For the case $d=2,$ we pick $a=\frac{1}{2},$ and 
	observe that $\frac{1}{4}\Vert\mathbf{x}\Vert^{p}+\frac{1}{4}%
	\Vert\mathbf{y}\Vert^{p}-\log\Vert\mathbf{x}-\mathbf{y}\Vert$ is bounded from
	below by a constant $c,$ since $z \mapsto -\log z$ is convex and $\lim_{s\rightarrow
		\infty}(s^{p}-\log s)=\infty.$ For the case $d=1,$ we observe that $K_{\Delta}(x)$ is continuous and also $\frac{1}{4}|x|^{p}+\frac{1}{4}%
	|y|^{p}-|x-y|\geq\frac{1}{4}|x|^{p}+\frac{1}{4}%
	|y|^{p}-|x|-|y|$ is bounded from
	below by a constant $c,$ since $\lim_{s\rightarrow\infty}(\frac{1}{4}s^{p}- s)=\infty.$  We verify Assumption \ref{C'1} by picking
	$\phi(s)=\frac{1}{4}s^{p}+\bar{C},$ where $\bar{C}$ is a suitable constant. Finally, it is
	also easy to see that the pair $(V,W)$ satisfies Assumptions \ref{B}
	and \ref{C}\ref{C3} with $\psi(\mathbf{x})=||\mathbf{x}||^{q},q<p,$ by
	applying Lemma \ref{phi3} for the first case and by picking $\phi(s)=\frac
	{1}{4}s^{{p}/{q}}+\bar{C},$ where $\bar{C}$ is a suitable constant, for the second. 
	Verification of Assumption \ref{C}\ref{C2} is a direct application of point (3) in \cite[Proposition 2.8]{Chafai2014}.
\end{proof}

Example \ref{pfffff} shows, in particular, that our assumptions are
satisfied in the cases covered in \cite{Chafai2014}, including the popular
case studied in \cite{Chafai2014,Arous1998,Hardy2012,Petz1998}, of $V(%
\mathbf{x})=\Vert\mathbf{x}\Vert^{2}$, $W(\mathbf{x},\mathbf{y})=-\log(%
\mathbf{x}-\mathbf{y})$ and $\ell$ being Lebesgue measure.

\subsubsection{Non-diverging, weakly confining potentials}\label{weakweak}

Given Assumption \ref{A}, Assumption \ref{C'1} is, at least seemingly, weaker than the condition $\lim_{\Vert 
	\mathbf{x}\Vert \rightarrow \infty }V\left( \mathbf{x}\right) =+\infty $
imposed in \cite{Chafai2014}. This can be directly seen if one takes $\phi
(t)=\frac{(1-a)}{2}\inf_{\Vert \mathbf{x}\Vert =t}V(\mathbf{x})+C^{\prime },$
where $C^{\prime }$ is chosen accordingly. However, the two assumptions have
similar origins. More specifically, since $W,V$ generate $H_{n}$ by a linear
combination, one can transfer the confining attributes of $V$ to $W$ by
taking for example any $V\geq 0$ such that $e^{-V}\ell$ is a probability measure and $W(\mathbf{x},\mathbf{y})=\Vert \mathbf{x}%
\Vert ^{2}+\Vert \mathbf{y}\Vert ^{2}$. The important point is
that for these cases, $W(\mathbf{x},\mathbf{y})+V(\mathbf{x})+V(%
\mathbf{y})$ is penalizing large
values of $\Vert \mathbf{x}\Vert $ and $\Vert \mathbf{y}\Vert .$

In \cite[p. 38]{Serfaty2015} there is an intuitive explanation for how this
containment can be applied for the case where $V$ is superlinear and $W$ is the
Coulomb potential in order to prove the LDP. The proof is based on the idea that in a closed subset of 
$\mathcal{P}(\mathbb{R}^{d}),$ the minimizers of  the rate
function $\mathcal{J}$ can be approximated by minimizers of $H_{n}$ when $n$
is sufficiently large. The approximation property is a consequence of the $%
\Gamma -$convergence of $H_{n}$ to $\mathcal{J}$ (see \cite{Braides2002a})
together with the \textquotedblleft coercivity" of $H_{n}$ (i.e., $H_{n}$
has minimizers in every closed set). To establish these
properties of $H_{n}$, the confining properties of $V$ play a crucial role.

We now provide the following example.

\begin{example}
	\label{Vnoinfinity} Let $q>1,d=3,$ $W\left( \mathbf{x},\mathbf{y}\right)
	=2/\Vert \mathbf{x}-\mathbf{y}\Vert ,$ $\ell $ be Lebesgue measure and let $%
	V $ be given by 
	\begin{equation*}
	V\left( \mathbf{x}\right) =%
	\begin{cases}
	0 & \mathbf{x}\in B_{\mathbf{z}}:=B(\mathbf{z},a_{\mathbf{z}})=\{\mathbf{x}%
	\in \mathbb{R}^{3}:||\mathbf{x}-\mathbf{z}||\leq a_{\mathbf{z}}\}%
	\mbox{ for
		some }\mathbf{z}\in \mathbb{Z}^{3}, \\ 
	\Vert \mathbf{x}\Vert ^{q} & \text{otherwise},%
	\end{cases}%
	\end{equation*}%
	where the constants $\{a_{\mathbf{z}}\}$ satisfy $a_{\mathbf{z}}\leq
	e^{-3\Vert \mathbf{z}\Vert ^{2q+3}},$ $\mathbf{z}\in \mathbb{Z}^{3}$. Then
	for $\beta _{n}=n^{2},$ Assumption \ref{C}\ref{C1} and Assumption \ref{C}\ref%
	{C3} hold with $A=\cup _{\mathbf{z}\in \mathbb{Z}^{3}}B(\mathbf{z},a_{%
		\mathbf{z}})$, $r_{n}=\sqrt[q+3]{n},$ and $\phi (x)=\frac{1}{2}x^{q}.$
	Finally, no minimizer of the rate function $\mathcal{J}_{\ast }$ has compact
	support.
\end{example}

\begin{remark}
	\emph{For this example there are non-trivial regions of }$\mathbb{R}^{2d}$\emph{%
		\ extending infinity on which }$W(x,y)+V(x)+V(y)$\emph{\ is close to zero,
		so there is not an obvious confining property. Also, since }$J$\emph{\ is a
		rate function it exhibits minimizers on every closed ball }$\bar{\mathcal{B}}
	$\emph{\ of }$P(R^{d})$\emph{. In contrast, }$H_{n}$\emph{\ not only does
		not have minimizers, but also }%
	\begin{equation}
	\sup_{n\in \mathbf{N}}\inf_{\{\mathbf{x}^{n}\in (\mathbb{R}^{nd})_{\neq
		}:L_{n}(\mathbf{x}^{n};\cdot )\in \,\bar{\mathcal{B}}\}}H_{n}<\inf_{\mu \in 
		\bar{\mathcal{B}}}\mathcal{J}.  \label{notenough}
	\end{equation}%
	\emph{However, for this example the LDP holds for example in the case }$%
	\beta _{n}=n^{2}.$
\end{remark}

\begin{remark}
	\emph{We can modify Example \ref{Vnoinfinity} slightly to take $V$ to be
		continuous, for example, if we set $V$ equal to zero only in $B(\mathbf{z}%
		,a_{\mathbf{z}}/2)$, equal to $\Vert x\Vert ^{q}$ outside $\cup _{\mathbf{z}%
			\in \mathbb{Z}^{3}}B(\mathbf{z},a_{\mathbf{z}}),$ and use Tietze's extension
		theorem to extend it to a continuous function. For such a continuous $V$, it
		is trivial to establish Assumption \ref{C}\ref{C2} by applying the results
		in \cite[Proposition 2.8]{Chafai2014}. It will be clear from the proof of
		Example \ref{Vnoinfinity} below that Assumptions \ref{C}\ref{C1} and \ref{C}%
		\ref{C3} also continue to hold with this modification since $V$ remains
		positive. Assumption \ref{C}\ref{C2} can be also proved directly for our
		initial example (with discontinuous $V$) in a straightforward manner. We
		omit the proof. }
\end{remark}

\begin{proof}[Proof of Example \ref{Vnoinfinity}]
	We start by verifying Assumption
	\ref{C}\ref{C1}, namely inequalities \eqref{C11}--\eqref{C13}. Let
	$\bar{\gamma}(x):=\frac{1}{2}||x||^{q}$. It follows immediately from the
	definitions that \eqref{C11} holds with $C=0$ and $\gamma=C^{\prime}%
	\bar{\gamma}$ for any $C^{\prime}\leq1$. Next, let $n\in\mathbb{N},$ and let
	$A_{1}^{n}:=A\cap B(0,r_{n})$, $A_{2}^{n}=A\setminus B(0,r_{n})$ be as in
	Assumption \ref{C}\ref{C1}. We split the set $\mathbb{Z}^{3}\cap A_{1}^{n},$
	into two set $Z_{1}^{n}$ and $Z_{2}^{n}:$ Let $Z_{1}^{n}$ contain the cubic
	integers $\mathbb{Z}^{3}\cap A_{1}^{n}$ for which $L_{n}(B_{\mathbf{z}}\cap
	A_{1}^{n})\leq\frac{1}{n}$ and let $Z_{2}^{n}$ contain the cubic integers for
	which $L_{n}(B_{\mathbf{z}}\cap A_{1}^{n})\geq\frac{2}{n}.$  The idea is to choose the cardinality $|Z_{1}^{n}|$ sufficiently
	small such that particles in the corresponding sets do not contribute much to
	the interaction energy, and at the same time, choose the balls around
	$Z_{2}^{n}$ to be so small that the pairs of particles that are in one of
	these sets create a significant interaction energy. We now verify \eqref{C12}.
	Fix $\mathbf{x}^{n}\in(\mathbb{R}^{dn})_{\neq}$ and denote $L_{n}%
	:=L_{n}(\mathbf{x}^{n},\cdot)$. For $\mathbf{z}\in Z_{2}^{n},$ since
	$W(\mathbf{x},\mathbf{y})=2/||\mathbf{x}-\mathbf{y}||\geq1/a_{\mathbf{z}}$
	when $\mathbf{x},\mathbf{y}\in B_{\mathbf{z}}$,  we have
	\[
	\int_{\left(  B_{\mathbf{z}}\cap A_{1}^{n}\right)  \times\left(
		B_{\mathbf{z}}\cap A_{1}^{n}\right)  _{\neq}}W\left(  \mathbf{x}%
	,\mathbf{y}\right)  L_{n}\left(  d\mathbf{x}\right)  L_{n}\left(
	d\mathbf{y}\right)  \geq\frac{2}{a_{\mathbf{z}}}\frac{1}{2}L_{n}%
	^{2}(B_{\mathbf{z}}\cap A_{1}^{n})=\frac{L_{n}^{2}(B_{\mathbf{z}}\cap
		A_{1}^{n})}{a_{\mathbf{z}}}.
	\]
	Combining this with the nonnegativity of $W$, the Cauchy-Schwarz inequality (in
	the third inequality below), the summability of $||\mathbf{z}||^{2q}%
	a_{\mathbf{z}},z\in\mathbb{Z}^{3}$ (in the fourth inequality), and using
	$C^{\prime\prime}>0$ to denote a constant that may change from line to line,
	we obtain
	\begin{align*}
	\int_{(A_{1}^{n}\times A_{1}^{n})_{\neq}}W\left(  \mathbf{x},\mathbf{y}%
	\right)  L_{n}\left(  d\mathbf{x}\right)  L_{n}\left(  d\mathbf{y}\right)    &
	\geq\sum_{\mathbf{z}\in Z_{2}^{n}}\int_{\left(  B_{\mathbf{z}}\cap A_{1}%
		^{n}\right)  \times\left(  B_{\mathbf{z}}\cap A_{1}^{n}\right)  _{\neq}%
	}W\left(  \mathbf{x},\mathbf{y}\right)  L_{n}\left(  d\mathbf{x}\right)
	L_{n}\left(  d\mathbf{y}\right)  \\
	& \geq\sum_{\mathbf{z}\in Z_{2}^{n}}\frac{L_{n}^{2}(B_{\mathbf{z}}\cap
		A_{1}^{n})}{a_{\mathbf{z}}}\\
	& \geq\frac{\left(  \sum_{\mathbf{z}\in Z_{2}^{n}}L_{n}(B_{\mathbf{z}}\cap
		A_{1}^{n})\Vert\mathbf{z}\Vert^{q}\right)  ^{2}}{\sum_{\mathbf{z}\in Z_{2}%
			^{n}}\Vert\mathbf{z}\Vert^{2q}a_{\mathbf{z}}}\\
	& \geq C^{\prime\prime}\left(  \sum_{\mathbf{z}\in Z_{2}^{n}}L_{n}%
	(B_{\mathbf{z}}\cap A_{1}^{n})\Vert\mathbf{z}\Vert^{q}\right)  ^{2}\\
	& \geq C^{\prime\prime}\left(  \int_{  \bigcup\limits_{\mathbf{z}\in
			Z_{2}^{n}}B_{\mathbf{z}}\cap A_{1}^{n}  }\Vert\mathbf{x}\Vert^{q}%
	L_{n}(d\mathbf{x})\right)  ^{2}.
	\end{align*}
	Expanding the last term and using the relations $s^{2}\geq s-1$, $|Z_{1}%
	^{n}|\leq r_{n}^{3}$, and $r_{n}=\sqrt[q+3]{n},$ we see that the left-hand
	side above is bounded below by
	\begin{align*}
	C^{\prime\prime}\left(  \int_{  \bigcup\limits_{\mathbf{z}\in
			Z_{2}^{n}}B_{\mathbf{z}}\cap A_{1}^{n}  }\Vert\mathbf{x}\Vert^{q}%
	L_{n}(d\mathbf{x})+\int_{\bigcup\limits_{\mathbf{z}\in Z_{1}^{n}%
		}B_{\mathbf{z}}\cap A_{1}^{n}  }\Vert\mathbf{x}\Vert^{q}L_{n}%
	(d\mathbf{x})-|Z_{1}^{n}|\frac{r_{n}^{q}}{n}\right)  ^{2} 
	& \geq 
	C^{\prime\prime}\left(  \int_{A_{1}^{n}}\Vert\mathbf{x}\Vert
	^{q}L_{n}(d\mathbf{x})-\frac{r_{n}^{q+3}}{n}-1\right) \\
	& =C^{\prime\prime}\int_{A_{1}^{n}}\bar{\gamma}(\mathbf{x}%
	)L_{n}(d\mathbf{x})-2C^{\prime\prime}.
	\end{align*}
	This implies that \eqref{C12} holds with $\gamma=C^{\prime}\bar{\gamma}%
	=\min(C^{\prime\prime},1)\bar{\gamma}$. Lastly, to prove \eqref{C13},
	recalling that $\beta_{n}=n^{2},r_{n}=\sqrt[q+3]{n}$ and $V=0$ on the set $A,$
	we have for any $a\in\lbrack0,1)$,
	\begin{align*}
	\frac{n}{\beta_{n}}\int_{A_{n}^{2}}e^{\frac{\beta_{n}}{n}\gamma\left(
		\mathbf{x}\right)  -(1-a)V\left(  \mathbf{x}\right)  }\ell(d\mathbf{x})  &
	\leq\frac{1}{n}\sum_{\mathbf{z}\in\left(  \mathbb{Z}^{3}\setminus
		B(0,r_{n})\right)  }\int_{A_{n}^{2}\cap B_{\mathbf{z}}}e^{n\Vert
		\mathbf{x}\Vert^{q}}d\mathbf{x}\\
	& \leq\frac{1}{n}\sum_{\mathbf{z}\in\left(  \mathbb{Z}^{3}\setminus
		B(0,r_{n})\right)  }e^{n2\Vert\mathbf{z}\Vert^{q}}a_{\mathbf{z}}\\
	& \leq\frac{1}{n}\sum_{\mathbf{z}\in\left(  \mathbb{Z}^{3}\cap B^{c}%
		(0,r_{n})\right)  }e^{n2\Vert\mathbf{z}\Vert^{q}}e^{-3\Vert\mathbf{z}%
		\Vert^{2q+3}}\\
	& \leq\frac{1}{n}\sum_{\mathbf{z}\in\left(  \mathbb{Z}^{3}\cap B^{c}%
		(0,r_{n})\right)  }e^{n2\Vert\mathbf{z}\Vert^{q}}e^{-3n\Vert\mathbf{z}%
		\Vert^{q}},
	\end{align*}
	which converges to zero as $n\rightarrow\infty$, and hence, \eqref{C13}
	follows. This concludes the proof of Assumption \ref{C}\ref{C1}, and also
	verifies Assumption \ref{C}\ref{C3} for any $\psi\in\Psi$ such that
	$\max_{x:||x||=m}\psi(x)/m^{q}\rightarrow0$.
	
	We now argue by contradiction to prove that $\mathcal{J}$ (equivalently,
	$\mathcal{J}_{\ast}$) does not have any minimizers with compact support. Let
	$\mu_{\min}$ be a minimizer with compact support $K$ that is contained in
	$B(0,R)$ for some $R\in\mathbb{N}$. Let $\mathbf{\tilde{z}}=(6R,0,0)$. We pick
	$\mathbf{\tilde{x}}\in K$ such that $\mu_{\min}(B(\mathbf{\tilde{x}%
	},a_{\mathbf{\tilde{z}}}))>0,$ and by choosing $R$ sufficiently large so that
	$a_{\mathbf{\tilde{z}}}$ is sufficiently small, we can assume without loss of
	generality that
	\begin{equation}
	0<\mu_{\min}(B(\mathbf{\tilde{x}},a_{\mathbf{\tilde{z}}}))<1.\label{strict}%
	\end{equation}
	We define a new measure $\mu_{\min}^{\ast}:=(\mu_{\min})_{|_{K\setminus
			B(\mathbf{\tilde{x}},a_{\mathbf{\tilde{z}}})}}+\mu_{\mathbf{\tilde{z}}},$
	where $\mu_{\mathbf{\tilde{z}}}$ is the measure that minimizes the energy
	$\int_{B_{\mathbf{\tilde{z}}}\times B_{\mathbf{\tilde{z}}}}W(\mathbf{x}%
	,\mathbf{y})\mu(d\mathbf{x})\mu(d\mathbf{y})$ amongst all measures $\mu$ on
	$\mathbb{R}^{3}$ with support in $B_{\mathbf{\tilde{z}}}$ that have
	$\mu(B_{\mathbf{\tilde{z}}})=\mu_{\min}(B(\mathbf{\tilde{x}},a_{\mathbf{\tilde
			{z}}})).$

	We now show that ${\mathcal{J}}(\mu_{\min}^{\ast})<{\mathcal{J}}(\mu_{\min})$.
	First, note that since $V$ is zero on the support of $\mu_{\mathbf{\tilde{z}}%
	}$,
	\begin{align*}
	\mathcal{V}(\mu_{\min})=\int_{\mathbb{R}^{d}}V(\mathbf{x})\mu_{\min
	}(d\mathbf{x})  & =\int_{\mathbb{R}^{d}\setminus B(\mathbf{\tilde{x}%
		},a_{\mathbf{\tilde{z}}})}V(\mathbf{x})\mu_{\min}(d\mathbf{x})+\int
	_{B(\mathbf{\tilde{x}},a_{\mathbf{\tilde{z}}})}V(\mathbf{x})\mu_{\min
	}(d\mathbf{x})\\
	& \geq\int_{\mathbb{R}^{d}}V(\mathbf{x})(\mu_{\min})_{|_{K\setminus
			B(\mathbf{\tilde{x}},a_{\mathbf{\tilde{z}}})}}(d\mathbf{x})+\int
	_{\mathbb{R}^{d}}V(\mathbf{x})\mu_{\mathbf{\tilde{z}}}(d\mathbf{x}%
	)=\mathcal{V}(\mu_{\min}^{\ast}).
	\end{align*}
	Next, by the symmetry and nonnegativity of $W$, we have
	\[
	\mathcal{W}(\mu_{\min})=\int_{\mathbb{R}^{d}\times\mathbb{R}^{d}}%
	W(\mathbf{x},\mathbf{y})\mu_{\min}(d\mathbf{x})\mu_{\min}(d\mathbf{y}%
	)=\widetilde{W}_{1}+\widetilde{W}_{2}+\widetilde{W}_{3},
	\]
	where
	\begin{align*}
	\widetilde{W_{1}}:=  & \int_{(\mathbb{R}^{d}\setminus B(\mathbf{\tilde{x}%
		},a_{\mathbf{\tilde{z}}}))\times(\mathbb{R}^{d}\setminus B(\mathbf{\tilde{x}%
		},a_{\mathbf{\tilde{z}}}))}W(\mathbf{x},\mathbf{y})\mu_{\min}(d\mathbf{x}%
	)\mu_{\min}(d\mathbf{y}),\\
	\widetilde{W}_{2}:=  & \int_{B(\mathbf{\tilde{x}},a_{\mathbf{\tilde{z}}%
		})\times B(\mathbf{\tilde{x}},a_{\mathbf{\tilde{z}}})}W(\mathbf{x}%
	,\mathbf{y})\mu_{\min}(d\mathbf{x})\mu_{\min}(d\mathbf{y}),\\
	\widetilde{W}_{3}:=  & 2\int_{(\mathbb{R}^{d}\setminus B(\mathbf{\tilde{x}%
		},a_{\mathbf{\tilde{z}}}))\times B(\mathbf{\tilde{x}},a_{\mathbf{\tilde{z}}}%
		)}W(\mathbf{x},\mathbf{y})\mu_{\min}(d\mathbf{x})\mu_{\min}(d\mathbf{y}).
	\end{align*}
	Now, since $\mu_{\min}$ has support in $K$, we have
	\[
	\widetilde{W}_{1}=\int_{\mathbb{R}^{d}\times\mathbb{R}^{d}}W(\mathbf{x}%
	,\mathbf{y})(\mu_{\min})_{|_{K\setminus B(\mathbf{\tilde{x}},a_{\mathbf{\tilde
					{z}}})}}(d\mathbf{x})(\mu_{\min})_{|_{K\setminus B(\mathbf{\tilde{x}%
			},a_{\mathbf{\tilde{z}}})}}(d\mathbf{y}).
	\]
	The definition of $\mu_{\mathbf{\tilde{z}}}$ and the fact that $W(\mathbf{x}%
	,\mathbf{y})=W(\mathbf{x}+$\textbf{$\tilde{z}$}$-\mathbf{z},\mathbf{y}%
	+$\textbf{$\tilde{z}$}$-\mathbf{z})\geq0$ for every $\mathbf{x},\mathbf{y}$,
	imply
	\[
	\widetilde{W}_{2}\geq\int_{B_{\mathbf{\tilde{z}}}\times B_{\mathbf{\tilde{z}}%
	}}W(\mathbf{x},\mathbf{y})\mu_{\mathbf{\tilde{z}}}(d\mathbf{x})\mu
	_{\mathbf{\tilde{z}}}(d\mathbf{y})=\int_{\mathbb{R}^{d}\times\mathbb{R}^{d}%
	}W(\mathbf{x},\mathbf{y})\mu_{\mathbf{\tilde{z}}}(d\mathbf{x})\mu
	_{\mathbf{\tilde{z}}}(d\mathbf{y}).
	\]
	Also, since $W(\mathbf{x},\mathbf{y})=2/||\mathbf{x}-\mathbf{y}||$, $\mu_{\min
	}$ has support in $K\subset B(0,R)$ and $B_{\mathbf{\tilde{z}}}\cap
	B(0,5R)=\emptyset$, we have $||\mathbf{x}-\mathbf{y}||\leq2R$ for
	$\mathbf{x},\mathbf{y}\in K$, and $||\mathbf{x}-\mathbf{y}||\geq4R$ for
	$\mathbf{x}\in K,\mathbf{y}\in B_{\mathbf{\tilde{z}}}$. Therefore, we have
	\begin{align*}
	\widetilde{W}_{3}  & \geq\frac{2}{R}\mu_{\min}(\mathbb{R}^{d}\setminus
	B(\mathbf{\tilde{x}},a_{\mathbf{\tilde{z}}}))\mu_{\min}(B(\mathbf{\tilde{x}%
	},a_{\mathbf{\tilde{z}}}))\\
	& =\frac{1}{R}\mu_{\min}(K\setminus B(\mathbf{\tilde{x}},a_{\mathbf{\tilde{z}%
	}}))\mu_{\min}(B(\mathbf{\tilde{x}},a_{\mathbf{\tilde{z}}}))+\frac{1}{R}%
	\mu_{\min}(\mathbb{R}^{d}\setminus B(\mathbf{\tilde{x}},a_{\mathbf{\tilde{z}}%
	}))\mu_{\min}(B(\mathbf{\tilde{x}},a_{\mathbf{\tilde{z}}}))\\
	& \geq2\int_{\mathbb{R}^{d}\times\mathbb{R}^{d}}W(\mathbf{x},\mathbf{y}%
	)(\mu_{\min})_{|_{{K\setminus B(\mathbf{\tilde{x}},a_{\mathbf{\tilde{z}}})}}%
	}(d\mathbf{x})\mu_{\mathbf{\tilde{z}}}(d\mathbf{y})+\frac{1}{R}\mu_{\min
	}(\mathbb{R}^{d}\setminus B(\mathbf{\tilde{x}},a_{\mathbf{\tilde{z}}}%
	))\mu_{\min}(B(\mathbf{\tilde{x}},a_{\mathbf{\tilde{z}}})).
	\end{align*}

	Combining the above relations, recalling that $\mu_{\min}^{\ast}:=(\mu_{\min
	})_{|_{K\setminus B(\mathbf{z},a_{\mathbf{\tilde{z}}})}}+\mu_{\mathbf{\tilde
			{z}}},$ and invoking \eqref{strict}, we conclude that
	\[
	\mathcal{W}(\mu_{\min})\geq\mathcal{W}(\mu_{\min}^{\ast})+\frac{1}{R}\mu
	_{\min}(\mathbb{R}^{d}\setminus B(\mathbf{\tilde{x}},a_{\mathbf{\tilde{z}}%
	}))\mu_{\min}(B(\mathbf{\tilde{x}},a_{\mathbf{\tilde{z}}}))>\mathcal{W}%
	(\mu_{\min}^{\ast}).
	\]
	Since we proved earlier that ${\mathcal{V}}(\mu_{\min})\geq{\mathcal{V}}%
	(\mu_{\min}^{\ast})$, we conclude that $\mathcal{J}(\mu_{\min})>\mathcal{J}%
	(\mu_{\min}^{\ast})$, which contradicts that $\mu_{\min}$ is a minimizer of
	$\mathcal{J}$.\newline
\end{proof}

Unlike Assumption $\ref{C'1}$ (and most conditions imposed in the
literature), which holds independently of $\{\beta _{n}\},$ Assumption $\ref%
{C}\ref{C1}$ is $\{\beta _{n}\}$ dependent. More specifically, for a fixed $%
n,$ one can pick $\beta _{n}$ sufficiently large that (by the Laplace
principle) the measure $P_{n}$ in \eqref{tometro} mainly charges
configurations that are near-infimizers of $H_{n}$, which (as noted above)
have strictly smaller values than the minimizer of $\mathcal{J}$. In these
cases, the LDP could not possibly hold in $\mathcal{P}(\mathbb{R}^{d})$ with
rate function $\mathcal{J}_{\star }$. One might expect that if $\mathcal{J}$
(the{\ anticipated} rate function) is a tightness function, then the LDP
will hold for all{\ speeds} bigger than $n$. However, the comments above
show that this is not true, and it is possible that $\mathcal{J}$ is a
tightness function but the LDP does not hold on $\mathcal{P}(\mathbb{R}^{d})$
with rate function $\mathcal{J}$. It is likely nevertheless, that a
non-trivial LDP still holds true on some compactified space like $\mathcal{P}%
(\mathbb{R}^{d}\cup \{\infty \}).$ The following natural questions arise.

\begin{open problem}
	When potentials are non-diverging, what is the right space to prove LDPs
	when the sequence $\{\beta _{n}\}$ diverges so fast that Assumption \ref{C}%
	\ref{C1} is not satisfied, and are there different rate functions depending
	on the rate of divergence of $\{\beta _{n}\}$? Also, does there exist a
	critical speed $\{\beta _{n}^{\ast }\}$ such that if the LDP holds in $%
	\mathcal{P}(\mathbb{R}^{d})$ for $\{\beta _{n}^{\ast }\}$ then it also holds
	for every sequence $\{\beta _{n}\}$ with $\frac{\beta _{n}^{\ast }}{\beta
		_{n}}\rightarrow \infty $ and $\frac{\beta _{n}}{n}\rightarrow \infty ?$
\end{open problem}

\begin{remark}[Weak confinement on $\mathbb{R}^{2}$]\label{coulos}
	\emph{\ Another example in which the rate function $\mathcal{J}$ was shown
		to have a minimizer that does not have compact support can be found in \cite%
		{Hardy2012} { or in the more recent \cite{Chafai2019}}, where the particular case of the Coulomb potential in dimension 
		$d=2$ is studied. The proofs in \cite{Hardy2012} are based on specific
		properties of the Coulomb potential $-\log |\mathbf{x}-\mathbf{y}|$ and the
		complex plane. Our result works for more general $W,d,\beta _{n},$ and
		topological spaces $\mathcal{P}_{\psi }(\mathbb{R}^{d})$.
		We expect that the weak convergence methods of \cite{Dupuis} that we use here
		can be used to study other problems with weakly confining potentials. }
\end{remark}

\subsubsection{Discontinuous interaction potentials}

Our assumptions are also satisfied for cases where either $V$ and/or $W$ are
discontinuous. Example \ref{Vnoinfinity} already provides one example where $%
V$ is discontinuous. We now given another illustrative example where $W$ is
discontinuous.

\begin{example}
	\label{egWdisc} \emph{\ Suppose $V(\mathbf{x})=\|\mathbf{x}\|^{2},$ $\ell$
		is Lebesgue measure, and 
		\begin{equation*}
		W(\mathbf{x},\mathbf{y})=\sum_{i=1}^{N}I_{B_{i}}(\mathbf{x})I_{B_{i}}(%
		\mathbf{y})h_{i}(\mathbf{x},\mathbf{y}),
		\end{equation*}
		where $B_{i}=B(\mathbf{x}_{i},R_{i})$, with $R_i > 0$, $\mathbf{x}_i \in 
		\mathbb{R}^d, i = 1, \ldots, N$, is a collection of open balls with the
		property that the minimum distance between any two balls is $D > 0$, and
		each $h_{i}:\mathbb{R}^{d}\times\mathbb{R}^{d}\rightarrow\mathbb{R}%
		_{+}\cup\{\infty\}$ is any l.s.c. function for which Assumptions $\ref{A}$
		and $\ref{C}\ref{C2}$ hold. Then Assumptions \ref{A}, \ref{C'1} and \ref{C}%
		\ref{C2} are satisfied for $(V,W).$ }
\end{example}

Before we provide a justification of our assertions, note that in this
example $W$ can be interpreted as an interaction that takes place only when
both particles are inside the same ball $B_{i}$. For visualization purposes
one can take $h_{i}=K_{\Delta }$ for every $i$, with $K_{\Delta }$ from
Example \ref{pfffff}. A situation like this can arise with an electric
potential between particles that are positioned in different regions with
isolating boundaries.

\begin{proof}
	[Proof of Example \ref{egWdisc}.] Since each $h_{i}$ satisfies Assumption
	\ref{A}, we get that Assumption \ref{A} holds for $(V,W)$ with $a=0$, and
	Assumption \ref{C'1} holds immediately due to the fact that
	$V=\Vert x\Vert^{2}$ and the 
	definition of $W$. We now sketch a proof of why Assumption $\ref{C}\ref{C2}$
	also holds. Let $\mu\in\mathcal{P}(\mathbb{R}^{d}),$ and $\epsilon>0$. We set
	$\mu_{i}(\cdot)=\mu(\cdot\cap B_{i})$, $M=(\cup_{i=1}^{N}B_{i})^{c}$ and
	$\mu_{0}(\cdot)=\mu(\cdot\cap M).$ Since interactions take place only inside
	$B_{i}$ and $\int V(x)\mu(dx)$ is linear with respect to $\mu,$ we have
	$\mathcal{J}(\mu)=\sum_{i=0}^{N}\mathcal{J}(\mu_{i})$. We would like to
	approximate each $\mu_{i}$ by an absolutely continuous measure with the same
	total mass, and with energy close to the original and support inside its
	original support (so no new interaction occurs). By properties of integration,
	if $\mu_{i}^{\delta}(\cdot)=\frac{\mu_{i}(B(\mathbf{x}_{i},R_{i}))}{\mu
		_{i}(B(\mathbf{x}_{i},R_{i}-\delta))}\mu_{i}(\cdot\cap B(\mathbf{x}_{i}%
	,R_{i}-\delta)),$ then for all small $\delta>0$ we have $|\mathcal{J}(\mu
	_{i}^{\delta})-\mathcal{J}(\mu_{i})|\leq\epsilon,$ for all $i\in\{1,\dots
	,N\}$. It is possible that $\mu_{0}(\partial M)>0$. However, we can move this mass to
	the interior of $M$ by pushing $\mu_{0}$ forward under%
	\[
	f_{\delta}(\mathbf{x})=\sum_{i=1}^{N}I_{\partial B_{i}}(\mathbf{x})\left(
	\mathbf{x}+\frac{\delta D}{2}(\mathbf{x}-\mathbf{x}_{i})\right)
	+I_{K\setminus\cup_{i}^{N}\partial B_{i}}(\mathbf{x})\mathbf{x}%
	\]
	with $\delta>0$ small. We can even assume that the resulting $\mu_{0}^{\delta
	}$ has compact support in the interior of $M$ by removing the mass in a small neighborhood of $\partial M$, and then renormalizing to keep the total mass constant as
	was done for the other $\mu_{i}.$ For small $\delta>0$ it easy to see that,
	since only the continuous confined potential acts on it, $|\mathcal{J}(\mu
	_{0})-\mathcal{J}(\mu_{0}^{\delta})|\leq\epsilon.$ Since Assumption
	$\ref{C}\ref{C2}$ is satisfied for each $h_{i},$ we can apply it to get a
	measure $\mu_{i}^{\delta,n},$ absolutely continuous with respect to the
	Lebesgue measure and with the same mass as $\mu^{i},$ that is supported only
	inside $B_{i}$ such that $|\mathcal{J}(\mu_{i}^{\delta,n})-\mathcal{J}(\mu
	_{i}^{\delta})|\leq\epsilon,$ and $\mu_{i}^{\delta,n},\mu_{i}^{\delta}$ are
	close in the weak topology.
	We set $\mu_{0}^{\delta,n}=\mu_{0}^{\delta}\ast G_{n},$ where $G_{n}$ is a
	truncated Gaussian of radius $1/n$, which creates an absolutely continuous
	measure with support in $K$ for which $|\mathcal{J}(\mu_{0}^{\delta
		,n})-\mathcal{J}(\mu_{0}^{\delta})|\leq\epsilon,$ for large enough $n$. Then
	$\mu^{\delta,n}=\sum_{i=0}^{N}\mu_{i}^{\delta,n},$ satisfies $|\mathcal{J}%
	(\mu^{\delta,n})-\mathcal{J}(\mu)|\leq(2N+2)\epsilon,$ and also by making $n$
	big enough and $\delta>0$ small enough we can have $d_{w}(\mu^{\delta,n}%
	,\mu)\leq\epsilon.$
\end{proof}

{ 
	\begin{open problem} 
		In \cite{Chafai2014}, the (extended) continuity of $W$ was used as a sufficient condition for Assumption \ref{C}\ref{C2} to hold. Example \ref{egWdisc} above
		shows that continuity is not necessary for Assumption \ref{C}\ref{C2} to hold.
		Our preliminary investigations suggest that in some of these cases, it
		may be possible to establish LDPs with different rate functions,
		namely those  that are given by some
		type of regularization of the $\mathfrak{J}$ functional (like appropriate
		$\Gamma$-convergence relaxations but only with sequences that belong to
		specific subsets of the set of probability measures). We pose the existence
		of LDPs for cases where Assumption \ref{C}\ref{C2} fails as an open problem.
	\end{open problem}
}
\subsection{Outline of the paper}

\label{subs-out}

The structure of the rest of the article is as follows. In Section \ref%
{sec:rate} we provide definitions and lemmas that are used throughout the
paper and then show that the candidate rate functions introduced above are
indeed rate functions. In Section \ref{sec:4} we prove results for the speed 
$\beta _{n}=n,$ and in Section \ref{sec:5} we consider the case of speeds $%
\beta _{n}$ that grow faster than $n.$  { An outline of the weak convergence
approach and} proofs of several lemmas that are
needed for the main theorems are collected in the { Appendices. } 

\section{Rate Function Property}

\label{sec:rate}

In what follows, 
recall the set $\Psi$ defined in \eqref{cond-psi}. In Section \ref%
{subs:ratefn}, we show that under various combinations of Assumptions \ref{A}%
-\ref{C}, the functions $\mathcal{I}_{\star}$ and $\mathcal{J}_{\star}$
defined in \eqref{def:Istar} and \eqref{def:Jstar}, and for $\psi \in \Psi$,
the functions $\mathcal{I}_\star^\psi$ and $\mathcal{J}_\star^\psi$ defined
in \eqref{def:Istarpsi} and \eqref{def:Jstarpsi} are rate functions on the
spaces $\mathcal{P}(\mathbb{R}^{d})$ and $\mathcal{P}_{\psi}(\mathbb{R}%
^{d}), $ respectively. To begin with, in Section \ref{subs:basdef} we first
introduce basic notions that will be used in the rest of the paper.

\subsection{Basic definitions}

\label{subs:basdef}

\begin{definition}
	Let $\inset$ be an index set and let $\{\lambda_{a},\,a\in \inset\}\subset 
	\mathcal{P}\left( S\right) $. The collection $\{\lambda_{a},\,a\in \inset\}$
	is said to be tight if for every $\epsilon>0,$ there is a compact set $%
	K_{\epsilon }\subset S,$ such that $\inf\{\lambda_{a}\left(
	K_{\epsilon}\right) ,\,a\in \inset\}\geq1-\epsilon.$
\end{definition}

Furthermore, a sequence of random variables is said to be tight if and only
if the corresponding distributions are tight. The proofs of the following
three lemmas can be found in \cite[Section 2.2]{Bud}.

\begin{lemma}
	\label{tightness} A collection $\left\{ \lambda_{a},a\in \inset\right\}
	\subset\mathcal{P}(S)$ is tight if and only if there exists a tightness
	function $g:S\rightarrow[0,\infty]$ such that $\sup_{a\in \inset%
	}\int_{S}g(x)\lambda _{a}(dx)<\infty$.
\end{lemma}

\begin{lemma}
	\label{ayid} Let $g$ be a tightness function on $S$. Define $G:\mathcal{P}%
	(S)\rightarrow\lbrack0,\infty]$ by 
	\begin{equation*}
	G(\mu)=\int_{S}g(x)\mu(dx).
	\end{equation*}
	Then for each $M<\infty$ the set $\left\{ \mu\in\mathcal{P}(S):G(\mu)\leq
	M\right\} $ is tight (and hence precompact), and moreover, $G$ is a
	tightness function on $\mathcal{P}(S)$.
\end{lemma}

\begin{lemma}
	\label{expectation} Let $\left\{ \Lambda_{a},a\in \inset\right\} $ be random
	elements taking values in $\mathcal{P}(S)$ 
	and let $\lambda_{a}=E\Lambda_{a}$. Then $\left\{ \Lambda_{a},a\in \inset%
	\right\} $ is tight if and only if $\left\{ \lambda_{a},a\in \inset\right\} $
	is tight. In other words, a collection of random probability measures is
	tight if and only if the corresponding collection of ``means'' is tight in
	the space of (deterministic) probability measures.
\end{lemma}

The next result identifies a convenient tightness function on $\mathcal{P}%
_{\psi}\left( \mathbb{R}^{d}\right)$; see Appendix \ref{sec-apB} for a proof.

\begin{lemma}
	\label{tightnessfunction} Let $\psi \in \Psi$ and $\phi \in \Phi$, with $%
	\Psi $ and $\Phi$ as defined in \eqref{cond-psi} and \eqref{def-Phi},
	respectively. Then 
	\begin{equation*}
	{\mathcal{T}}(\mu):=\int_{\mathbb{R}^{d}}\phi\left( \psi\left( \mathbf{x}%
	\right) \right) \mu\left( d\mathbf{x}\right)
	\end{equation*}
	is a tightness function on $\mathcal{P}_{\psi}\left( \mathbb{R}^{d}\right) .$
\end{lemma}

Finally, it will be convenient to introduce the following projection
operators to define marginal distributions.

\begin{definition}
	\label{projection} We denote by $\pi^{k},k=1,2,$ the projection operators on
	a product space $S_{1}\times S_{2}$ defined by 
	\begin{equation*}
	\pi^{1}:(x_{1},x_{2})\rightarrow x_{1}\in S_{1},\hspace{20pt}%
	\pi^{2}:(x_{1},x_{2})\rightarrow x_{2}\in S_{2}.
	\end{equation*}
\end{definition}

\subsection{Verification of the rate function property}

\label{subs:ratefn}


\begin{lemma}
	\label{lsc} Suppose Assumption \ref{A} holds. Then $\mathcal{I}_{\star }$
	and $\mathcal{J}_{\star }$ defined in \eqref{def:Istar} and \eqref{def:Jstar}%
	, respectively, are lsc on $\mathcal{P}(\mathbb{R}^{d})$. Moreover, for $%
	\psi \in \Psi $, ${\mathcal{I}}_{\ast }^{\psi }$ and ${\mathcal{J}}_{\ast
	}^{\psi }$ defined in \eqref{def:Istarpsi} and \eqref{def:Jstarpsi},
	respectively, are lsc on ${\mathcal{P}}_{\psi }(\mathbb{R}^{d})$.
\end{lemma}

\begin{proof} 
	We start by showing that  the functional
	$\mathcal{J}_{a}$  defined in \eqref{fnal-J} is lsc.    For $\mu\in
	\mathcal{P}(\mathbb{R}^{d})$, let $\mu\otimes\mu$ denote the
	corresponding product measure on $\mathbb{R}^{d}\times\mathbb{R}^{d}$,
	and recall from \eqref{fnal-J} that  $\mathcal{J}_{a}(\mu)=\mathfrak{J}_{a}(\mu
	\otimes\mu)$, with $\mathfrak{J}_{a}$ defined as in \eqref{scr}. 
	The map $\mu\rightarrow\mu\otimes\mu$ from $\mathcal{P}%
	(\mathbb{R}^{d})$ to $\mathcal{P}(\mathbb{R}^{d}\times\mathbb{R}^{d})$ is
	continuous, and by Fatou's lemma (for weak convergence) the map $\zeta
	\mapsto\mathfrak{J}_{a}(\zeta)$ is lower semicontinuous if $W(\mathbf{x},\mathbf{y})+aV(\mathbf{x}%
	)+aV(\mathbf{y})$ is lower
	semicontinuous and bounded from below.  Since the latter property holds under
	Assumption \ref{A}, it follows that $\mathcal{J}_{a}%
	$ is lsc.  Since $\mathcal{I}=\mathcal{J}_{a}+\mathcal{R}%
	(\cdot|e^{-(1-a)V}\ell)$ and, as is well known, $\mathcal{R}\left(  \cdot|e^{-(1-a)V}\ell\right)  $ is
	lsc on $\mathcal{P}(\mathbb{R}^{d})$, this shows that $\mathcal{I}$,
	and hence $\mathcal{I}_*$, are lsc. 
	By the same argument,  the lower semicontinuity of $\mathcal{J}$ can be deduced from the fact that
	$\mathcal{J}=\mathfrak{J_{1}}(\mu\otimes\mu)$ where $\mathfrak{J}_{a}$ is given in
	\eqref{scr}, and the fact that $(\mathbf{x}, \mathbf{y})
	\mapsto$ $W(\mathbf{x},\mathbf{y})+V(\mathbf{x}%
	)+V(\mathbf{y})$ is lsc  and uniformly bounded from below due
	to Assumption \ref{A}, and from 
	\eqref{def:Jstar}, it follows that $\mathcal{J}_*$ is lsc. 	
	Since the topology on
	$\mathcal{P}_{\psi}(\mathbb{R}^{d})$ is stronger than 
	that on $\mathcal{P}(\mathbb{R}^{d}),$ it follows that both
	$\mathcal{I}_{\star}^\psi$ and  $\mathcal{J}_{\star}^\psi$  defined
	in \eqref{def:Istarpsi} and \eqref{def:Jstarpsi}, respectively,
	are also lsc on $\mathcal{P}_{\psi}(\mathbb{R}^{d}).$ 
\end{proof}

\begin{lemma}
	\label{lem-ratepsi} Suppose Assumption \ref{A} is satisfied. Then $\mathcal{I%
	}$ is a rate function on $\mathcal{P}(\mathbb{R}^{d}).$ If, in addition,
	there exists $\psi \in \Psi$ such that Assumption \ref{B} is satisfied, then 
	${\mathcal{I}}_{\star}^\psi$ is a rate function on $\mathcal{P}_{\psi}(%
	\mathbb{R}^{d})$.
\end{lemma}

\begin{proof}
	Since $\mathcal{I}_{\star}$ is lsc on  $\mathcal{P}(\mathbb{R}^{d})$
	by Lemma \ref{lsc},  
	it only remains to show that the level sets of $\mathcal{I}_{\star}$,
	or equivalently $\mathcal{I}$, are precompact  on
	$\mathcal{P}(\mathbb{R}^{d})$. 
	Since  
	$\mathcal{I}=\mathcal{J}_{a}+\mathcal{R}(\cdot|e^{-(1-a)V}\ell)$,
	this holds because
	$\mathcal{R}(\cdot|e^{-(1-a)V}\ell)$ is a rate function
	on $\mathcal{P}(\mathbb{R}^{d})$ and $\mathcal{J}_{a}$ is bounded below due to
	Assumption \ref{A}. Likewise, for $\psi \in \Psi$, to show
	that $\mathcal{I}_{\star}^\psi$ is a rate function, due to 
	Lemma \ref{lsc} it
	suffices to show that the level sets of $\mathcal{I}$ are compact in
	$\mathcal{P}_{\psi}(\mathbb{R}^{d}).$ 
	This follows from Lemma \ref{tightness} 
	and the fact that  Assumption
	\ref{B} implies there exists a function $\phi \in \Phi$ 
	such that if $\mathcal{I} \leq C$ then 
	$\int_{\mathbb{R}^{d}}\phi\left(  \psi\left(  \mathbf{x}\right)  \right)
	\mu\left(  d\mathbf{x}\right)  \leq C.$ 
\end{proof}

In what follows, for $a\in \lbrack 0,1]$ and $\mu \in {\mathcal{P}}(\mathbb{R%
}^{d})$, let 
\begin{equation}
\mathcal{J}_{a,\neq }(\mu ):=\frac{1}{2}\int_{(\mathbb{R}^{d}\times \mathbb{R%
	}^{d})_{\neq }}\left( W(\mathbf{x},\mathbf{y})+aV(\mathbf{x})+aV(\mathbf{y}%
)\right) \mu (d\mathbf{x})\mu (d\mathbf{y}),  \label{Janeq}
\end{equation}%
where we recall that $A_{\neq }$ is the set $A$ with its diagonal excised.
Also, denote ${\mathcal{J}}_{1,\neq }$ simply as ${\mathcal{J}}_{\neq }$.

\begin{lemma}
	\label{Jrate} Let $V$ and $W$ satisfy Assumption \ref{A} and Assumption \ref%
	{C}\ref{C1} (resp. Assumption \ref{C}\ref{C3} for some $\psi \in \Psi$),
	then $\mathcal{J}_{\star}$ (resp. $\mathcal{J}^{\psi}_{\star})$ are
	tightness functions on $\mathcal{P}(\mathbb{R}^{d})$ (resp. $\mathcal{P}%
	_{\psi}(\mathbb{R}^{d})).$
\end{lemma}

\begin{proof}
	From the definitions of $\mathcal{J}_{\star}$ and $\mathcal{J}_{\star}^{\psi}$
	in (\ref{def:Jstar}) and (\ref{def:Jstarpsi}), it is clear that to prove the
	lemma, it suffices to show that under Assumptions \ref{C}\ref{C1} and
	\ref{C}\ref{C3}, $\mathcal{J}$ is a tightness function in the respective
	spaces $\mathcal{P}(\mathbb{R}^{d})$ and $\mathcal{P}_{\psi}(\mathbb{R}^{d})$.
	The fact that $\mathcal{J}$ is lsc follows from Lemma \ref{lsc}. It remains to
	prove that the functionals have precompact level sets. For this, by Lemmas
	\ref{tightness} and \ref{ayid}, it suffices to prove that there exist
	$C^{\prime}>0$ and $C_{1}\in\mathbb{R}$ such that for every $\mu\in
	\mathcal{P}(\mathbb{R}^{d})$
	\begin{equation}
	2{\mathcal{J}}(\mu)=\int_{\mathbb{R}^{d}\times\mathbb{R}^{d}}\left(  V\left(
	\mathbf{x}\right)  +V\left(  \mathbf{y}\right)  +W\left(  \mathbf{x}%
	,\mathbf{y}\right)  \right)  \mu\left(  d\mathbf{x}\right)  \mu\left(
	d\mathbf{y}\right)  \geq C^{\prime}\int_{\mathbb{R}^{d}}\gamma(\mathbf{x}%
	)\mu(d\mathbf{x})+C_{1}\label{toshow}%
	\end{equation}
	for some tightness function $\gamma$. We will first prove that this is true
	for every $\mu\in\mathcal{P}(\mathbb{R}^{d})$ without atoms and with compact
	support,
	and then use a limiting argument. By Assumption \ref{C}\ref{C1}
	(resp.\ Assumption \ref{C}\ref{C3}), there exist a tightness function
	$\gamma:\mathbb{R}^{d}\mapsto\mathbb{R}$, $A\in{\mathcal{B}}(\mathbb{R}^{d})$
	and $C\in\mathbb{R}$ such that the inequality (\ref{C11}) holds:
	\begin{equation}
	\gamma\left(  \mathbf{x}\right)  I_{A^{c}}(\mathbf{x})+\gamma\left(
	\mathbf{y}\right)  I_{A^{c}}(\mathbf{y})\leq V\left(  \mathbf{x}\right)
	+V\left(  \mathbf{y}\right)  +W\left(  \mathbf{x},\mathbf{y}\right)
	+C.\label{gamma-ineq}%
	\end{equation}
	Fix $R<\infty$ and $\mu\in\mathcal{P}(\mathbb{R}^{d})$ whose support lies in
	$B(0,R)$. Integrating both sides of \eqref{gamma-ineq} with respect to $\mu$,
	we have
	\begin{equation}
	2\int_{A^{c}}\gamma\left(  \mathbf{x}\right)  \mu(d\mathbf{x})\leq
	\int_{\mathbb{R}^{d}\times\mathbb{R}^{d}}\left(  V\left(  \mathbf{x}\right)
	+V\left(  \mathbf{y}\right)  +W\left(  \mathbf{x},\mathbf{y}\right)  \right)
	\mu\left(  d\mathbf{x}\right)  \mu\left(  d\mathbf{y}\right)
	+C.\label{estimateone}%
	\end{equation}
	Since $\mu$ has no atoms, by Lemma \ref{approx} there exists a sequence
	$\mathbf{x}^{n}\in\mathbb{R}_{\neq}^{dn}$, $n\in\mathbb{N}$, such that
	$L_{n}:=L(\mathbf{x}^{n},\cdot)$ has support in $B(0,R)$, and $L_{n}%
	\overset{w}{\rightarrow}\mu$ and ${\mathcal{J}}_{\neq}(L_{n})\rightarrow
	{\mathcal{J}}(\mu)$ as $n\rightarrow\infty$. Therefore, by Assumption
	\ref{C}\ref{C1} (resp.\ Assumption \ref{C}\ref{C3}), and in particular
	(\ref{C12}), there exists $n_{0}\in\mathbb{N}$, such that $n\geq n_{0}$
	implies $r_{n}>R$, and hence%
	\begin{align*}
	\int_{A\cap B(0,R)}\gamma(\mathbf{x})L_{n}(d\mathbf{x})  & \leq\int_{(A\cap
		B(0,R)\times A\cap B(0,R))_{\neq}}\left(  V\left(  \mathbf{x}\right)
	+V\left(  \mathbf{y}\right)  +W\left(  \mathbf{x},\mathbf{y}\right)  \right)
	L_{n}(d\mathbf{x})L_{n}(d\mathbf{y})+C\\
	& \leq\int_{(\mathbb{R}^{d}\times\mathbb{R}^{d})_{\neq}}\left(  V\left(
	\mathbf{x}\right)  +V\left(  \mathbf{y}\right)  +W\left(  \mathbf{x}%
	,\mathbf{y}\right)  \right)  L_{n}(d\mathbf{x})L_{n}(d\mathbf{y})+C+|c|\\
	& ={\mathcal{J}}_{\neq}(L_{n})+C+|c|,
	\end{align*}
	where $c$ is the lower bound in Assumption \ref{A}. Combining this with the
	lower semicontinuity of $\gamma$, the fact that $L_{n}\overset{w}{\rightarrow
	}\mu$ and ${\mathcal{J}}_{\neq}(L_{n})\rightarrow{\mathcal{J}}(\mu)$ as
	$n\rightarrow\infty$, we see that
	\begin{align}
	\int_{A}\gamma(\mathbf{x})\mu(d\mathbf{x})=\int_{A\cap B(0,R)}\gamma
	(\mathbf{x})\mu(d\mathbf{x})  & \leq\liminf_{n\rightarrow\infty}\int_{A\cap
		B(0,R)}\gamma(\mathbf{x})L_{n}(d\mathbf{x})\nonumber\label{estimatetwo}\\
	& \leq\lim_{n\rightarrow\infty}{\mathcal{J}}_{\neq}(L_{n})+C+|c|\nonumber\\
	& ={\mathcal{J}}(\mu)+C+|c|.
	\end{align}
	Together, \eqref{estimateone} and \eqref{estimatetwo} imply there exists
	$C^{\prime}>0$, $C_{1}\in\mathbb{R}$ such that \eqref{toshow} holds for any
	$\mu\in{\mathcal{P}}(\mathbb{R}^{d})$ with compact support and without atoms.
	For general $\mu\in\mathcal{P}(\mathbb{R}^{d})$ without atoms, we define the
	sequence of measures $\mu_{N}(\cdot)=\frac{\mu(\cdot\cap B(0,N))}{\mu
		(B(0,N))},$ $N\in\mathbb{N}$, each of which has compact support and is without
	atoms. The relation \eqref{toshow} holds for each $N\in\mathbb{N},$ from which
	we obtain
	\[
	C^{\prime}\displaystyle\frac{\int_{\mathbb{R}^{d}}I_{B(0,N)}(\mathbf{x}%
		)\gamma\left(  \mathbf{x}\right)  \mu\left(  d\mathbf{x}\right)  }%
	{{\mu(B(0,N))}}+C_{1}\leq\displaystyle\frac{\int_{\mathbb{R}^{d}%
			\times\mathbb{R}^{d}}I_{B(0,N)}(\mathbf{x})I_{B(0,N)}(\mathbf{y})\left(
		V\left(  \mathbf{x}\right)  +V\left(  \mathbf{y}\right)  +W\left(
		\mathbf{x},\mathbf{y}\right)  \right)  \mu\left(  d\mathbf{x}\right)
		\mu\left(  d\mathbf{y}\right)  }{\mu^{2}(B(0,N))}.
	\]
	The integrands in the last inequality are bounded from below. Therefore,
	without loss of generality, we can assume that they are actually positive
	since otherwise we can just add and subtract their respective infima. By
	applying the monotone convergence theorem in the last relation, it follows
	that \eqref{toshow} holds for any $\mu\in{\mathcal{P}}(\mathbb{R}^{d})$
	without atoms.
	
	Finally, fix an arbitrary $\mu\in{\mathcal{P}}(\mathbb{R}^{d})$. Assume
	without loss of generality that ${\mathcal{J}}(\mu)<\infty$, for if not,
	\eqref{toshow} holds trivially. Then by Assumption \ref{C}\ref{C2}, there
	exists a sequence $\{\mu_{n}\}\subset{\mathcal{P}}(\mathbb{R}^{d})$ such that
	each $\mu_{n}$ is absolutely continuous with respect to the measure $\ell$
	(and consequently, non-atomic since $\ell$ is non-atomic), $\mu_{n}\overset
	{w}{\rightarrow}\mu$ and ${\mathcal{J}}(\mu_{n})\rightarrow{\mathcal{J}}(\mu
	)$. Since, as shown above, \eqref{toshow} holds when $\mu$ is replaced with
	$\mu_{n}$ for each $n$, taking the limit inferior as $n\rightarrow\infty$ of
	both sides and using the fact that $\liminf_{n\rightarrow\infty}%
	\int_{\mathbb{R}^{d}}\gamma(\mathbf{x})\mu(d\mathbf{x})\leq\liminf
	_{n\rightarrow\infty}\int_{\mathbb{R}^{d}}\gamma(\mathbf{x})\mu_{n}%
	(d\mathbf{x})$ since $\gamma$ is lsc, it follows that \eqref{toshow} also
	holds for any $\mu\in{\mathcal{P}}(\mathbb{R}^{d})$. 
\end{proof}

\section{Proof of Theorem \protect\ref{nspeed}}


\label{sec:4}

Throughout this section, we assume that Assumption \ref{A} is satisfied. To
establish the LDP stated in Theorem \ref{nspeed}, by \cite[Theorem 1.2.3]%
{Dupuis} we can equivalently verify the Laplace principle. For any
probability measure $P$, we use $\mathbb{E}_{P}$ to denote the corresponding
expectation, and for conciseness denote $\mathbb{E}_{\mathbb{P}}$ by $%
\mathbb{E}$. In view of the rate function property of $\mathcal{I}_{\star }$
and $\mathcal{I}_{\star }^{\psi }$ already established in Lemmas \ref{lsc}
and \ref{lem-ratepsi}, it suffices to show the following: 
for any bounded and continuous function $f$ on $S$, the Laplace principle 
\begin{equation}
\lim_{n\rightarrow \infty }-\frac{1}{n}\log \mathbb{E}_{Q_{n}}\left[ e^{-nf}%
\right] =\inf_{\mu \in S}\left\{ f\left( \mu \right) +\mathcal{I}_{\star
}\left( \mu \right) \right\} ,  \label{laplace-nspeed}
\end{equation}%
holds both for $S=\mathcal{P}(\mathbb{R}^{d})$ and (under the additional
condition stated as Assumption \ref{B}) with $S=\mathcal{P}_{\psi }(\mathbb{R%
}^{d})$ and $\mathcal{I}_{\star }$ replaced by $\mathcal{I}_{\star }^{\psi }$%
.

\begin{remark}{(Completeness of $S$ is not necessary.)}
	\emph{While the statement of \cite[Theorem 1.2.3]{Dupuis} assumes
		completeness of the space $S$, a review of the proof shows that this
		property is not needed (though compactness of the level sets of }$\mathcal{I}%
	_{\star}$\emph{\ is used).}
\end{remark}

To establish the bound \eqref{laplace-nspeed}, we first express $-\frac{1}{n}%
\log \mathbb{E}_{Q_{n}}\left[ e^{-nf}\right] $ in terms of a variational
problem (equivalently, a stochastic control problem). We then prove
tightness of nearly minimizing controls, and finally prove convergence of
the values of the corresponding controlled problems to the value of the
limiting variational problem. The last step is reminiscent of the notion of $%
\Gamma $-convergence that is often used for analyzing variational problems
in the analysis community. For a nice exposition of the relationship between
LDPs and $\Gamma $-convergence, the reader is referred to \cite{Mariani2012}.

\subsection{Representation formula}

\label{subs-repth1}

Recall that $P_{n}$ is the probability measure $\mathbb{R}^{dn}$ defined in %
\eqref{tometro} and $Q_{n}$ is the push forward of $P_{n}$ under $L_{n}.$
Let $a\in \lbrack 0,1)$ be the constant in Assumption \ref{A} and let $%
P_{n}^{\star }$ be the measure on $\mathbb{R}^{nd}$ defined by 
\begin{equation}
P_{n}^{\star }\left( d\mathbf{x}_{1},\ldots ,d\mathbf{x}_{n}\right)
:=e^{-\sum_{i=1}^{n}(1-a)V\left( \mathbf{x}_{i}\right) }\ell (d\mathbf{x}%
_{1})\cdots \ell (d\mathbf{x}_{n}),  \label{def-pnstar}
\end{equation}%
and note that it is a probability measure due to Assumption \ref{A} and
Remark \ref{rem-A}. Let $\mathcal{J}_{a,\neq }$ be defined as in %
\eqref{Janeq}: 
$\mathcal{J}_{a,\neq }(\mu )=\frac{1}{2}\int_{(\mathbb{R}^{d}\times \mathbb{R%
	}^{d})_{\neq }}\left( W(\mathbf{x},\mathbf{y})+aV(\mathbf{x})+aV(\mathbf{y}%
)\right) \mu (d\mathbf{x})\mu (d\mathbf{y}).$ 
When $\beta _{n}=n$, using \eqref{def-pnstar}, \eqref{tometro} and %
\eqref{def-Hnequiv} to calculate $dP_{n}/dP_{n}^{\ast }$, we see that 
for any measurable function $f$ on $\mathcal{P}(\mathbb{R}^{d})$ (or on $%
\mathcal{P}_{\psi }(\mathbb{R}^{d})$), we have 
\begin{equation}
-\frac{1}{n}\log \mathbb{E}_{Q_{n}}\left[ e^{-nf}\right] =-\frac{1}{n}\log 
\mathbb{E}_{P_{n}}\left[ e^{-nf\circ L_{n}}\right] =-\frac{1}{n}\log \mathbb{%
	E}_{P_{n}^{\star }}\left[ \frac{1}{Z_{n}}e^{-n(f+\mathcal{J}_{a,\neq }+\frac{%
		a}{n}\mathcal{V})\circ L_{n}}\right] ,  \label{eq-split}
\end{equation}%
where $Z_{n}$ is the normalizing constant defined in \eqref{Zndef} and ${%
	\mathcal{V}}$ is the functional defined in \eqref{Vdef}. 

We next state a representation for the quantity on the right-hand side of %
\eqref{eq-split}. To avoid confusion with the original distributions and
random variables, we use an overbar (e.g., $\bar{L}_{n}$) for quantities
that will appear in the representation, and refer to them as
\textquotedblleft controlled\textquotedblright\ versions. Given a
probability measure $\bar{P}^{n}\in {\mathcal{P}}(\mathbb{R}^{dn})$, 
we can factor it into conditional distributions in the following manner: 
\begin{equation*}
\bar{P}^{n}(d\mathbf{x}_{1},\ldots ,d\mathbf{x}_{n})={\bar{P}}_{\{1\}}^{n}(d%
\mathbf{x}_{1}){\bar{P}}_{\{2\}|\{1\}}^{n}(d\mathbf{x}_{2}|\mathbf{x}%
_{1})\cdots {\bar{P}}_{\{n\}|\{1,..,n-1\}}^{n}(d\mathbf{x}_{n}|\mathbf{x}%
_{1},\ldots ,\mathbf{x}_{n-1}),
\end{equation*}%
where for $i=1,...,n,\,\bar{P}_{\{i\}|\{1,...,i-1\}}(\cdot |\mathbf{x}%
_{1},...,\mathbf{x}_{i-1})$ denotes the conditional distribution of the $i$%
-th marginal given $\mathbf{x}_{1},...,\mathbf{x}_{i-1}.$ Thus, if $\{\bar{%
	\boldsymbol{X}}_{j}^{n}\}_{1\leq j\leq n}$ are random variables with joint
distribution $\bar{P}^{n}(d\mathbf{x}_{1}\cdots d\mathbf{x}_{n})$ on some
probability space $(\Omega ,\mathcal{F},\mathbb{P})$, then $\bar{\mu}_{i}^{n}
$, the conditional distribution of $\bar{\mathbf{X}}_{i}^{n}$ given $\bar{%
	\mathbf{X}}_{1}^{n},\dots ,\bar{\mathbf{X}}_{i-1}^{n}$, can be expressed as 
\begin{equation}
\bar{\mu}_{i}^{n}(d\mathbf{x}_{i}):={\bar{P}}_{\{i\}|\{1,...,i-1\}}^{n}(d%
\mathbf{x}_{i}|\mathbf{\bar{X}}_{1}^{n},\dots ,\mathbf{\bar{X}}_{i-1}^{n}).
\label{Page 4}
\end{equation}%
Note that $\bar{\mu}_{i}^{n}$, $1\leq i\leq n$, are random probability
measures, and the $i$th measure is measurable with respect to the $\sigma $%
-algebra generated by $\{\bar{\boldsymbol{X}}_{j}^{n}\}_{j<i}.$ We refer to
the collection $\{\bar{\mu}_{i}^{n},1\leq i\leq n\}$ as a control, and let ${%
	\bar{L}}_{n}(\cdot )=L_{n}(\bar{\mathbf{X}}^{n};\cdot )$, with $L_{n}$
defined by \eqref{Ln}, be the (random) empirical measure of $\{\bar{%
	\boldsymbol{X}}_{j}^{n}\}_{1\leq j\leq n},$ which we refer to as the
controlled empirical measure.

Let $f$ belong to the space of functions on $\mathcal{P}(\mathbb{R}^{d})$
(or $\mathcal{P}_{\psi }(\mathbb{R}^{d})$) such that the map $\mathbf{x}%
^{n}\mapsto $ $f(L_{n}(\mathbf{x}^{n};\cdot ))$ from $\mathbb{R}^{nd}$ to $%
\mathbb{R}$ is measurable and bounded from below. This space clearly
includes all bounded continuous functions on $\mathcal{P}(\mathbb{R}^{d})$
(respectively, $\mathcal{P}_{\psi }(\mathbb{R}^{d})).$ Then, since the
functional $\mathcal{J}_{a,\neq }$ is also measurable and bounded from below
(due to Assumption \ref{B}), we can apply \cite[Proposition 4.5.1]{Dupuis}
to the function $\mathbf{x}^{n}\in \mathbb{R}^{d}\mapsto f(L_{n}(\mathbf{x}%
^{n};\cdot ))+\mathcal{J}_{a,\neq }(L_{n}(\mathbf{x}^{n};\cdot )),$ to
obtain 
\begin{equation}
-\frac{1}{n}\log \mathbb{E}_{P_{n}^{\star }}\!\left[ e^{-n(f+\mathcal{J}%
	_{a,\neq }+\frac{a}{n}\mathcal{V})\circ L_{n}}\!\right] {=}\inf_{\{\bar{\mu}%
	^{n}\}}\!\mathbb{E}\!\left[ f\left( {\bar{L}}_{n}\right) +\mathcal{J}%
_{a,\neq }(\bar{L}_{n})+\frac{a}{n}\mathcal{V}\left( {\bar{L}}_{n}\right) +%
\mathcal{R}\left( \bar{P}^{n}|\otimes _{n}e^{-(1-a)V}\ell \right) \right] ,
\end{equation}%
where $\bar{L}_{n}$ is the controlled empirical measure associated with $%
\bar{P}^{n}$ as defined above, and the infimum is over all controls $\{\bar{%
	\mu}_{i}^{n}\}$ defined in terms of some joint distribution $\bar{P}^{n}\in {\mathcal{P}}(\mathbb{R}^{dn})$ via \eqref{Page 4}. 
Factoring $\bar{P}^{n}$ as above and using the chain rule for relative
entropy (see \cite[Theorem B.2.1]{Dupuis}), we then have 
\begin{equation}
{-}\frac{1}{n}\!\log \mathbb{E}_{P_{n}^{\star }}\!\left[ e^{-n(f+\mathcal{J}%
	_{a,\neq }+\frac{a}{n}\mathcal{V})\circ L_{n}}\!\right] {=}\inf_{\{\bar{\mu}%
	_{i}^{n}\}}\!\mathbb{E}\!\left[ f\left( {\bar{L}}_{n}\right) +\mathcal{J}%
_{a,\neq }\left( {\bar{L}}_{n}\right) +\frac{a}{n}\mathcal{V}\left( {\bar{L}}%
_{n}\right) +\frac{1}{n}\!\sum_{i=1}^{n}\!\mathcal{R}\left( \bar{\mu}%
_{i}^{n}|e^{-(1-a)V}\ell \right) \!\right] \!\!,  \label{representation}
\end{equation}%
where the infimum is over all controls $\{\bar{\mu}_{i}^{n}\}$
(equivalently, joint distributions $\bar{P}^{n}$ $\in {\mathcal{P}}(\mathbb{R%
}^{nd})$). Also, setting $f=0$ in (\ref{representation}) and recalling the
definition of $Z_{n}$ from \eqref{Zndef} gives 
\begin{equation}
-\frac{1}{n}\!\log \!Z_{n}{=}-\!\frac{1}{n}\!\log \mathbb{E}_{P_{n}^{\star
}}\!\left[ \!e^{-n\mathcal{J}_{a,\neq }(L_{n})+\frac{a}{n}\mathcal{V}\left(
	L_{n}\right) }\right] \!{=}\!\inf_{\{\bar{\mu}_{i}^{n}\}}\!\mathbb{E}\!\left[
\!\mathcal{J}_{a,\neq }\!\left( {\bar{L}}_{n}\right) {+}\frac{a}{n}\mathcal{V%
}\!\left( {\bar{L}}_{n}\right) +\!\frac{1}{n}\!\sum_{i=1}^{n}\!\mathcal{R}%
\left( \bar{\mu}_{i}^{n}|e^{-(1-a)V}\ell \right) \!\!\right] \!\!.
\label{rep-zn}
\end{equation}%
We claim that to prove Theorem \ref{nspeed}, it suffices to show that for
every bounded and continuous (in the respective topology) function $f$, the
lower bound 
\begin{equation}
\liminf_{n\rightarrow \infty }\inf_{\{\bar{\mu}_{i}^{n}\}}\mathbb{E}\left[
f\left( \bar{L}_{n}\right) +\mathcal{J}_{a,\neq }\left( \bar{L}_{n}\right) +%
\frac{a}{n}\mathcal{V}\left( {\bar{L}}_{n}\right) +\frac{1}{n}\sum_{i=1}^{n}%
\mathcal{R}\left( \bar{\mu}_{i}^{n}|e^{-V}\ell \right) \right] \geq
\inf_{\mu \in S}\{f\left( \mu \right) +\mathcal{I}\left( \mu \right) \}
\label{th1-lb}
\end{equation}%
and upper bound 
\begin{equation}
\limsup_{n\rightarrow \infty }\inf_{\{\bar{\mu}_{i}^{n}\}}\mathbb{E}\left[
f\left( \bar{L}_{n}\right) +\mathcal{J}_{a,\neq }\left( \bar{L}_{n}\right) +%
\frac{a}{n}\mathcal{V}\left( {\bar{L}}_{n}\right) +\frac{1}{n}\sum_{i=1}^{n}%
\mathcal{R}\left( \bar{\mu}_{i}^{n}|e^{-V}\ell \right) \right] \leq
\inf_{\mu \in S}\{f\left( \mu \right) +\mathcal{I}\left( \mu \right) \}
\label{th1-ub}
\end{equation}%
hold. Indeed, when combined with \eqref{representation}, \eqref{rep-zn} and %
\eqref{eq-split}, these bounds imply the desired limit \eqref{laplace-nspeed}%
. The lower and upper bounds are established in Sections \ref{subs-lbth1}
and \ref{Upn}, respectively. First, in Section \ref{subs-tightprops}, we
establish some tightness properties of the controls that will be used in the
proofs of these bounds.

\subsection{Properties of the controls}

\label{subs-tightprops}

We continue to use the notation for the controls introduced in the previous
section. We start with a simplifying observation.

\begin{remark}
	\label{tightremark} \emph{In the proof of the lower bound \eqref{th1-lb}, we
		can assume that there exists $C_{0}<\infty$ such that 
		\begin{equation}
		\sup_{n\in\mathbb{N}}\inf_{\{\bar{\mu}_{i}^{n}\}}\mathbb{E}\left[ \mathcal{J}
		_{a,\neq}\left( \bar{L}_{n}\right) +\frac{a}{n}\mathcal{V}\left( {\bar{L}}%
		_{n}\right) +\frac{1}{n}\sum_{i=1}^{n}\mathcal{R}\left( \bar{\mu}%
		_{i}^{n}|e^{-(1-a)V}\ell\right) \right] \leq C_{0}.  \label{idontknowa}
		\end{equation}
		If this were not true, we could restrict to a subsequence that has such a
		property, because for any subsequence for which the left-hand side of %
		\eqref{idontknowa} is infinite, the lower bound \eqref{th1-lb} is satisfied
		by default. Furthermore, since under Assumption \ref{A}, $\mathcal{J}
		_{a,\neq}>\min\{0,2c\}$, we can restrict to controls for which the relative
		entropy cost is bounded by $C_{0}+2|c|$: that is, for which 
		\begin{equation}
		\sup_{n}\mathbb{E}\left[ \frac{1}{n}\sum_{i=1}^{n}\mathcal{R}\left( \bar {\mu%
		}_{i}^{n}|e^{-(1-a)V}\ell\right) \right] \leq C_{0}+2|c|.
		\label{controlingcontrols}
		\end{equation}
	}
\end{remark}

\begin{lemma}
	\label{aaa} Let $V$ satisfy Assumption \ref{A}, let $\{\bar{\mu }_{i}^{n}\},
	n\in\mathbb{N},$ be a sequence of controls for which %
	\eqref{controlingcontrols} holds, let $\bar{L}^{n}$ be the associated
	sequence of controlled empirical measures and let 
	\begin{equation}  \label{hat}
	\hat{\mu}_{n}:=\frac{1}{n}\sum_{i=1}^{n}\bar{\mu}_{i}^{n}.
	\end{equation}
	Then $\left\{ \left( {\bar{L}}_{n},\hat{\mu}_{n}\right) ,n\in \mathbb{N}%
	\right\} $ is tight as a sequence of $\mathcal{P}(\mathbb{R}^{d})\times%
	\mathcal{P}(\mathbb{R}^{d})$-valued random elements.
\end{lemma}

\begin{proof}
	Let $\{\bar{\mu}_{i}^{n}\}, n \in \mathbb{N},$ be a sequence of controls that satisfies
	\eqref{controlingcontrols}. By the convexity of relative entropy and Jensen's
	inequality
	\[
	\sup_{n}\mathbb{E}\left[  \mathcal{R}\left(  \hat{\mu}^{n}|e^{-(1-a)V}\ell\right)
	\right]  <\infty.
	\]
	We know that $\mathcal{R}\left(  \cdot|e^{-(1-a)V}\ell\right)  $ is a tightness
	function on $\mathcal{P}\left(  \mathbb{R}^{d}\right)  $
	and hence, by Lemma \ref{tightness}, the sequence of random
	probability measures $\{\hat{\mu}_{n},n\in\mathbb{N}\}$ is tight.
	By Lemma \ref{expectation}, the sequence of probability measures
	$\{\mathbb{E}[\hat{\mu}_{n}],\,n\in\mathbb{N}\}$ is tight.
	Since $\bar{\mu}_{i}^{n}$ is the conditional distribution of $\mathbf{\bar{X}%
	}_{i}^{n}$ given $(\bar{\mathbf{X}}_{1}^{n},...,\bar{\mathbf{X}}_{i-1}^{n})$,
	for any measurable function $g:\mathbb{R}^{d}\mapsto\mathbb{R}$ that is
	bounded from below, we have
	\begin{equation*}%
	\begin{split}
	\mathbb{E}\left[  \int_{\mathbb{R}^{d}}g\left(  \mathbf{x}\right)  {\bar{L}%
	}_{n}\left(  d\mathbf{x}\right)  \right]   &  =\mathbb{E}\left[  \frac{1}%
	{n}\sum_{i=1}^{n}g\left(  \mathbf{\bar{X}}_{i}^{n}\right)  \right]  =\mathbb{E}\left[  \frac{1}{n}\sum_{i=1}^{n}\int_{\mathbb{R}^{d}}g\left(
	\mathbf{x}\right)  \bar{\mu}_{i}^{n}\left(  d\mathbf{x}\right)  \right]
	=\mathbb{E}\left[  \int_{\mathbb{R}^{d}}g\left(  \mathbf{x}\right)  \hat{\mu
	}_{n}\left(  d\mathbf{x}\right)  \right]  .
	\end{split}
	\end{equation*} 
	Thus, $\mathbb{E}\left[  {\bar{L}}_{n}\right]  =\mathbb{E}\left[  \hat{\mu
	}_{n}\right]  ,$ and so $\{ \mathbb{E} [ \bar{L}_n], n \in
	\mathbb{N}\}$ is also tight.  Another application of Lemma
	\ref{expectation} then shows that $\{{\bar{L}}_{n},n\in\mathbb{N}\},$ is
	tight, which  together with the tightness of 
	$\{\hat{\mu}^{n}\}$ established above,  implies 
	$\left\{  \left(  \hat{\mu}_{n},{\bar{L}}_{n}\right)  ,n\in \mathbb{N}\right\}  $ is tight.
\end{proof}

The following lemma, which uses an elementary martingale argument, appears
in \cite{Dupuis}. For the reader's convenience the proof is given in
Appendix \ref{sec-apC}.

\begin{lemma}
	\label{lem:weak_limit_Sanov} Suppose $\bar{L}_{n},\hat{\mu}_{n}$, $n\in%
	\mathbb{N}$, are as defined in Lemma \ref{aaa} and further assume that $%
	\left\{ \left( \bar{L}_{n},\hat{\mu}_{n}\right) ,n\in\mathbb{N}\right\} $
	converges along a subsequence to $\left( \bar{L},\hat{\mu}\right) .$ Then $%
	\bar{L}=\hat{\mu}$ w.p.1.
\end{lemma}

For the next result, it will be convenient to first define a collection of
auxiliary random measures that extend the ones that appear in the
representation (\ref{representation}). Let $\bar{P}^{n}$ be a probability
measure on $\mathbb{R}^{dn}$, and let $(\bar{\mathbf{X}}_{1}^{n},\ldots ,%
\bar{\mathbf{X}}_{n}^{n})$ be random variables with joint distribution $\bar{%
	P}^{n}$. For $J\subset\{1,...,n\},$ let $\bar{P}_{J}^{n}$ equal the marginal
distribution of $\bar{P}^{n}$ on $\{\mathbf{x}_{j},j\in J\}$, and for
disjoint subsets $I_{1}$ and $I_{2}$ of $\left\{ 1,\ldots,n\right\} $, let $%
\bar {P}_{I_{1}|I_{2}}^{n}$ denote the stochastic kernel defined as follows: 
\begin{equation*}
\bar{P}_{I_{1}|I_{2}}^{n}(d\mathbf{x}_{i},i\in I_{1}|\mathbf{x}_{k},k\in
I_{2})\bar{P}_{I_{2}}^{n}(d\mathbf{x}_{k},k\in I_{2})=\bar{P}_{I_{1}\cup
	I_{2}}^{n}(d\mathbf{x}_{j},j\in I_{1}\cup I_{2}).
\end{equation*}
Let $K_{k}:=\{1,\ldots,k-1\}.$ In the sequel we fix $i<j$ (the case $j<i$
can be handled in a symmetric way), and define 
\begin{equation}
\bar{\mu}_{ij}^{n}(d\mathbf{x}_{i}d\mathbf{x}_{j}):=\bar{P}%
_{\{i,j\}|K_{i}}^{n}(d\mathbf{x}_{i}d\mathbf{x}_{j}|\mathbf{\bar{X}}%
_{k}^{n},k\in K_{i}).  \label{bar-muijn}
\end{equation}
Also, note that with this notation 
\begin{equation}  \label{barmuin}
\bar{\mu}_{i}^{n}(d\mathbf{x}_{i})={\bar{P}}_{\{i\}|K_{i}}^{n}(d\mathbf{x}%
_{i}|\mathbf{\bar{X}}_{1}^{n},\dots,\mathbf{\bar{X}}_{i-1}^{n})
\end{equation}
are the controls used in the representation (\ref{representation}). We claim
that 
\begin{equation}
\pi_{\#}^{1}\bar{\mu}_{ij}^{n}=\bar{\mu}_{i}^{n}\quad\mbox{ and }\quad\pi
_{\#}^{2}\bar{\mu}_{ij}^{n}=\mathbb{E}[\bar{\mu}_{j}^{n}|\mathbf{\bar{X}}%
_{k}^{n},k\in K_{i}],  \label{pi-claim}
\end{equation}
where $\pi^{k},k=1,2,$ and $\#$ are the projection and push-forward
operators introduced in Definition \ref{projection} and Definition \ref{push}%
. The first relation in \eqref{pi-claim} is an immediate consequence of the
definitions of $\bar{\mu}_i^n$ and $\bar {\mu}_{ij}^{n}.$ Due to the
asymmetry in the first and second (equivalently, $i$ and $j$) coordinates in
the definition of $\bar{\mu}_{ij}^n$ in \eqref{bar-muijn}, the proof of the
second equality in \eqref{pi-claim} is a little more involved. Indeed, note
that for every $A \subset {\mathcal{B}}(\mathbb{R}^d)$, 
\begin{align*}
& \pi_{\#}^{2}\bar{\mu}_{ij}^{n}(A) =\pi_{\#}^{2}\bar{P}_{\{i,j%
	\}|K_{i}}^{n}(A|\mathbf{\bar{X}}_{k}^{n},k\in K_{i})=\int\bar{P}%
_{\{j\}|K_{i+1}}^{n}(A|\mathbf{\bar{X}}_{1}^{n},...,\mathbf{\bar{X}}%
_{i-1}^{n},\mathbf{x}_{i})\bar{P}_{\{i\}|K_{i}}^{n}(d\mathbf{x}_{i}|\mathbf{%
	\bar{X}}_{k}^{n},k\in K_{i}) \\
& \quad=\int\bar{P}_{\{j\}|K_{j}}^{n}(A|\mathbf{\bar{X}}_{1}^{n},...,\mathbf{%
	\bar{X}}_{i-1}^{n},\mathbf{x}_{i},...,\mathbf{x}_{j-1})\bar{P}%
_{(K_{j}\setminus K_{i})|K_{i}}^{n}(d\mathbf{x}_{i}\cdots d\mathbf{x}_{j-1}|%
\mathbf{\bar{X}}_{1}^{n},...,\mathbf{\bar{X}}_{i-1}^{n}) \\
& \quad=\mathbb{E}[\bar{P}_{\{j\}|K_{j}}^{n}(A|\mathbf{\bar{X}}_{k}^{n},k\in
K_{j})|\mathbf{\bar{X}}_{k}^{n},k\in K_{i}]=\mathbb{E}[\bar{\mu}_{j}^{n}|%
\mathbf{\bar{X}}_{k}^{n},k\in K_{i}](A),
\end{align*}
from which the second equality in \eqref{pi-claim} follows.

\begin{lemma}
	\label{aaaa} For $\psi \in \Psi $ let $V$ and $W$ satisfy Assumptions \ref{A}
	and \ref{B}, and let $\{\bar{\mu}_{i}^{n}\},n\in \mathbb{N},$ be a sequence
	of controls for which 
	\begin{equation}
	\sup_{n\in \mathbb{N}}\mathbb{E}\left[ \mathcal{J}_{a,\neq }\left( \bar{L}%
	_{n}\right) +\frac{a}{n}\mathcal{V}\left( \bar{L}_{n}\right) +\frac{1}{n}%
	\sum_{i=1}^{n}\mathcal{R}\left( \bar{\mu}_{i}^{n}|e^{-V}\ell \right) \right]
	<\infty ,  \label{controlingcontrols2}
	\end{equation}%
	and let $\hat{\mu}_{n}$ be as defined in \eqref{hat}. Then $\left\{ \left( {%
		\bar{L}}_{n},\hat{\mu}_{n}\right) ,n\in \mathbb{N}\right\} $ is tight in $%
	\mathcal{P}_{\psi }(\mathbb{R}^{d})\times \mathcal{P}_{\psi }(\mathbb{R}%
	^{d}) $.
\end{lemma}

\begin{proof}
	Let $\theta$ be a probability measure on $\mathbb{R}^{d}$. By the chain rule
	for relative entropy, we have
	\begin{align*}
	&  \mathcal{R}(\bar{P}_{\{i,j\}|K_{i}}^{n}(d\mathbf{x}_{i}d\mathbf{x}%
	_{j}|\mathbf{x}_{k},k\in K_{i})\left\Vert \theta(d\mathbf{x}_{i}%
	)\theta(d\mathbf{x}_{j})\right.  )\\
	&  \quad=\int\mathcal{R}(\bar{P}_{\{j\}|K_{i+1}}^{n}(d\mathbf{x}%
	_{j}|\mathbf{x}_{k},k\in K_{i+1})\left\Vert \ell(d\mathbf{x}_{j})\right.
	)\bar{P}_{\{i\}|K_{i}}^{n}(d\mathbf{x}_{i}|\mathbf{x}_{k},k\in K_{i})\\
	&  \quad\quad+\mathcal{R}(\bar{P}_{\{i\}|K_{i}}^{n}(d\mathbf{x}_{i}%
	|\mathbf{x}_{k},k\in K_{i})\left\Vert \theta(d\mathbf{x}_{i})\right.  ).
	\end{align*}
	In addition, Jensen's inequality gives%
	\begin{align*}
	&  \mathcal{R}(\bar{P}_{\{j\}|K_{i+1}}^{n}(d\mathbf{x}_{j}|\mathbf{x}_{k},k\in
	K_{i+1})\left\Vert \theta(d\mathbf{x}_{j})\right.  )\\
	&  \,=\mathcal{R}\left(  \left.  \int\bar{P}_{\{j\}|K_{j}}^{n}(d\mathbf{x}%
	_{j}|\mathbf{x}_{k},k\in K_{j})\bar{P}_{(K_{j}\setminus K_{i+1})|K_{i+1}}%
	^{n}(d\mathbf{x}_{i+1}\cdots d\mathbf{x}_{j-1}|\mathbf{x}_{k},k\in
	K_{i+1})\right\Vert \theta(d\mathbf{x}_{j})\right) \\
	&  \,\leq\int\mathcal{R}\left(  \left.  \bar{P}_{\{j\}|K_{j}}^{n}%
	(d\mathbf{x}_{j}|\mathbf{x}_{k},k\in K_{j})\right\Vert \theta(d\mathbf{x}%
	_{j})\right)  \bar{P}_{(K_{j}\setminus K_{i+1})|K_{i+1}}^{n}(d\mathbf{x}%
	_{i+1}\cdots d\mathbf{x}_{j-1}|\mathbf{x}_{k},k\in K_{i+1}).
	\end{align*}
	Combining the last two displays with \eqref{bar-muijn} and \eqref{barmuin}, we
	obtain
	\begin{equation}
	\mathbb{E}\left[  \mathcal{R}(\bar{\mu}_{ij}^{n}(d\mathbf{x}_{i}%
	d\mathbf{x}_{j})\left\Vert \theta(d\mathbf{x}_{i})\theta(d\mathbf{x}%
	_{j})\right.  )\right]  \leq\mathbb{E}[\mathcal{R}(\bar{\mu}_{j}%
	^{n}(d\mathbf{x}_{j})\left\Vert \theta(d\mathbf{x}_{j})\right.  )+\mathcal{R}%
	(\bar{\mu}_{i}^{n}(d\mathbf{x}_{i})\left\Vert \theta(d\mathbf{x}_{i})\right.
	)]. \label{eqn:REbound}%
	\end{equation}

	Using (\ref{eqn:REbound}) with $\theta=e^{-(1-a)V}\ell$, the definition of
	${\mathcal{J}}_{a,\neq}$ in \eqref{Janeq} and the tower property of
	conditional expectations to get the first inequality below, we have
	\begin{equation}%
	\begin{split}
	&  \mathbb{E}\left[  \mathcal{J}_{a,\neq}\left(  {\bar{L}}_{n}\right)
	+\frac{a}{n}\mathcal{V}\left(  \bar{L}_{n}\right)  +\frac{1}{n}\sum_{i=1}%
	^{n}\mathcal{R}\left(  \bar{\mu}_{i}^{n}|e^{-(1-a)V}\ell\right)  \right]  \\
	&  =\mathbb{E}\left[  \mathcal{J}_{a,\neq}\left(  \bar{L}_{n}\right)
	+\frac{a}{n}\mathcal{V}\left(  \bar{L}_{n}\right)  +\frac{1}{n(n-1)}%
	(n-1)\sum_{i=1}^{n}\mathcal{R}\left(  \bar{\mu}_{i}^{n}|e^{-(1-a)V}%
	\ell\right)  \right]  \\
	&  \geq\mathbb{E}\Bigg[\frac{1}{2n^{2}}\sum_{i<j}\int_{\mathbb{R}^{d}}\left(
	W\left(  \mathbf{\bar{X}}_{i}^{n},\mathbf{x}_{j}\right)  +aV(\mathbf{\bar{X}%
	}_{i}^{n})+aV(\mathbf{x}_{j})\right)  \bar{P}_{\{j\}|K_{i+1}}(d\mathbf{x}%
	_{j}|\mathbf{\bar{X}}_{k}^{n},k\in K_{i+1})\\
	&  +\frac{1}{2n^{2}}\sum_{j<i}\int_{\mathbb{R}^{d}}\left(  W\left(
	\mathbf{x}_{i},\mathbf{\bar{X}}_{j}^{n}\right)  +aV(\mathbf{x}_{i}%
	)+aV(\mathbf{\bar{X}}_{j}^{n})\right)  \bar{P}_{\{i\}|K_{j+1}}(d\mathbf{x}%
	_{i}|\mathbf{\bar{X}}_{k}^{n},k\in K_{j+1})+\frac{ac}{n}\\
	&  +\frac{1}{n(n-1)}\sum_{i<j}\mathcal{R}\left(  \bar{\mu}_{ij}^{n}%
	|e^{-(1-a)V}\ell\otimes e^{-(1-a)V}\ell\right)  +\frac{1}{n(n-1)}\sum
	_{j<i}\mathcal{R}\left(  \bar{\mu}_{ij}^{n}|e^{-(1-a)V}\ell\otimes
	e^{-(1-a)V}\ell\right)  \Bigg]\\
	&  =\mathbb{E}\Bigg[\frac{1}{2n^{2}}\sum_{i<j}\int_{\mathbb{R}^{d}%
		\times\mathbb{R}^{d}}\left(  W\left(  \mathbf{x}_{i},\mathbf{x}_{j}\right)
	+aV(\mathbf{x}_{i})+aV(\mathbf{x}_{j})\right)  \bar{P}_{\{i,j\}|K_{i}%
	}(d\mathbf{x}_{i}d\mathbf{x}_{j}|\mathbf{\bar{X}}_{k}^{n},k\in K_{i})\\
	&  +\frac{1}{2n^{2}}\sum_{j<i}\int_{\mathbb{R}^{d}\times\mathbb{R}^{d}}\left(
	W\left(  \mathbf{x}_{i},\mathbf{x}_{j}\right)  +aV(\mathbf{x}_{i}%
	)+aV(\mathbf{x}_{j})\right)  \bar{P}_{\{i,j\}|K_{j}}(d\mathbf{x}%
	_{j}d\mathbf{x}_{i}|\mathbf{\bar{X}}_{k}^{n},k\in K_{j})+\frac{ac}{n}\\
	&  +\frac{1}{n(n-1)}\sum_{i\neq j}\mathcal{R}\left(  \bar{\mu}_{ij}%
	^{n}|e^{-(1-a)V}\ell\otimes e^{-(1-a)V}\ell\right)  \Bigg]\\
	&  =\mathbb{E}\Bigg[\frac{1}{n^{2}}\sum_{i\neq j}\mathfrak{J}_{a}\left(
	\bar{\mu}_{ij}^{n}\right)  +\frac{ac}{n}+\frac{1}{n(n-1)}\sum_{i\neq
		j}\mathcal{R}\left(  \bar{\mu}_{ij}^{n}|e^{-(1-a)V}\ell\otimes e^{-(1-a)V}%
	\ell\right)  \Bigg],
	\end{split}
	\label{baca}%
	\end{equation}
	where $\mathfrak{J}_{a}$ is the functional defined in \eqref{scr} and $c$ is a
	lower bound for ${\mathcal{V}}$. Next, let
	\begin{equation}
	\hat{\mu}^{2,n}:=\frac{1}{n(n-1)}\sum_{i\neq j}\bar{\mu}_{ij}^{n}%
	.\label{madhater}%
	\end{equation}
	Then combining \eqref{baca} with the convexity of $\mathcal{R}$ in both
	arguments (see \cite[Lemma 1.4.3]{Dupuis}), the linearity of ${\mathfrak{J}%
		_{a}}$, and the definition of $\hat{\mu}^{2,n}$ in \eqref{madhater}, we
	obtain
	\begin{equation}%
	\begin{split}
	&  \mathbb{E}\left[  \mathcal{J}_{a,\neq}\left(  \bar{L}_{n}\right)  +\frac
	{a}{n}\mathcal{V}\left(  \bar{L}_{n}\right)  +\frac{1}{n}\sum_{i=1}%
	^{n}\mathcal{R}\left(  \bar{\mu}_{i}^{n}|e^{-(1-a)V}\ell\right)  \right]  \\
	&  \geq\mathbb{E}\left[  \frac{n-1}{n}{\mathfrak{J}_{a}}\left(  \hat{\mu
	}^{2,n}\right)  +\frac{ac}{n}+\mathcal{R}\left(  \hat{\mu}^{2,n}%
	|e^{-(1-a)V}\ell\otimes e^{-(1-a)V}\ell\right)  \right]  .
	\end{split}
	\label{idontknow2}%
	\end{equation}

	We now use (\ref{idontknow2}) to establish tightness of both $\{\bar{L}_{n}\}$
	and $\{\hat{\mu}_{n}\}$ in the $d_{\psi}$ topology. Note that $\hat{\mu}%
	^{2,n}$ is a random probability measure on $\mathbb{R}^{d}\times\mathbb{R}%
	^{d}$ and that it has identical marginals. Since $V$ and $W$ satisfy
	Assumption \ref{B} and relative entropy is nonnegative, there exists a
	superlinear function $\phi$ for which we have the inequalities
	\begin{equation}%
	\begin{split}
	&  \mathbb{E}\left[  \frac{n-1}{n}{\mathfrak{J}_{a}}\left(  \hat{\mu}%
	^{2,n}\right)  +\mathcal{R}\left(  \hat{\mu}^{2,n}|e^{-(1-a)V}\ell\otimes
	e^{-(1-a)V}\ell\right)  \right]  \\
	&  \geq\!\mathbb{E}\!\left[  \!\frac{n-1}{n}\Big[{\mathfrak{J}_{a}}\left(
	\hat{\mu}^{2,n}\right)  {+}\mathcal{R}\left(  \hat{\mu}^{2,n}|e^{-(1-a)V}%
	\ell\otimes e^{-(1-a)V}\ell\right)  \Big]\!\right]  {\geq}\frac{n-1}%
	{n}\mathbb{E}\left[  \int_{\mathbb{R}^{d}}\phi\left(  \psi\left(
	\boldsymbol{x}\right)  \right)  (\pi_{\#}^{1}\hat{\mu}^{2,n})\left(
	d\boldsymbol{x}\right)  \right]  .
	\end{split}
	\label{idontknow3}%
	\end{equation}
	For $n\geq2,$ combining \eqref{idontknow2} and \eqref{idontknow3} gives
	\begin{equation}
	2\mathbb{E}\left[  \mathcal{J}_{a,\neq}\left(  \bar{L}_{n}\right)  +\frac
	{a}{n}\mathcal{V}\left(  \bar{L}_{n}\right)  +\frac{1}{n}\sum_{i=1}%
	^{n}\mathcal{R}\left(  \bar{\mu}_{i}^{n}|e^{-(1-a)V}\ell\right)  \right]
	\geq\mathbb{E}\left[  \int_{\mathbb{R}^{d}}\phi\left(  \psi\left(
	\boldsymbol{x}\right)  \right)  (\pi_{\#}^{1}\hat{\mu}^{2,n})\left(
	d\boldsymbol{x}\right)  \right]  +\frac{2ac}{n}.\label{cbound}%
	\end{equation}

	Note that \eqref{pi-claim} implies $\mathbb{E}[\pi_{\#}^{1}\bar{\mu}_{ij}%
	^{n}]=\mathbb{E}[\bar{\mu}_{i}^{n}]$ and $\mathbb{E}[\pi_{\#}^{2}\bar{\mu
	}_{ij}^{n}]=\mathbb{E}[\bar{\mu}_{j}^{n}]$. Further, recalling the definition
	of $\hat{\mu}_{n}$ in \eqref{hat} and $\hat{\mu}^{2,n}$ in \eqref{madhater},
	this shows that
	\begin{equation}
	\mathbb{E}[\pi_{\#}^{1}\hat{\mu}^{2,n}]=\mathbb{E}[\pi_{\#}^{2}\hat{\mu}%
	^{2,n}]=\mathbb{E}[\hat{\mu}_{n}].\label{expeq}%
	\end{equation}
	Substituting this into the right-hand side of (\ref{cbound}) and letting
	$C_{0}<\infty$ denote the left-hand side of \eqref{controlingcontrols2}, we
	obtain the bound
	\[
	\mathbb{E}\left[  \int_{\mathbb{R}^{d}}\phi\left(  \psi\left(  \boldsymbol{x}%
	\right)  \right)  \hat{\mu}_{n}\left(  d\boldsymbol{x}\right)  \right]
	\leq2{C_{0}}-\frac{2ac}{n}\leq2C_{0}+1,
	\]
	for all sufficiently large $n$. However, since we know from Lemma
	\ref{tightnessfunction} that $\Phi(\mu)=\int_{\mathbb{R}^{d}}\phi\left(
	\psi\left(  \mathbf{x}\right)  \right)  \mu\left(  d\mathbf{x}\right)  $ is a
	tightness function on $\mathcal{P}_{\psi}\left(  \mathbb{R}^{d}\right)  $, it
	follows that $\{\hat{\mu}_{n}\}$ is tight as a collection of $\mathcal{P}%
	_{\psi}\left(  \mathbb{R}^{d}\right)  $-valued random elements. Finally, note
	that we have the equality%
	\[
	\mathbb{E}\left[  \int_{\mathbb{R}^{d}}g\left(  \mathbf{x}\right)  {\bar{L}%
	}_{n}\left(  d\mathbf{x}\right)  \right]  =\mathbb{E}\left[  \frac{1}{n}%
	\sum_{i=1}^{n}g\left(  \mathbf{\bar{X}}_{i}^{n}\right)  \right]
	=\mathbb{E}\left[  \frac{1}{n}\sum_{i=1}^{n}\int_{\mathbb{R}^{d}}g\left(
	\mathbf{x}\right)  \bar{\mu}_{i}^{n}\left(  d\mathbf{x}\right)  \right]
	=\mathbb{E}\left[  \int_{\mathbb{R}^{d}}g\left(  \mathbf{x}\right)  \hat{\mu
	}_{n}\left(  d\mathbf{x}\right)  \right]  .
	\]
	Setting $g(\mathbf{x})=\phi(\psi(\mathbf{x}))$, and again invoking Lemma
	\ref{tightnessfunction}, we see that $\{\bar{L}_{n}\}$ is also tight.
\end{proof}

\begin{remark}
	\emph{In the remainder of the proof, which is carried out in Sections \ref%
		{subs-lbth1} and \ref{Upn}, the arguments for both $\mathcal{P}(\mathbb{R}%
		^{d})$ and $\mathcal{P}_{\psi}(\mathbb{R}^{d})$ are similar, and so we will
		treat both cases simultaneously. The functions $f$ used will be considered
		continuous in the respective topology and any infimum taken should be with
		respect to the corresponding set $\mathcal{P}(\mathbb{R}^{d})$ or $\mathcal{P%
		}_{\psi}(\mathbb{R}^{d})$. }
\end{remark}

\begin{remark}
	\emph{\label{samesame} Due to Remark \ref{tightremark} and Lemmas \ref{aaa}
		and \ref{aaaa}, it is without loss of generality, for the lower bound %
		\eqref{th1-lb}, to restrict to controls for which $\left\{ \left( \bar{L}%
		_{n},\hat{\mu}_{n}\right) , n\in\mathbb{N}\right\} $ is tight in $\mathcal{P}%
		(\mathbb{R}^{d})\times\mathcal{P}(\mathbb{R}^{d}),$ or (with the additional
		Assumption \ref{B}) in $\mathcal{P}_{\psi}(\mathbb{R}^{d})\times\mathcal{P}%
		_{\psi}(\mathbb{R}^{d})$.}
\end{remark}

\subsection{Proof of the lower bound}

\label{subs-lbth1}

For the proof of the lower bound (\ref{th1-lb}) we will use some auxiliary
functionals. For $d^{\prime}\in\mathbb{N}$, an arbitrary function $F:\mathbb{%
	R}^{d^{\prime}}\rightarrow(-\infty,\infty]$ and $M\in\lbrack0,\infty)$, let $%
F^{M}(\mathbf{z}):=\min\{F(\mathbf{z}),M\}$. For $\mu\in\mathcal{P}(\mathbb{R%
}^{d}),$ let 
\begin{equation*}
\mathcal{J}_{a}^{M}\left( \mu\right) :=\frac{1}{2}\int_{\mathbb{R}^{d}\times%
	\mathbb{R}^{d}}\left(W^{M}\left( \mathbf{x},\mathbf{y}\right) +aV^{M}(%
\mathbf{x}) +aV^{M}(\mathbf{y})\right) \mu\left( d\mathbf{x}\right)
\mu\left( d\mathbf{y}\right) ,
\end{equation*}%
\begin{equation*}
\mathcal{J} _{a,\neq}^{M}\left( \mu\right) :=\frac{1}{2}\int_{(\mathbb{R}%
	^{d}\times\mathbb{R}^{d})_{\neq}}\left(W^{M}\left( \mathbf{x},\mathbf{y}%
\right) +aV^{M}(\mathbf{x}) +aV^{M}(\mathbf{y}) \right)\mu\left( d\mathbf{x}%
\right) \mu\left( d\mathbf{y}\right) ,
\end{equation*}
and note that for every $\mu\in\mathcal{P}(\mathbb{R}^{d}),$ 
\begin{equation}
\mathcal{J}_{a}^{M}\left( \mu\right) \leq\mathcal{J} _{a,\neq}^{M}\left(
\mu\right) +\frac{3M}{2}(\mu\otimes\mu)\{(x,x):x\in\mathbb{R}^{d}\}.
\label{W-WM}
\end{equation}

Let $\epsilon >0$ be given. Then by Remark \ref{tightremark} and the
boundedness of $f$, there exist $C^{\prime }\in \mathbb{R}$ and a sequence
of controls $\{\tilde{\mu}_{i}^{n}\}$ with associated sequence of controlled
empirical measures $\{\tilde{L}_{n}\}$, such that%
\begin{equation}
\begin{array}{rl}
C^{\prime } & >\inf_{\{\bar{\mu}_{i}^{n}\}}\mathbb{E}\left[ f\left( \bar{L}%
_{n}\right) +\mathcal{J}_{a,\neq }\left( \bar{L}_{n}\right) +\frac{a}{n}%
\mathcal{V}\left( \bar{L}_{n}\right) +\frac{1}{{n}}\sum_{i=1}^{n}\mathcal{R}%
\left( \bar{\mu}_{i}^{n}|e^{-(1-a)V}\ell \right) \right] +\epsilon \\ 
& \geq \mathbb{E}\left[ f\left( \tilde{L}_{n}\right) +\mathcal{J}_{a,\neq
}\left( \tilde{L}_{n}\right) +\frac{ac}{n}+\frac{1}{n}\sum_{i=1}^{n}\mathcal{%
	R}\left( \tilde{\mu}_{i}^{n}|e^{-(1-a)V}\ell \right) \right] \\ 
& \geq \mathbb{E}\left[ f\left( \tilde{L}_{n}\right) +\mathcal{J}_{a,\neq
}^{M}\left( \tilde{L}_{n}\right) +\frac{ac}{n}+\frac{1}{{n}}\sum_{i=1}^{n}%
\mathcal{R}\left( \tilde{\mu}_{i}^{n}|e^{-(1-a)V}\ell \right) \right] \\ 
& \geq \mathbb{E}\left[ f\left( \tilde{L}_{n}\right) +\mathcal{J}%
_{a}^{M}\left( \tilde{L}_{n}\right) -\frac{3M}{2n}+\frac{ac}{n}+\frac{1}{n}%
\sum_{i=1}^{n}\mathcal{R}\left( \tilde{\mu}_{i}^{n}|e^{-(1-a)V}\ell \right) %
\right] ,%
\end{array}
\label{th1-bd}
\end{equation}%
where $\mathcal{J}_{a,\neq }\geq \mathcal{J}_{a,\neq }^{M}$ is used for the
third inequality and the last inequality uses \eqref{W-WM} and the fact that 
$\bar{L}^{n}\otimes \bar{L}^{n}$ put mass at most $1/n$ on the diagonal of $%
\mathbb{R}^{d}\times \mathbb{R}^{d}$.

Let $\hat{\mu}_{n}:=\frac{1}{n}\sum_{i=1}^{n}\tilde{\mu}_{i}^{n}$. 
Since Lemma \ref{aaaa} implies $\{(\tilde{L}_{n},\hat{\mu}_{n})\}$ is tight,
we can extract a further subsequence, which we denote again by $\{(\tilde{L}%
_{n},\hat{\mu}_{n})\}$, which converges in distribution to some limit $(%
\tilde{L},\hat{\mu}).$ If the lower bound is established for this
subsequence, a standard argument by contradiction establishes the lower
bound for the original sequence. Let $\{M_{n}\}$ be an increasing sequence
such that $\lim_{n\rightarrow \infty }M_{n}=\infty $ and $\lim_{n\rightarrow
	\infty }\frac{M_{n}}{n}=0$, and let $m\in \mathbb{N}$. By the monotonicity
of $n\mapsto {\mathcal{W}}^{M_{n}}$, Jensen's inequality, the definition of $%
\hat{\mu}_{n}$, and Fatou's lemma we have 
\begin{equation}
\begin{split}
& \liminf_{n\rightarrow \infty }\mathbb{E}\left[ f\left( \tilde{L}%
_{n}\right) +\mathcal{J}_{a}^{M_{n}}\left( \tilde{L}_{n}\right) -\frac{3M_{n}%
}{2n}+\frac{ac}{n}+\frac{1}{n}\sum_{i=1}^{n}\mathcal{R}\left( \tilde{\mu}%
_{i}^{n}|e^{-(1-a)V}\ell \right) \right] \\
& \quad \geq \liminf_{n\rightarrow \infty }\mathbb{E}\left[ f\left( \tilde{L}%
_{n}\right) +\mathcal{J}_{a}^{M_{m}}\left( \tilde{L}_{n}\right) -\frac{3M_{n}%
}{2n}+\frac{ac}{n}+\mathcal{R}\left( \hat{\mu}_{n}|e^{-(1-a)V}\ell \right) %
\right] \\
& \quad \geq \mathbb{E}\left[ f\left( \tilde{L}\right) +\mathcal{J}%
_{a}^{M_{m}}\left( \tilde{L}\right) +\mathcal{R}\left( \hat{\mu}%
|e^{-(1-a)V}\ell \right) \right] ,
\end{split}
\label{idontknow4}
\end{equation}%
where the continuity of $f$ and lower semicontinuity of ${\mathcal{J}}%
_{a}^{M_{m}}$ and ${\mathcal{R}}(\cdot |e^{-(1-a))V}\ell )$ are also used in
the last inequality. Since this inequality holds for arbitrary $m\in \mathbb{%
	N}$, the monotone convergence theorem, the property that $\tilde{L}=\hat{\mu}
$ almost surely (due to Lemma \ref{lem:weak_limit_Sanov}) and the definition
of $\mathcal{I}$ in (\ref{def:Istar}), together imply 
\begin{align}
\lim_{m\rightarrow \infty }\mathbb{E}\left[ f\left( \tilde{L}\right) +%
\mathcal{J}_{a}^{M_{m}}\left( \tilde{L}\right) +\mathcal{R}\left( \hat{\mu}%
|e^{-(1-a)V}\ell \right) \right] & =\mathbb{E}\left[ f\left( \hat{\mu}%
\right) +\mathcal{J}_{a}\left( \hat{\mu}\right) +\mathcal{R}(\hat{\mu}%
|e^{-(1-a)V}\ell )\right]  \notag  \label{idontknow5} \\
& \geq \inf_{\mu \in S}\{f\left( \mu \right) +\mathcal{I}\left( \mu \right)
\}.
\end{align}%
Since $\epsilon >0$ is arbitrary, \eqref{th1-bd}, \eqref{idontknow4} and %
\eqref{idontknow5} together imply the lower bound \eqref{th1-lb}.


\subsection{Proof of the upper bound}

\label{Upn}

Again, fix $f$ to be a bounded continuous function on $\mathcal{P}(\mathbb{R}%
^{d})$, let $\epsilon>0$ and let $\mu^{\ast}\in\mathcal{P}(\mathbb{R}^{d})$
\thinspace(respectively, $\mathcal{P}_{\psi}(\mathbb{R}^{d})$) be such that 
\begin{equation}
f\left( \mu^{\ast}\right) +\mathcal{J}_{a}\left( \mu^{\ast}\right) +\mathcal{%
	R}\left( \mu^{\ast}|e^{-(1-a)V}\ell\right) \leq\inf_{\mu}\left[ f\left(
\mu\right) +\mathcal{I}(\mu)\right] +\epsilon.  \label{eps-opt}
\end{equation}
For $n\in\mathbb{N}$, let $\{\tilde{\mu}_{i}^{n},1\leq i\leq n\}$ denote the
particular control defined by $\tilde{\mu}_{i}^{n}:=\mu^{\ast}$ for all $n\in%
\mathbb{N}$ and $i\in\{1,...,n\},$ and let $\tilde{\mathbf{X}}_{i}^{n}, i =
1, \ldots, n,$ and $\tilde{L}_{n}$ denote the associated controlled objects.
Recall that $\ell$ and hence $\mu^{\ast}$ are non-atomic. From the
definition of $\mathcal{J}_{a}$ and $\mathcal{J} _{a,\neq}$ in \eqref{fnal-J}
and \eqref{Janeq}, respectively, we have 
\begin{align}
\mathbb{E}\left[ \mathcal{J} _{a,\neq}\left( \tilde{L}_{n}\right) \right] & =%
\frac{1}{2}\mathbb{E}\left[ \frac{1}{n^{2}}\sum_{i=1}^{n}\sum_{j=1, j\neq
	i}^{n}(W\left( \tilde{\mathbf{X}}_{i}^{n},\tilde{\mathbf{X}}_{j}^{n}\right)
+ aV(\tilde{\mathbf{X}}_{i}^{n}) + aV(\tilde{\mathbf{X}}_{j}^{n})) \right] 
\notag \\
& =\frac{n-1}{2n}\int_{\mathbb{R}^{d}\times\mathbb{R}^{d}}\left(W\left( 
\mathbf{x},\mathbf{y}\right) +aV(\mathbf{x}) +aV(\mathbf{y})\right) {\mu}%
^{\ast}\left( d\mathbf{x}\right) {\mu }^{\ast}\left( d\mathbf{y}\right) =%
\frac{n-1}{n}\mathcal{J}_{a}(\mu^{\ast}).  \label{Wcalc}
\end{align}

Define $\check{\mu}_{n}:=\frac{1}{n}\sum_{i=1}^{n}\tilde{\mu}_{i}^{n}=\mu
^{\ast }$. Then, due to \eqref{eps-opt}, the conditions of Lemma \ref{aaaa}
hold for $\{(\tilde{L}_{n},\check{\mu}_{n})\}$. Together with Lemma \ref{aaa}%
, this shows that $\{\tilde{L}_{n}\}$ is tight in ${\mathcal{P}}(\mathbb{R}%
^{d})$ and ${\mathcal{P}}_{\psi }(\mathbb{R}^{d})$. When combined with the
almost sure convergence $\tilde{L}_{n}\rightarrow \mu ^{\ast }$, which holds
due to Lemma \ref{lem:weak_limit_Sanov} (or the Glivenko-Cantelli lemma),
this implies convergence of $\tilde{L}_{n}$ to $\mu ^{\ast }$ with respect
to both $d_{w}$ and $d_{\psi }$, as appropriate. Since $f$ is bounded and
continuous, $\lim_{n\rightarrow \infty }\mathbb{E}[f(\tilde{L}_{n})]=f\left(
\mu ^{\ast }\right) $ by the dominated convergence theorem. The above
observations, together with \eqref{Wcalc}, the uniform lower bound on $%
\mathcal{J}_{a}$ and ${\mathcal{V}}$ and \eqref{eps-opt} show that 
\begin{align*}
& \limsup_{n\rightarrow \infty }\inf_{\{\bar{\mu}_{i}^{n}\}}\mathbb{E}\left[
f\left( {\bar{L}}_{n}\right) +\mathcal{J}_{a,\neq }\left( {\bar{L}}%
_{n}\right) +\frac{a}{n}{\mathcal{V}}(\bar{L}_{n})+\frac{1}{n}\sum_{i=1}^{n}%
\mathcal{R}\left( \bar{\mu}_{i}^{n}|e^{-(1-a)V}\ell \right) \right] \\
& \quad \leq \limsup_{n\rightarrow \infty }\mathbb{E}\left[ f\left( \tilde{L}%
_{n}\right) +\frac{n-1}{n}\mathcal{J}_{a}\left( \mu ^{\ast }\right) +\frac{a%
}{n}{\mathcal{V}}(\mu ^{\ast })+\frac{1}{n}\sum_{i=1}^{n}\mathcal{R}\left( 
\tilde{\mu}_{i}^{n}|e^{-(1-a)V}\ell \right) \right] \\
& \quad \leq f\left( \mu ^{\ast }\right) +\mathcal{J}_{a}(\mu ^{\ast })+%
\mathcal{R}\left( \mu ^{\ast }|e^{-(1-a)V}\ell \right) \leq \inf_{\mu \in
	S}\{f\left( \mu \right) +\mathcal{I}(\mu )\}+\epsilon .
\end{align*}%
Since $\epsilon $ is arbitrary, this implies the upper bound \eqref{th1-ub},
which together with \eqref{th1-lb} and the discussion at the end of Section %
\ref{subs-repth1} completes the proof of Theorem \ref{nspeed}.

\section{Proof of Theorem \protect\ref{biggernspeed}}


\label{sec:5}

This section is devoted to the proof of Theorem \ref{biggernspeed}. The
structure of the proof is similar to that of the case with speed $%
\beta_{n}=n.$ In view of Lemmas \ref{lsc} and \ref{lem-ratepsi} and Theorem
1.2.3 in \cite{Dupuis}, it suffices to prove that for any bounded and
continuous function $f$ on $S$ (where $S={\mathcal{P}}(\mathbb{R}^{d})$ or $%
S={\mathcal{P}}_{\psi}(\mathbb{R}^{d})$, as appropriate), as $n\rightarrow
\infty$, 
\begin{equation}
-\frac{1}{n}\log\mathbb{E}_{Q_{n}}\left[ e^{-\beta_{n}f}\right]
\rightarrow\inf_{\mu\in S}\left\{ f\left( \mu\right) +\mathcal{J}_{\star
}\left( \mu\right) \right\} .  \label{laplace-biggernspeed}
\end{equation}

\subsection{Representation formula}

As before, fix $a\in \lbrack 0,1)$ as in Assumption \ref{A}, and let $%
P_{n}^{\star }(d\mathbf{x}^{n})=e^{-(1-a)\sum_{i=1}^{n}V(x_{i})}\otimes
_{i=1}^{n}\ell (d\mathbf{x}_{i})$ be the probability measure on $\mathbb{R}%
^{dn}$ defined in \eqref{def-pnstar}. 
We now introduce the functional $\mathcal{J}_{n,\neq }:\mathcal{P}(\mathbb{R}%
^{d})\rightarrow (-\infty ,\infty ]$ given by 
\begin{equation}
\mathcal{J}_{n,\neq }\left( \mu \right) :=\frac{1}{2}\int_{(\mathbb{R}%
	^{d}\times \mathbb{R}^{d})_{\neq }}\left( \left( 1-\frac{(1-a)n}{\beta _{n}}%
\right) \left( V\left( \mathbf{x}\right) +V\left( \mathbf{y}\right) \right)
+W\left( \mathbf{x},\mathbf{y}\right) \right) \mu \left( d\mathbf{x}\right)
\mu \left( d\mathbf{y}\right) .  \label{idontknow6}
\end{equation}%
Note that $\mathcal{J}_{n,\neq }\left( \mu \right) $ is bounded below for
all sufficiently large $n$ due to Assumption \ref{C}\ref{C1} and the fact
that $\beta _{n}/n\rightarrow \infty $. When $\mathbf{x}^{n}\in (\mathbb{R}%
^{dn})_{\neq },$ using \eqref{idontknow6} we can rewrite $\beta _{n}H_{n},$
where $H_{n}$ was defined in \eqref{def-Hn}, as follows: 
\begin{equation*}
\begin{split}
\beta _{n}H_{n}(\mathbf{x}^{n})=& \frac{\beta _{n}}{2}\int_{(\mathbb{R}%
	^{d}\times \mathbb{R}^{d})_{\neq }}\left( \left( 1-\frac{(1-a)n}{\beta _{n}}%
\right) \left[ V\left( \mathbf{x}\right) +V\left( \mathbf{y}\right) \right]
+W\left( \mathbf{x},\mathbf{y}\right) \right) L_{n}\left( \mathbf{x}^{n};d%
\mathbf{x}\right) L_{n}\left( \mathbf{x}^{n};d\mathbf{y}\right) \\
& +\left( \frac{\beta _{n}}{n}-(1-a)\right) \int_{\mathbb{R}^{d}}V\left( 
\mathbf{x}\right) L_{n}\left( \mathbf{x}^{n};d\mathbf{x}\right) +(1-a)n\int_{%
	\mathbb{R}^{d}}V\left( \mathbf{x}\right) L_{n}\left( \mathbf{x}^{n};d\mathbf{%
	x}\right) \\
& =\beta _{n}\left( \mathcal{J}_{n,\neq }\left( L_{n}\left( \mathbf{x}%
^{n};\cdot \right) \right) +\left( \frac{1}{n}-\frac{(1-a)}{\beta _{n}}%
\right) \mathcal{V}\left( L_{n}\left( \mathbf{x}^{n};\cdot \right) \right)
\right) +(1-a)\sum_{i=1}^{n}V(\mathbf{x}_{i}).
\end{split}%
\end{equation*}%
Let $f$ be a measurable function on $\mathcal{P}(\mathbb{R}^{d})$ (or on $%
\mathcal{P}_{\psi }(\mathbb{R}^{d})$) that is bounded below (in particular $%
f $ could be bounded and continuous). Then by the definition of $P_{n}^{\ast
}$, we have%
\begin{equation}
-\frac{1}{\beta _{n}}\log \mathbb{E}_{Q_{n}}\left[ e^{-\beta _{n}f}\right] =-%
\frac{1}{\beta _{n}}\log \mathbb{E}_{P_{n}}\left[ e^{-\beta _{n}f\circ L_{n}}%
\right] =-\frac{1}{\beta _{n}}\log \mathbb{E}_{P_{n}^{\star }}\left[ \frac{1%
}{Z_{n}}e^{-\beta _{n}\left( f+\mathcal{J}_{n,\neq }+\left( \frac{1}{n}-%
	\frac{(1-a)}{\beta _{n}}\right) \mathcal{V}\right) \circ L_{n}}\right] \!\!,
\label{biggernrep}
\end{equation}%
where $Z_{n}$ is the normalization constant defined in \eqref{Zndef}.

Using the same notation and arguments as in Section \ref{subs-repth1}, the
following representations are valid. Fix any function $f$ on $\mathcal{P}(%
\mathbb{R}^{d})$ (or $\mathcal{P}_{\psi}(\mathbb{R}^{d})$), such that $%
f\circ L_{n}$ is measurable in $\mathbb{R}^{dn}$ and bounded from below
(this includes all continuous and bounded functions on $\mathcal{P}(\mathbb{R%
}^{d})$ or $\mathcal{P}_{\psi}(\mathbb{R}^{d})$. Then, since the function $(%
\mathbf{x}, \mathbf{y})$ $\mapsto$ $\left( 1-\frac{(1-a)n}{\beta_{n}}\right)
V\left( \mathbf{x}\right) +\left( 1-\frac{(1-a)n}{\beta_{n}}\right) V\left( 
\mathbf{y}\right) +W\left( \mathbf{x},\mathbf{y}\right)$ is measurable and
bounded from below, we can apply \cite[Proposition 4.5.1]{Dupuis} to $%
f(L_{n}(\boldsymbol{x}^{n};\cdot))+\mathcal{J}_{n,\neq }(L_{n}(\boldsymbol{x}%
^{n};\cdot))+\left(\frac{1}{n}-\frac{1-a}{\beta_{n}}\right)\mathcal{V}(L_{n}(%
\boldsymbol{x}^{n};\cdot)),$ to obtain 
\begin{equation}
\begin{split}
&-\frac{1}{\beta_{n}}\log\mathbb{E}_{P_{n}^{\star}}\left[ e^{-\beta_{n}%
	\left( f+\mathcal{J}_{n,\neq} + \left(\frac{1}{n}-\frac{1-a}{\beta_{n}}%
	\right)\mathcal{V}\right)\circ L_{n} }\right] \\
& =\inf_{\{\bar{\mu}_{i}^{n}\}}\mathbb{E}\left[ f\left( \bar{L}_{n}\right) +%
\mathcal{J}_{n,\neq}\left( \bar{L}_{n}\right) +\left(\frac{1}{n}-\frac{1-a}{%
	\beta_{n}}\right)\mathcal{V}(\bar{L}_{n}) +\frac{1}{\mathcal{\beta}_{n}}\sum
_{i=1}^{n}\mathcal{R}\left( \bar{\mu}_{i}^{n}|e^{-(1-a)V}\ell\right) \right]
.
\end{split}
\label{snd_rep}
\end{equation}
Setting $f=0$ in the last display, we have 
\begin{equation}
\begin{split}
&-\frac{1}{\beta_{n}}\log\left( Z_{n}\right) = -\frac{1}{\beta_{n}}\log%
\mathbb{E}_{P_{n}^{\star}}\left[ e^{-\beta_{n}\left(\mathcal{J}_{n,\neq}
	+\left(\frac{1}{n}-\frac{1-a}{\beta_{n}}\right)\mathcal{V}\right)\circ L_{n}}%
\right] \\
& = \inf_{\{\bar{\mu}_{i}^{n}\}}\mathbb{E}\left[ \mathcal{J}_{n,\neq}\left( 
\bar{L}_{n}\right) + \left(\frac{1}{n}-\frac{1-a}{\beta_{n}}\right)\mathcal{V%
}(\bar{L}_{n}) +\frac{1}{\beta_{n}}\sum_{i=1}^{n}\mathcal{R}\left( \bar{\mu}%
_{i}^{n}|e^{-(1-a)V}\ell\right) \right] .  \label{def-zn2}
\end{split}%
\end{equation}

As before, to establish Theorem \ref{biggernspeed}, in view of %
\eqref{snd_rep}, \eqref{def-zn2} and \eqref{biggernrep}, it suffices to
establish the lower bound 
\begin{equation}
\begin{split}
& \liminf_{n\rightarrow \infty }\inf_{\{\bar{\mu}_{i}^{n}\}}\mathbb{E}\left[
f\left( \bar{L}_{n}\right) +\mathcal{J}_{n,\neq }\left( \bar{L}_{n}\right)
+\left( \frac{1}{n}-\frac{1-a}{\beta _{n}}\right) \mathcal{V}(\bar{L}_{n})+%
\frac{1}{\beta _{n}}\sum_{i=1}^{n}\mathcal{R}\left( \bar{\mu}%
_{i}^{n}|e^{-(1-a)V}\ell \right) \right] \\
& \geq \inf_{\mu \in S}\{f\left( \mu \right) +\mathcal{J}\left( \mu \right)
\},
\end{split}
\label{prom}
\end{equation}%
and the upper bound 
\begin{equation}
\begin{split}
& \limsup_{n\rightarrow \infty }\inf_{\{\bar{\mu}_{i}^{n}\}}\mathbb{E}\left[
f\left( \bar{L}_{n}\right) +\mathcal{J}_{n,\neq }\left( \bar{L}_{n}\right)
+\left( \frac{1}{n}-\frac{1-a}{\beta _{n}}\right) \mathcal{V}(\bar{L}_{n})+%
\frac{1}{\beta _{n}}\sum_{i=1}^{n}\mathcal{R}\left( \bar{\mu}%
_{i}^{n}|e^{-(1-a)V}\ell \right) \right] \\
& \leq \inf_{\mu \in S}\{f\left( \mu \right) +\mathcal{J}\left( \mu \right)
\},
\end{split}
\label{prom2}
\end{equation}%
for all bounded and continuous functions $f$ (with respect to the
corresponding topologies).

\subsection{Tightness of controls}

\label{subs-tighcontr2}

As in Remark \ref{tightremark}, we have the following observation that
simplifies the proof of the lower bound.

\begin{remark}
	\label{tightrema} \emph{\ Without loss of generality we can assume that 
		\begin{equation}
		\sup_{n\in \mathbb{N}}\inf_{\{\bar{\mu}_{i}^{n}\}}\mathbb{E}\left[ J_{n,\neq
		}\left( \bar{L}_{n}\right) +\left( \frac{1}{n}-\frac{1-a}{\beta _{n}}\right) 
		\mathcal{V}(\bar{L}_{n})+\frac{1}{\mathcal{\beta }_{n}}\sum_{i=1}^{n}%
		\mathcal{R}\left( \bar{\mu}_{i}^{n}|e^{-(1-a)V}\ell \right) \right] <\infty .
		\label{idontknow}
		\end{equation}
	}
\end{remark}

\begin{lemma}
	\label{controlingcontrols4} Let $\{\bar{\mu}_{i}^{n}\}$ be a sequence of
	controls such that the associated controlled empirical measures satisfy %
	\eqref{idontknow}. Assume also that $V$ and $W$ satisfy Assumptions \ref{A}
	and \ref{C}\ref{C1}. Then $\{\bar{L}_{n}\}$ is tight in $\mathcal{P}\left(%
	\mathbb{R}^{d}\right)$. Further, if Assumption $\ref{C}\ref{C3}$ is also
	satisfied for some $\psi \in \Psi$, then $\{\bar{L}_{n}\}$ is tight on $%
	\mathcal{P}_{\psi}\left(\mathbb{R}^{d}\right).$
\end{lemma}

\begin{proof}
	First, note that by \eqref{idontknow6}, $\mathcal{J}_{n,\neq}(\bar{L}_{n})$
	can be rewritten as
	\begin{align*}
	\frac{1}{2}\mathcal{J}_{\neq}(\bar{L}_{n}) + \frac{1}{4}\int_{(\mathbb{R}%
		^{d}\times\mathbb{R}^{d})_{\neq}}\left(  \left(  1-\frac{2(1-a)n}{\beta_{n}%
	}\right)  \left(  V\left(  \mathbf{x}\right)  + V\left(  \mathbf{y}\right)
	\right)  +W\left(  \mathbf{x},\mathbf{y}\right)  \right)  \bar{L}_{n}\left(
	d\mathbf{x}\right) \bar{L}_{n}\left(  d\mathbf{y}\right)
	\end{align*}
	For large enough $n,$ Assumption \ref{A} implies that all the integrands  are
	bounded from below. Therefore, we have
	\begin{equation}
	\label{Restimate}\sup_{n\in\mathbb{N}}\mathbb{E}\left[  \mathcal{J}_{\neq
	}(\bar{L}_{n}) \right] <\infty\qquad\mbox{ and } \qquad\sup_{n\in\mathbb{N}%
	}\mathbb{E}\left[ \frac{1}{\mathcal{\beta}_{n}}\sum_{i=1}^{n}\mathcal{R}%
	\left(  \bar{\mu}_{i}^{n}|e^{-(1-a)V}\ell\right)  \right] <\infty.
	\end{equation}

	Now, let the set $A$, sequence $\{r_{n}\}$ and lsc function $\gamma
	:\mathbb{R}^{d}\mapsto\mathbb{R}$ be as in Assumption \ref{C}{1}, and let
	$A_{n}^{1}$ and $A_{n}^{2}$ be the associated sets defined therein. Then the
	integral $\mathcal{J}_{\neq}(L_{n})$ can be decomposed as the sum of integrals
	over the following sets:
	\[
	A_{n}^{1}\times A_{n}^{2},\,(A_{n}^{1}\times A_{n}^{1})_{\neq},\,(A_{n}%
	^{2}\times A_{n}^{2})_{\neq},A_{n}^{2}\times A_{n}^{1},\,(A^{c}\times
	A^{c})_{\neq},\,A\times A^{c},\,A^{c}\times A.
	\]
	Each of these terms is bounded from below by direct application of Assumption
	\ref{A}, and since $\mathbb{E}[\mathcal{J}_{\neq}(L_{n})]$ is uniformly
	bounded from above we have that the expectation of each of them is also
	bounded from above.\newline
	
	To prove that $\{\bar{L}_{n}\}$ is tight, by Lemmas \ref{tightness} and
	\ref{ayid} it suffices to prove that $\mathbb{E}\left[  \int_{\mathbb{R}^{d}%
	}\gamma(\mathbf{x})\bar{L}_{n}(d\mathbf{x})\right]  $ is uniformly bounded in
	$n$.  Note that
	\begin{equation}%
	\mathbb{E}\left[  \int_{\mathbb{R}^{d}}\gamma(\mathbf{x})\bar{L}%
	_{n}(d\mathbf{x})\right]  =\mathbb{E}\left[  \int_{A^{c}}\gamma(\mathbf{x}%
	)\bar{L}_{n}(d\mathbf{x})\right]  +\mathbb{E}\left[  \int_{A_{n}^{1}}%
	\gamma(\mathbf{x})\bar{L}_{n}(d\mathbf{x})\right]  +\mathbb{E}\left[
	\int_{A_{n}^{2}}\gamma(\mathbf{x})\bar{L}_{n}(d\mathbf{x})\right]  .
	\label{bds}%
	\end{equation}
	We now show that each of the three terms in the last inequality is uniformly
	bounded from above. By applying \eqref{C11} of Assumption \ref{C}\ref{C1}, we
	obtain for $n\geq2$,
	\begin{align*}
	\mathbb{E}\left[  \int_{A^{c}}\gamma(\mathbf{x})\bar{L}_{n}(d\mathbf{x}%
	)\right]    & =\frac{1}{2}\mathbb{E}\left[  \int_{\mathbb{R}^{d}%
		\times\mathbb{R}^{d}}\left(  \gamma(\mathbf{x})I_{A^{c}}(\mathbf{x}%
	)+\gamma(\mathbf{y})I_{A^{c}}(\mathbf{y})\right)  \bar{L}_{n}(d\mathbf{x}%
	)\bar{L}_{n}(d\mathbf{y})\right]  \\
	& =\frac{n}{2(n-1)}\mathbb{E}\left[  \int_{(\mathbb{R}^{d}\times\mathbb{R}%
		^{d})_{\neq}}\left(  \gamma(\mathbf{x})I_{A^{c}}(\mathbf{x})+\gamma
	(\mathbf{y})I_{A^{c}}(\mathbf{y})\right)  \bar{L}_{n}(d\mathbf{x})\bar{L}%
	_{n}(d\mathbf{y})\right]  \\
	& \leq\frac{n}{2(n-1)}\mathbb{E}\left[  \int_{(\mathbb{R}^{d}\times
		\mathbb{R}^{d})_{\neq}}\left(  V\left(  \mathbf{x}\right)  +V\left(
	\mathbf{y}\right)  +W\left(  \mathbf{x},\mathbf{y}\right)  \right)  \bar
	{L}_{n}(d\mathbf{x})\bar{L}_{n}(d\mathbf{y})\right]  +C\\
	& \leq\mathbb{E}\left[  {\mathcal{J}}_{\neq}(\bar{L}_{n})\right]  +C+|c|.
	\end{align*}
	By Assumption \ref{C}\ref{C1} we also have
	\begin{align*}
	\mathbb{E}\left[  \int_{A_{n}^{1}}\gamma(\mathbf{x})\bar{L}_{n}(d\mathbf{x}%
	)\right]    & \leq\mathbb{E}\left[  \int_{\left(  A_{n}^{1}\times A_{n}%
		^{1}\right)  _{\neq}}\left(  V\left(  \mathbf{x}\right)  +V\left(
	\mathbf{y}\right)  +W\left(  \mathbf{x},\mathbf{y}\right)  \right)  \bar
	{L}_{n}\left(  d\mathbf{x}\right)  \bar{L}_{n}\left(  d\mathbf{y}\right)
	\right]  +C\\
	& \leq\mathbb{E}[{\mathcal{J}}_{\neq}((\bar{L}_{n})]+C+|c|.
	\end{align*}
	Due to \eqref{Restimate}, the last two displays show that the first two terms
	on the right-hand side of \eqref{bds} are uniformly bounded. Finally for the
	third term, since $\hat{\mu}^{n}=\frac{1}{n}\sum_{i=1}^{n}\bar{\mu}_{i}^{n},$
	recalling (from Section \ref{subs-repth1}) that $\{\bar{\mathbf{X}}_{j}%
	^{n}\}_{1\leq j\leq n}$ are the controlled random variables with joint
	distribution $\bar{P}^{n}(dx_{1},\ldots,dx_{n})$, and using the tower property
	of conditional expectations, we have
	\begin{align*}
	\mathbb{E}\left[  \int_{A_{n}^{2}}\gamma(\mathbf{x})\bar{L}_{n}(d\mathbf{x}%
	)\right]  =\mathbb{E}\left[  \frac{1}{n}\sum_{i=1}^{n}\int_{A_{n}^{2}}%
	\gamma(\mathbf{x})\delta_{X_{i}}(d\mathbf{x})\right]    & =\mathbb{E}\left[
	\frac{1}{n}\sum_{i=1}^{n}\int_{A_{n}^{2}}\gamma(\mathbf{x})\bar{\mu}_{i}%
	^{n}(d\mathbf{x})\right]  \\
	& =\mathbb{E}\left[  \int_{A_{n}^{2}}\gamma(\mathbf{x})\hat{\mu}%
	^{n}(d\mathbf{x})\right]  .\\
	&
	\end{align*}
	Recalling that $I_{A_{n}^{2}}$ denotes the indicator function of the set
	$A_{n}^{2}$, by an extension of the formula that relates exponential integrals
	and relative entropy \cite[Proposition 4.5.1]{Dupuis}, the bound \eqref{C13}
	in Assumption \ref{C}\ref{C1} (resp. Assumption \ref{C}\ref{C3}), and
	\eqref{Restimate}, we see that
	\begin{align*}
	\mathbb{E}\left[  \int_{A_{n}^{2}}\gamma(\mathbf{x})\hat{\mu}^{n}%
	(d\mathbf{x})\right]    & \leq\frac{n}{\beta_{n}}\log\left(  \mathbb{E}\left[
	\int_{\mathbb{R}^{d}}e^{\frac{\beta_{n}}{n}I_{A_{n}^{2}}(\mathbf{x}%
		)\gamma(\mathbf{x})-(1-a)V\left(  \mathbf{x}\right)  }\ell\left(
	d\mathbf{x}\right)  \right]  \right)  +\frac{n}{\beta_{n}}\mathbb{E}\left[
	\mathcal{R}\left(  \hat{\mu}^{n}|e^{-(1-a)V}\ell\right)  \right]  \\
	& \leq\frac{n}{\beta_{n}}\log\left(  \int_{A_{n}^{2}}e^{\frac{\beta_{n}}%
		{n}\gamma(\mathbf{x})-(1-a)V\left(  \mathbf{x}\right)  }\ell\left(
	d\mathbf{x}\right)  +1\right)  +\frac{1}{\beta_{n}}\sum_{i=1}^{n}%
	\mathbb{E}\left[  \mathcal{R}\left(  \bar{\mu}_{i}^{n}|e^{-(1-a)V}\ell\right)
	\right]  .
	\end{align*}
	The first term on the right-hand side is uniformly bounded by \eqref{C13} in
	Assumption \ref{C}\ref{C1} (resp. Assumption \ref{C}\ref{C2}) and the second
	term is uniformly bounded by \eqref{Restimate}. This concludes the proof.
	
\end{proof}

\subsection{Proof of the lower bound}

\label{subs-bignlb}

For the proof of the lower bound we use some auxiliary functionals on $%
\mathcal{P}(\mathbb{R}^{d})$. For a function $F$ on $\mathbb{R}^{d^{\prime
}} $ and $M<\infty $ we define $F^{M}(\mathbf{z}):=\min \{F(\mathbf{z}),M\}$%
. Let 
\begin{equation*}
\mathcal{J}_{n,\neq }^{M}\left( \mu \right) :=\frac{1}{2}\int_{(\mathbb{R}%
	^{d}\times \mathbb{R}^{d})_{\neq }}\left( \left( 1-\frac{(1-a)n}{\beta _{n}}%
\right) \left[ V^{M}\left( \mathbf{x}\right) +V^{M}\left( \mathbf{y}\right) %
\right] +W^{M}\left( \mathbf{x},\mathbf{y}\right) \right) \bar{L}_{n}\left( d%
\mathbf{x}\right) \bar{L}_{n}\left( d\mathbf{y}\right) ,
\end{equation*}%
\begin{eqnarray*}
	\mathcal{J}_{n}^{M}\left( \mu \right) := &&\frac{1}{2}\int_{\mathbb{R}%
		^{d}\times \mathbb{R}^{d}}\left( \left( 1-\frac{(1-a)n}{\beta _{n}}\right) %
	\left[ V^{M}\left( \mathbf{x}\right) +V^{M}\left( \mathbf{y}\right) \right]
	+W^{M}\left( \mathbf{x},\mathbf{y}\right) \right) \bar{L}_{n}\left( d\mathbf{%
		x}\right) \bar{L}_{n}\left( d\mathbf{y}\right) , \\
	\mathcal{J}^{M}\left( \mu \right) := &&\frac{1}{2}\int_{\mathbb{R}^{d}\times 
		\mathbb{R}^{d}}\left( V^{M}\left( \mathbf{x}\right) +V^{M}\left( \mathbf{y}%
	\right) +W^{M}\left( \mathbf{x},\mathbf{y}\right) \right) \mu \left( d%
	\mathbf{x}\right) \mu \left( d\mathbf{y}\right) .
\end{eqnarray*}%
%
%
%
%
%
%
%
These integrals are well defined for sufficiently large $n$ because of
Assumption \ref{A}. For every $M,n\in \mathbb{N}$, 
\begin{align}
& \inf_{\{\bar{\mu}_{i}^{n}\}}\mathbb{E}\left[ f\left( \bar{L}_{n}\right) +%
\mathcal{J}_{n,\neq }\left( \bar{L}_{n}\right) +\left( \frac{1}{n}-\frac{1-a%
}{\beta _{n}}\right) \mathcal{V}(\bar{L}_{n})+\frac{1}{\beta _{n}}%
\sum_{i=1}^{n}\mathcal{R}\left( \bar{\mu}_{i}^{n}|e^{-(1-a)V}\ell \right) %
\right]  \label{eqn:scnd-pf} \\
& \quad \quad \geq \inf_{\{\bar{\mu}_{i}^{n}\}}\mathbb{E}\left[ f\left( \bar{%
	L}_{n}\right) +\mathcal{J}_{n,\neq }^{M}\left( \bar{L}_{n}\right) +\left( 
\frac{1}{n}-\frac{1-a}{\beta _{n}}\right) \mathcal{V}(\bar{L}_{n})+\frac{1}{%
	\beta _{n}}\sum_{i=1}^{n}\mathcal{R}\left( \bar{\mu}_{i}^{n}|e^{-(1-a)V}\ell
\right) \right] .  \notag
\end{align}%
Let $\epsilon >0$ and $\{\tilde{\mu}_{i}^{n}\}$ be such that%
\begin{align*}
C^{\prime }& >\inf_{\{\bar{\mu}_{i}^{n}\}}\mathbb{E}\left[ f\left( \bar{L}%
_{n}\right) +\mathcal{J}_{n,\neq }\left( \bar{L}_{n}\right) +\left( \frac{1}{%
	n}-\frac{1-a}{\beta _{n}}\right) \mathcal{V}(\bar{L}_{n})+\frac{1}{\beta _{n}%
}\sum_{i=1}^{n}\mathcal{R}\left( \bar{\mu}_{i}^{n}|e^{-V}\ell \right) \right]
+\epsilon \\
& \geq \mathbb{E}\left[ f\left( \tilde{L}_{n}\right) +\mathcal{J}_{n,\neq
}\left( \tilde{L}_{n}\right) +\left( \frac{1}{n}-\frac{1-a}{\beta _{n}}%
\right) \mathcal{V}(\tilde{L}_{n})+\frac{1}{\beta _{n}}\sum_{i=1}^{n}%
\mathcal{R}\left( \tilde{\mu}_{i}^{n}|e^{-(1-a)V}\ell \right) \right] \\
& \geq \mathbb{E}\left[ f\left( \tilde{L}_{n}\right) +\mathcal{J}%
_{n}^{M}\left( \tilde{L}_{n}\right) -\frac{3M}{n}+\left( \frac{1}{n}-\frac{%
	1-a}{\beta _{n}}\right) \mathcal{V}(\tilde{L}_{n})+\frac{1}{\beta _{n}}%
\sum_{i=1}^{n}\mathcal{R}\left( \tilde{\mu}_{i}^{n}|e^{-(1-a)V}\ell \right) %
\right] ,
\end{align*}%
where $C^{\prime }$ is a finite upper bound, which exists by Remark \ref%
{tightrema} and the boundedness of $f$, and the last inequality follows from
(\ref{eqn:scnd-pf}) and the fact $\tilde{L}_{n}\left( d\mathbf{x}\right) 
\tilde{L}_{n}\left( d\mathbf{y}\right) $ puts mass $1/n$ on the diagonal $%
\mathbf{x=y}$.

Owing to tightness (see Lemma \ref{controlingcontrols4}) we can extract a
further subsequence of $\{(\tilde{L}_{n},\hat{\mu}_{n})\}$, which (with some
abuse of notation) we denote again by $\{(\tilde{L}_{n},\hat{\mu}_{n})\}$,
for which $\hat{\mu}_{n}:=\frac{1}{n}\sum_{i=1}^{n}\tilde{\mu}_{i}^{n}$,
that converges weakly to some limit $(\tilde{L},\hat{\mu}).$ 
Let $M_{n}$ be a sequence that goes to infinity such that $%
\lim_{n\rightarrow \infty }\frac{M_{n}}{n}=0$ and let $m\in \mathbb{N}$. By
Fatou's lemma, the nonnegativity of ${\mathcal{R}}(\cdot |e^{-V})$, the
definition of $\mathcal{V}$ in (\ref{Vdef}), and the fact that $n/\beta
_{n}\rightarrow 0$, we have 
\begin{align*}
& \liminf_{n\rightarrow \infty }\mathbb{E}\left[ f\left( \tilde{L}%
_{n}\right) +\mathcal{J}_{n}^{M_{n}}\left( \tilde{L}_{n}\right) -\frac{3M_{n}%
}{n}+\left( \frac{1}{n}-\frac{1-a}{\beta _{n}}\right) \mathcal{V}(\bar{L}%
_{n})+\frac{1}{\beta _{n}}\sum_{i=1}^{n}\mathcal{R}\left( \tilde{\mu}%
_{i}^{n}|e^{-(1-a)V}\ell \right) \right] \\
& \quad \geq \liminf_{n\rightarrow \infty }\mathbb{E}\left[ f\left( \tilde{L}%
_{n}\right) +\mathcal{J}_{n}^{M_{m}}\left( \tilde{L}_{n}\right) \right] =%
\mathbb{E}\left[ f\left( \tilde{L}\right) +\mathcal{J}^{M_{m}}\left( \tilde{L%
}\right) \right] .
\end{align*}%
Since the above inequality holds for arbitrary $m,$ using the monotone
convergence theorem 
\begin{align*}
& \liminf_{n\rightarrow \infty }\mathbb{E}\left[ f\left( \tilde{L}%
_{n}\right) +\mathcal{J}_{n}^{M_{n}}\left( \tilde{L}_{n}\right) -\frac{3M_{n}%
}{n}+\left( \frac{1}{n}-\frac{1-a}{\beta _{n}}\right) \mathcal{V}(\bar{L}%
_{n})+\frac{1}{\beta _{n}}\sum_{i=1}^{n}\mathcal{R}\left( \tilde{\mu}%
_{i}^{n}|e^{-(1-a)V}\ell \right) \right] \\
& \hspace{20pt}\geq \mathbb{E}\left[ f\left( \tilde{L}\right) +\mathcal{J}%
\left( \tilde{L}\right) \right] \geq \inf_{\mu \in S}\{f\left( \mu \right) +%
\mathcal{J}\left( \mu \right) \}.
\end{align*}%
Since $\epsilon >0$ is arbitrary, this establishes \eqref{prom}.

\subsection{Proof of the upper bound}

\label{Upn2}

We start by making an observation, whose proof is deferred to Appendix \ref%
{ap-existenceex}.

\begin{lemma}
	\label{existanceex} Suppose Assumptions \ref{A} and \ref{C} hold, and let $%
	a\in \lbrack 0,1)$ be the associated constant. Given any $\mu \in {\mathcal{P%
	}}(\mathbb{R}^{d})$, there exists a sequence $\{\mu _{n}\}$ with each $\mu
	_{n}\ll \ell $ such that the density $\rho _{n}:=e^{(1-a)V}d\mu _{n}/d\ell $
	is uniformly bounded, $\mu _{n}\overset{w}{\rightarrow }\mu $ and ${\mathcal{%
			J}}(\mu _{n})\rightarrow {\mathcal{J}}(\mu )$. Furthermore, if $\mu \in {%
		\mathcal{P}}_{\psi }(\mathbb{R}^{d})$ for some $\psi \in \Psi $, then we can
	assume in addition that $d_{\psi }(\mu _{n},\mu )\rightarrow 0$.
\end{lemma}


Now, let $f$ be a bounded and continuous function on ${\mathcal{P}}(\mathbb{R%
}^{d})$ (or ${\mathcal{P}}_{\psi }(\mathbb{R}^{d})$), let $\epsilon >0$ and
let $\mu ^{\ast }$ be such that 
\begin{equation*}
f\left( \mu ^{\ast }\right) +\mathcal{J}\left( \mu ^{\ast }\right) \leq
\inf_{\mu \in S}\{f\left( \mu \right) +\mathcal{J}\left( \mu \right)
\}+\epsilon .
\end{equation*}%
We can also assume that $\mathcal{R}\left( \mu ^{\ast }|e^{-(1-a)V}\ell
\right) <\infty ,$ due to Assumption \ref{C}\ref{C2} and Lemma \ref%
{existanceex}. Then let $\tilde{\mu}_{i}^{n}=\mu ^{\ast }$ for all $n\in 
\mathbb{N}$ and $i\in \{1,...,n\},$ and let the random variables $\tilde{%
	\mathbf{X}}_{i}^{n},1\leq i\leq n,n\in \mathbb{N},$ be iid with distribution 
$\mu ^{\ast }.$ By Lemma \ref{lem:weak_limit_Sanov}, the weak limit of $%
\tilde{L}_{n}$ equals $\mu ^{\ast }.$ Calculations very similar to those of (%
\ref{Wcalc}) yield 
\begin{equation*}
\begin{split}
& \mathbb{E}\left[ f\left( \tilde{L}_{n}\right) +\mathcal{J}_{n,\neq }\left( 
\tilde{L}_{n}\right) +\left( \frac{1}{n}-\frac{1-a}{\beta _{n}}\right) 
\mathcal{V}(\bar{L}_{n})+\frac{1}{\beta _{n}}\sum_{i=1}^{n}\mathcal{R}\left(
\mu ^{\ast }|e^{-(1-a)V}\ell \right) \right] \\
& \quad =\mathbb{E}\left[ f\left( \tilde{L}_{n}\right) \right] +\frac{n-1}{n}%
\mathcal{J}(\mu ^{\ast })+\left( \frac{1}{n}-\frac{1-a}{\beta _{n}}\right) 
\mathcal{V}(\mu ^{\ast })+\frac{n}{\beta _{n}}\mathcal{R}\left( \mu ^{\ast
}|e^{-(1-a)V}\ell \right) .
\end{split}%
\end{equation*}%
Thus, $\tilde{L}_{n}\overset{w}{\rightarrow }\mu $, the dominated
convergence theorem and the fact that $n/\beta _{n}\rightarrow 0$ imply 
\begin{equation*}
\limsup_{n\rightarrow \infty }\left( \mathbb{E}\left[ f\left( \tilde{L}%
_{n}\right) \right] +\frac{n}{n-1}\mathcal{J}(\mu ^{\ast })+\left( \frac{1}{n%
}-\frac{1-a}{\beta _{n}}\right) \mathcal{V}(\mu ^{\ast })+\frac{n}{\beta _{n}%
}\mathcal{R}\left( \mu ^{\ast }|e^{-(1-a)V}\ell \right) \right)
\end{equation*}%
is equal to $f\left( \mu ^{\ast }\right) +\mathcal{J}\left( \mu ^{\ast
}\right) $. Thus, we have shown that 
\begin{align*}
& \limsup_{n\rightarrow \infty }\inf_{\{\bar{\mu}_{i}^{n}\}}\mathbb{E}\left[
f\left( \bar{L}_{n}\right) +\mathcal{J}_{n,\neq }\left( \bar{L}_{n}\right)
+\left( \frac{1}{n}-\frac{1-a}{\beta _{n}}\right) \mathcal{V}(\bar{L}_{n})+%
\frac{1}{\beta _{n}}\sum_{i=1}^{n}\mathcal{R}\left( \bar{\mu}%
_{i}^{n}|e^{-(1-a)V}\ell \right) \right] \\
& \quad \leq \limsup_{n\rightarrow \infty }\mathbb{E}\left[ f\left( \tilde{L}%
_{n}\right) +\mathcal{J}_{n,\neq }\left( \tilde{L}_{n}\right) +\left( \frac{1%
}{n}-\frac{1-a}{\beta _{n}}\right) \mathcal{V}(\tilde{L}_{n})+\frac{1}{\beta
	_{n}}\sum_{i=1}^{n}\mathcal{R}\left( \mu ^{\ast }|e^{-(1-a)V}\ell \right) %
\right] \\
& \quad =f\left( \mu ^{\ast }\right) +\mathcal{J}\left( \mu ^{\ast }\right)
\\
& \quad \leq \inf_{\mu \in S}\{f\left( \mu \right) +\mathcal{J}\left( \mu
\right) \}+\epsilon .
\end{align*}%
Since $\epsilon >0$ is arbitrary, we obtain the upper bound \eqref{prom2},
thus completing the proof of Theorem \ref{biggernspeed}. \\

\noindent 
{ 
\noindent
    {\bf Acknowledgments. }  We would like to thank a referee of the first
    version of this article \cite{DupLasRam15v1} for pointing out a
    small technical error in Lemma 1.8 therein.  However, this did not affect
    the validity of the core arguments in \cite{DupLasRam15v1} and,  
    since we had in the meanwhile
     identified ways to extend the paper more substantially resulting in the
    current version, which no longer
    relied on this lemma, this did not have any ramifications for the current version.
    We would also like to thank both referees of the
    present article for their valuable feedback that led to improvements
    in the exposition. 
} 

\appendix

{
	
	\section{The Weak Convergence Approach to Large Deviations}
	\label{sec:weakconv}
	
	Here, we provide a brief outline to the weak convergence approach to large deviations. 
	The weak convergence approach to large deviations was initiated in the book
	\cite{Dupuis}. Using a test function characterization and representations for
	exponential integrals in terms of relative entropy, it converts the problem of
	proving an LDP  to that of analyzing the asymptotics of
	related variational problems, with the asymptotic analysis done via weak
	convergence. One of the original motivations for the approach was the
	expectation that, since weak convengence methods are well suited to the
	analysis of problems involving nonsmoothness and singular behaviors, it would
	be a natural tool for large deviation problems with similar features, as in
	the present paper.

	A conceptual picture of the approach is as follows. Suppose we want to show a
	sequence of random variables $\{X^{n}\}$ that take values in some space
	$\mathcal{X}$ satisfies a large deviation principle with rate function
	$I:\mathcal{X}\rightarrow\lbrack0,\infty]$ and speed $\{n\}$.  Then under
	suitable structural assumptions on $\mathcal{X}$, it is enough to show that
	for every bounded and continuous function $f:\mathcal{X}\rightarrow\mathbb{R}%
	$, as $n \rightarrow \infty$, 
	\begin{equation}
	-\frac{1}{n}\log\mathbb{E}\left[e^{-f(X^{n})}\right]\rightarrow\inf_{x\in\mathcal{X}}
	\left[  f(x)+I(x)\right]  .\label{eqn:LPL}%
	\end{equation}
	Under minimal conditions for discrete index models and mild conditions for
	continuous models, one can prove convenient relative entropy-based variational
	representations of the form 
	\[
	-\frac{1}{n}\log\mathbb{E}\left[e^{-f(X^{n})}\right]=\inf\mathbb{E}\left[  f(\bar{X}^{n})+C(X^{n}:\bar{X}^{n})\right]  ,
	\]
	where $\bar{X}^{n}$ is a \textquotedblleft controlled\textquotedblright%
	\ version of $X^{n}$, $C(X^{n}:\bar{X}^{n})$ is a non-negative cost
	for perturbing from the original distribution to that of $\bar{X}^{n}$, and
	the infimum is over all possible perturbations. (For the most complete account
	of such representations, see \cite{Dupuis}.)  Typically, one can show
	that boundedness of the expected costs $\mathbb{E}\left[  C(X^{n}:\bar{X}%
	^{n})\right]$ will imply tightness of  $\{\bar{X}^{n}\}$, and when proving a
	lower bound such boundedness can be assumed without loss since $f$ is bounded.
	Assuming tightness, one then establishes the lower bound (which corresponds to
	the large deviation upper bound) by showing that if $\bar{X}^{n} \rightarrow
	\bar{X}$ in distribution, then $\liminf_{n\rightarrow\infty}
	\mathbb{E}\left[  C(X^{n}:\bar{X}^{n})\right]  \geq\mathbb{E}[I(\bar{X})]$,
	and therefore
	\[
	\liminf_{n\rightarrow\infty}\mathbb{E}\left[  f(\bar{X}^{n})+C(X^{n}:\bar{X}^{n})\right]  \geq\mathbb{E}\left[  f(\bar{X})+I(\bar
	{X})\right]  \geq\inf_{x\in\mathcal{X}}\left[  f(x)+I(x)\right]  .
	\]
	Since $\{\bar{X}^{n}\}$ is arbitrary, this gives the needed bound (in fact, in
	this part of the analysis one often identifies a candidate for the rate function).
	
	The reverse bound is obtained as follows. Given any $x^{\ast}$ that is within
	$\varepsilon>0$ of the infimum of $\inf_{x\in\mathcal{X}}\left[
	f(x)+I(x)\right]  $, one identifies controls that will drive $\bar{X}^{n}$ to
	$x^{\ast}$ (recall that large deviations uses a law of large numbers scaling),
	and with costs that satisfy $\limsup_{n\rightarrow\infty}\mathbb{E}\left[
	C(X^{n}:\bar{X}^{n})\right]  \leq I(x^{\ast})$. The reverse bound
	follows since $\varepsilon>0$ is arbitrary, and together the bounds give
	(\ref{eqn:LPL}). To successfully carry out these steps, one typically needs a
	very good understanding of the law of large numnbers analysis of the original system
	$\{X^{n}\}$, since the weak convergence analysis ends up being a law of large numbers 
	analysis of the controlled versions, and methods that are useful for the first
	problem can often be adapted to deal with the second.
}

\section{Proof of Lemma \protect\ref{phi3}}

\label{sec-apA}

The proof of Lemma \ref{phi3} is based on two preliminary results,
established in Lemma \ref{phi1} and Lemma \ref{phi2} below.

\begin{lemma}
	\label{phi1} Let $\nu\in\mathcal{P}\left( \mathbb{R}^{m}\right) $ and let $%
	\bar{\psi}:\mathbb{R}^{m}\rightarrow\mathbb{R}_{+}$ be measurable. Then 
	\begin{equation}
	\int_{\mathbb{R}^{m}}e^{\lambda\bar{\psi}\left( \mathbf{z}\right) }\nu\left(
	d\mathbf{z}\right) <\infty  \label{lambda}
	\end{equation}
	for all $\lambda<\infty$ if and only if there exists a convex, increasing
	and superlinear function $\bar{\phi}:\mathbb{R}_{+}\rightarrow\mathbb{R}$
	such that 
	\begin{equation}  \label{nu-bd}
	\int_{\mathbb{R}^{m}}e^{\bar{\phi}\left( \bar{\psi}\left( \mathbf{z}\right)
		\right) }\nu\left( d\mathbf{z}\right) <\infty.
	\end{equation}
\end{lemma}

\begin{proof}
	($\Rightarrow$) If \eqref{lambda} holds, for every $k\in\mathbb{N}$ we can
	find $M_{k}\in(0,\infty)$ such that
	\[
	\int_{\{\mathbf{z}:\,\bar{\psi}(\mathbf{z})\geq M_{k}\}}e^{k\bar{\psi}\left(
		\mathbf{z}\right)  }\nu\left(  d\mathbf{z}\right)  <\frac{1}{2^{k}}.
	\]
	Without loss of generality, we can assume $M_{k+1}\geq M_{k}$, and
	$\lim_{k\rightarrow\infty}M_{k}=\infty.$ We then define a continuous function
	$\bar{\phi}$ according to $d\bar{\phi}\left(  s\right)  /ds=k,s\in (M_{k},M_{k+1})$, and $\bar{\phi}\left(  s\right)  =M_{1},s\in\lbrack0,M_{1}]$,
	which implies $\lim_{s\rightarrow\infty}\frac{\bar{\phi}\left(  s\right)  }%
	{s}=\infty$ and also that $\bar{\phi}$ is convex and increasing. Finally, we
	have
	\[
	\int_{\mathbb{R}^{m}}e^{\bar{\phi}\left(  \bar{\psi}\left(  \mathbf{z}\right)
		\right)  }\nu\left(  d\mathbf{z}\right)  \leq e^{M_{1}}+\sum_{k=1}^{\infty
	}\int_{\{\mathbf{z}:\,\bar{\psi}(\mathbf{z})\geq M_{k}\}}e^{k\bar{\psi}\left(
		\mathbf{z}\right)  }\nu\left(  d\mathbf{z}\right)  \leq e^{M_{1}}+\sum
	_{k=1}^{\infty}\frac{1}{2^{k}}<\infty.
	\]

	($\Leftarrow$) Let $\bar{\phi}$ be as in the statement of the lemma. Since
	$\bar{\phi}$ satisfies $\lim_{s\rightarrow\infty}\frac{\bar{\phi}\left(
		s\right)  }{s}=\infty$, for every $\lambda<\infty$ there exists $M_{\lambda
	}<\infty$ such that $\bar{\phi}\left(  s\right)  \geq$ $\lambda s$ if $s\geq
	M_{\lambda}$. Then we have
	\begin{align*}
	\int_{\mathbb{R}^{m}}e^{\lambda\bar{\psi}\left(  \mathbf{z}\right)  }%
	\nu\left(  d\mathbf{z}\right)   &  =\int_{\mathbb{R}^{m}}1_{\{\bar{\psi
		}\left(  \mathbf{z}\right)  <M_{\lambda}\}}e^{\lambda\bar{\psi}\left(
		\mathbf{z}\right)  }\nu\left(  d\mathbf{z}\right)  +\int_{\mathbb{R}^{m}%
	}1_{\{\bar{\psi}\left(  \mathbf{z}\right)  \geq M_{\lambda}\}}e^{\lambda
		\bar{\psi}\left(  \mathbf{z}\right)  }\nu\left(  d\mathbf{z}\right) \\
	&  \leq e^{\lambda M_{\lambda}}+\int_{\mathbb{R}^{m}}e^{\bar{\phi}\left(
		\bar{\psi}\left(  \mathbf{z}\right)  \right)  }\nu\left(  d\mathbf{z}\right)
	<\infty.
	\end{align*}
	
\end{proof}

\begin{lemma}
	\label{phi2} Let $\nu\in\mathcal{P}\left( \mathbb{R}^{m}\right) $ and let $%
	\bar{\psi}:\mathbb{R}^{m}\rightarrow\mathbb{R}_{+}$ be measurable. Then 
	\begin{equation}
	\int_{\mathbb{R}^{m}}e^{\lambda\bar{\psi}\left( \mathbf{z}\right) }\nu\left(
	d\mathbf{z}\right) <\infty  \label{lambda2}
	\end{equation}
	for all $\lambda<\infty$ if and only if there exists a convex, increasing
	and superlinear function $\bar{\phi}:\mathbb{R}_{+}\rightarrow\mathbb{R}$
	and a constant $C<\infty$ such that for any $\mu\in\mathcal{P}\left( \mathbb{%
		R}^{m}\right) $, 
	\begin{equation}
	\int_{\mathbb{R}^{m}}\bar{\phi}\left( \bar{\psi}\left( \mathbf{z}\right)
	\right) \mu\left( d\mathbf{z}\right) \leq\mathcal{R}\left( \mu|\nu\right) +C.
	\label{tigh}
	\end{equation}
\end{lemma}

\begin{proof}
	($\Rightarrow$) First assume that \eqref{lambda2} holds. Then by the previous
	lemma there exists a positive convex function $\bar{\phi}:\mathbb{R}%
	\rightarrow\mathbb{R},$ with $\lim_{s\rightarrow\infty}\frac{\bar{\phi}\left(
		s\right)  }{s}=\infty$ such that \eqref{nu-bd} holds. 
	Since $-\bar{\phi} \leq0$, by using Proposition 4.5.1 in  \cite{Dupuis} with
	$k=-\bar{\phi}$, we get
	\begin{equation}
	\sup_{\mu\in\mathcal{P}\left(  \mathbb{R}^{m}\right)  :\mathcal{R}\left(
		\mu|\nu\right)  <\infty}\left\{  \int_{\mathbb{R}^{m}}\bar{\phi}\left(
	\bar{\psi}\left(  \mathbf{z}\right)  \right)  \mu\left(  d\mathbf{z}\right)
	-\mathcal{R}\left(  \mu|\nu\right)  \right\}  =\log\int_{\mathbb{R}^{m}%
	}e^{\bar{\phi}\left(  \bar{\psi}\left(  \mathbf{z}\right)  \right)  }%
	\nu\left(  d\mathbf{z}\right)  <\infty, \label{DV}%
	\end{equation}
	from which we obtain
	\[
	\int_{\mathbb{R}^{m}}\bar{\phi}\left(  \bar{\psi}\left(  \mathbf{z}\right)
	\right)  \mu\left(  d\mathbf{z}\right)  \leq\mathcal{R}\left(  \mu|\nu\right)
	+\log\int_{\mathbb{R}^{m}}e^{\bar{\phi}\left(  \bar{\psi}\left(
		\mathbf{z}\right)  \right)  }\nu\left(  d\mathbf{z}\right)
	\]
	for all $\mu\in\mathcal{P}\left(  \mathbb{R}^{m}\right)  $ with $\mathcal{R}%
	\left(  \mu|\nu\right)  <\infty$. Thus, \eqref{tigh} follows.
	
	($\Leftarrow$) For the converse, if we assume that \eqref{tigh} is true, then
	we have
	\[
	\sup_{\mu\in\mathcal{P}\left(  \mathbb{R}^{m}\right)  }\left\{  \int
	_{\mathbb{R}^{m}}\bar{\phi}\left(  \bar{\psi}\left(  \mathbf{z}\right)
	\right)  \mu\left(  d\mathbf{z}\right)  -\mathcal{R}\left(  \mu|\nu\right)
	\right\}  \leq C,
	\]
	and \eqref{DV} implies that $\log\int_{\mathbb{R}^{m}}e^{\bar{\phi}\left(
		\bar{\psi}\left(  \mathbf{z}\right)  \right)  }\nu\left(  d\mathbf{z}\right)
	$ is bounded, which proves \eqref{lambda2}.
\end{proof}

\begin{proof}
	[Proof of Lemma \ref{phi3}]Consider the probability measure on $\mathbb{R}%
	^{d}\times\mathbb{R}^{d}$ defined by
	\[
	\nu(d\mathbf{x}d\mathbf{y})=\frac{1}{Z}e^{-\left(  V\left(  \mathbf{x}\right)
		+V\left(  \mathbf{y}\right)  +W\left(  \mathbf{x},\mathbf{y}\right)  \right)
	}\ell(d\mathbf{x})\ell(d\mathbf{y}),
	\]
	where $Z$ is the normalization constant that makes $\nu$ a probability
	measure; the finiteness of $Z$ follows on setting $\lambda=0$ in
	\eqref{alalal}. Since $\psi$ satisfies \eqref{alalal}, we can apply Lemma
	\ref{phi2} with $\bar{\psi}(\mathbf{x},\mathbf{y})=\psi(\mathbf{x}%
	)+\psi(\mathbf{y})$ to conclude that there exists a convex and increasing
	function $\bar{\phi} :\mathbb{R}_{+}\mapsto\mathbb{R}$ with $\lim
	_{s\rightarrow\infty}\bar{\phi}(s)/s=\infty$ such that for any $\zeta
	\in\mathcal{P}(\mathbb{R}^{d}\times\mathbb{R}^{d})$,
	\begin{equation}
	\label{last display}\int_{\mathbb{R}^{d}\times\mathbb{R}^{d}}\bar{\phi}\left(
	\psi\left(  \mathbf{x}\right)  +\psi\left(  \mathbf{y}\right)  \right)
	\zeta\left(  d\mathbf{x}d\mathbf{y}\right)  \leq\mathcal{R}\left(
	\zeta|e^{-\left(  V\left(  \mathbf{x}\right)  +V\left(  \mathbf{y}\right)
		+W\left(  \mathbf{x},\mathbf{y}\right)  \right)  }\ell(d\mathbf{x}%
	)\ell(d\mathbf{y})/Z\right)  +C.
	\end{equation}

	We claim, and prove below, that for every $\zeta$ we have
	\begin{equation}\label{ab}
	\int_{\mathbb{R}^{d}\times\mathbb{R}^{d}}\bar{\phi}\left(  \psi\left(
	\mathbf{x}\right)  +\psi\left(  \mathbf{y}\right)  \right)  \zeta\left(
	d\mathbf{x}d\mathbf{y}\right)  \leq2\mathfrak{J}_{a}\left(  \zeta\right)
	+\mathcal{R}\left(  \zeta|e^{-(1-a)V}\ell\otimes e^{-(1-a)V}\ell\right)  +\log
	Z + C.
	\end{equation}
	If the claim
	holds, then since $\bar{\phi}$ is increasing and since $\psi$ and
	${\mathcal{R}}$ are positive, for $i=1,2,$ we have
	\begin{equation}
	\int_{\mathbb{R}^{d}}\bar{\phi}\left(  \psi\left(  \mathbf{x}\right)  \right)
	(\pi_{\#}^{i}\zeta)\left(  d\mathbf{x}\right)  \leq2\mathfrak{J}_{a}\left(
	\zeta\right)  +2\mathcal{R}\left(  \zeta|e^{-(1-a)V}\ell\otimes e^{-(1-a)V}%
	\ell\right)  +C+\log Z,\label{above}%
	\end{equation}
	where recall from Definition \ref{projection} and Definition \ref{push} that
	$\pi_{\#}^{i}\zeta$ represents the $i$th marginal of $\zeta$.  Adding the inequality
	\eqref{above} for $i=1$ and $i=2$ we have
	\[
	\int_{\mathbb{R}^{d}}\bar{\phi}\left(  \psi\left(  \mathbf{x}\right)  \right)
	(\pi_{\#}^{1}\zeta)\left(  d\mathbf{x}\right)  +\int_{\mathbb{R}^{d}}\bar
	{\phi}\left(  \psi\left(  \mathbf{x}\right)  \right)  (\pi_{\#}^{2}%
	\zeta)\left(  d\mathbf{x}\right)  \leq4\mathfrak{J}_{a}\left(  \zeta\right)
	+4\mathcal{R}\left(  \zeta|e^{-(1-a)V}\ell\otimes e^{-(1-a)V}\ell\right)
	+2(C+\log Z).
	\]
	If $\zeta\in\Pi(\mu,\mu)$ then $\pi_{\#}^{1}\zeta=\pi_{\#}^{2}\zeta=\mu$.
	Dividing both sides by $2$, equation (\ref{m}), which is the conclusion of
	Lemma \ref{phi3}, holds with $\phi:=[\bar{\phi}-C-\log Z]/2$.
	
	We now turn to the proof of the claim \eqref{above}. We can assume without
	loss of generality that $\zeta(d\mathbf{x}d\mathbf{y})$ has a density with
	respect to the measure $e^{-(1-a)V}\ell\otimes e^{-(1-a)V}\ell$, because
	otherwise \eqref{above} holds trivially, since $W(\mathbf{x},\mathbf{y})
	+aV(\mathbf{x}) + aV(\mathbf{y})$ is bounded from below. Denoting this density
	(with some abuse of notation) by $\zeta(\mathbf{x},\mathbf{y}),$
	\eqref{last display} then gives
	\[%
	\begin{split}
	&  \int_{\mathbb{R}^{d}\times\mathbb{R}^{d}}\bar{\phi}\left(  \psi\left(
	\mathbf{x}\right)  +\psi\left(  \mathbf{y}\right)  \right)  \zeta\left(
	d\mathbf{x}d\mathbf{y}\right) \\
	&  \quad\leq\int_{\mathbb{R}^{d}\times\mathbb{R}^{d}}\zeta(\mathbf{x}
	,\mathbf{y})\log\frac{\zeta(\mathbf{x},\mathbf{y})}{e^{-W\left(
			\mathbf{x},\mathbf{y} \right)  +a \left( V\left(  \mathbf{x}\right)  +V\left(
			\mathbf{y}\right) \right)  }/Z}e^{- (1-a)\left( V\left(  \mathbf{x}\right)
		+V\left(  \mathbf{y}\right) \right) }\ell(d\mathbf{x})\ell(d\mathbf{y})+C\\
	&  \quad\leq\int_{\mathbb{R}^{d}\times\mathbb{R}^{d}}\left( W\left(
	\mathbf{x},\mathbf{y} \right)  +a \left( V\left(  \mathbf{x}\right)  +V\left(
	\mathbf{y}\right) \right)  \right) \zeta(\mathbf{x},\mathbf{y})e^{-
		(1-a)\left( V\left(  \mathbf{x}\right)  +V\left(  \mathbf{y}\right) \right)
	}\ell(d\mathbf{x})\ell(d\mathbf{y})\\
	& \quad\quad+\int_{\mathbb{R}^{d}\times\mathbb{R}^{d}}\zeta(\mathbf{x}%
	,\mathbf{y})\log\zeta(\mathbf{x},\mathbf{y})e^{-(1-a)\left(  V\left(
		\mathbf{x}\right)  +V\left(  \mathbf{y}\right)  \right)  }\ell(d\mathbf{x}%
	)\ell(d\mathbf{y})+\log Z+C.
	\end{split}
	\]
	Therefore, recalling the definition of $\mathfrak{J}_{a}$ in \eqref{scr}, we
	have
	\[
	\int_{\mathbb{R}^{d}\times\mathbb{R}^{d}}\bar{\phi}\left(  \psi\left(
	\mathbf{x}\right)  +\psi\left(  \mathbf{y}\right)  \right)  \zeta\left(
	d\mathbf{x}d\mathbf{y}\right)  \leq2\mathfrak{J}_{a}\left(  \zeta\right)
	+\mathcal{R}\left(  \zeta|e^{-(1-a)V}\ell\otimes e^{-(1-a)V}\ell\right)  +\log
	Z + C,
	\]
	which completes the proof of the claim, and therefore the lemma.
\end{proof}

\section{Proof of Lemma \protect\ref{tightnessfunction}}

\label{sec-apB}

We first establish a preliminary result in Lemma \ref{Polish} below. Let $%
B(0,r)$ denote the closed ball about $0$ of radius $r$, and let $B^{c}(0,r)$
denote its complement.

\begin{lemma}
	\label{Polish} Let $\psi,\mathcal{P}_{\psi}(\mathbb{R}^{d}),$ and $d_{\psi}$
	be defined as in (\ref{cond-psi})-(\ref{psidef}). Then $d_{\psi}(\mu_{n},%
	\mu)\rightarrow0$ as $n\rightarrow\infty$ if and only if 
	\begin{equation}
	d_{w}(\mu_{n},\mu)\rightarrow0\hspace{4pt}\text{and}\hspace{4pt}%
	\lim_{r\rightarrow\infty}\sup_{n}\left\{ \int_{B^{c}(0,r)}\psi(\mathbf{x}%
	)\mu_{n}(d\mathbf{x})\right\} =0.  \label{equiv-conv}
	\end{equation}
	Furthermore, the metric space $(\mathcal{P}_{\psi}(\mathbb{R}^{d}),d_{\psi})$
	is separable.
\end{lemma}

\begin{proof}
	($\Rightarrow$). Let $\mu_{n},n\in\mathbb{N}$, $\mu\in\mathcal{P}_{\psi
	}(\mathbb{R}^{d})$ be such that $d_{\psi}(\mu_{n},\mu)\rightarrow0.$ Since
	$d_{w}(\mu_{n},\mu)\leq d_{\psi}(\mu_{n},\mu)$, this implies $d_{w}(\mu
	_{n},\mu)\rightarrow0$. Let $\epsilon>0.$ By the integrability of $\psi$ there
	exists $r<\infty$ such that $\int_{B^{c}(0,r)}\psi(\mathbf{x})\mu
	(d\mathbf{x})<\frac{\epsilon}{3},$ and also $\mu(\partial B(0,r))=0.$ Hence,
	we have
	\begin{equation}%
	\begin{split}
	& \int_{B^{c}(0,r)}\psi(\mathbf{x})\mu_{n}(d\mathbf{x}) =\int_{B^{c}(0,r)}%
	\psi(\mathbf{x})(\mu_{n}(d\mathbf{x})-\mu(d\mathbf{x}))+\int_{B^{c}(0,r)}%
	\psi(\mathbf{x})\mu(d\mathbf{x})\\
	&  \leq\left\vert \int_{\mathbb{R}^{d}}\psi(\mathbf{x})\mu_{n}(d\mathbf{x}%
	)-\int_{\mathbb{R}^{d}}\psi(\mathbf{x})\mu(d\mathbf{x})\right\vert +\left\vert
	\int_{B(0,r)}\psi(\mathbf{x})\mu_{n}(d\mathbf{x})-\int_{B(0,r)}\psi
	(\mathbf{x})\mu(d\mathbf{x})\right\vert +\frac{\epsilon}{3}.
	\end{split}
	\label{bigll}%
	\end{equation}
	From the definition of $d_{\psi}$ in \eqref{psidef} and the nonnegativity of
	$d_{w}$, we can find $n_{0}\in\mathbb{N}$ such that $\forall n>n_{0},$ we
	have
	\[
	\left\vert \int_{\mathbb{R}^{d}}\psi(\mathbf{x})\mu_{n}(d\mathbf{x}%
	)-\int_{\mathbb{R}^{d}}\psi(\mathbf{x})\mu(d\mathbf{x})\right\vert
	<\frac{\epsilon}{3}.
	\]
	Since $\mu(\partial B(0,r))=0$, the $\mu$-measure of the discontinuity points
	of $\mathbf{x}\rightarrow\psi(\mathbf{x})1_{B(0,r)}(\mathbf{x})$ is zero.
	Since $\psi(\mathbf{x})$ can be extended outside of $B(0,r)$ to obtain a
	bounded and continuous function on $\mathbb{R}^{d}$, the fact that $d_{w}%
	(\mu_{n},\mu)\rightarrow0$ implies that there exists $n_{0}^{\prime}<\infty$
	such that $\forall n\geq n_{0}^{\prime}$,
	\begin{equation}
	\left\vert \int_{B(0,r)}\psi(\mathbf{x})\mu_{n}(d\mathbf{x})-\int_{B(0,r)}%
	\psi(\mathbf{x})\mu(d\mathbf{x})\right\vert <\frac{\epsilon}{3}. \label{twra}%
	\end{equation}
	Combining the above estimates for all terms in \eqref{bigll} we obtain
	\[
	\sup_{n\geq\max\{n_{0},n_{0}^{\prime}\}}\left\{  \int_{B^{c}(0,r)}%
	\psi(\mathbf{x})\mu_{n}(d\mathbf{x})\right\}  <\epsilon.
	\]
	Since $\psi$ is integrable with respect to each $\mu_{n}$, for all $n\leq
	\max\{n_{0},n_{0}^{\prime}\}$ we can find an $r_{n}<\infty$ such that
	$\int_{B^{c}(0,r_{n})}\psi(\mathbf{x})\mu_{n}(d\mathbf{x})<\epsilon.$ Taking
	$r^{\prime}=\max\{r_{1},...,r_{\max\{n_{0},n_{0}^{\prime}\}},r\}$ yields
	\[
	\sup_{n}\left\{  \int_{B^{c}(0,r^{\prime})}\psi(\mathbf{x})\mu_{n}%
	(d\mathbf{x})\right\}  <\epsilon.
	\]
	Since $\epsilon$ is arbitrary, the conclusion follows.
	
	($\Leftarrow$) To prove the converse, let $\mu_{n},n\in\mathbb{N}$, $\mu
	\in\mathcal{P}_{\psi}(\mathbb{R}^{d})$, be such that \eqref{equiv-conv}
	holds.
	For $\epsilon>0$ there exists $r<\infty$ such that $\mu(\partial B(0,r))=0$
	and
	\[
	\sup_{n}\left\{  \int_{B^{c}(0,r)}\psi(\mathbf{x})\mu_{n}(d\mathbf{x}%
	)\right\}  <\frac{\epsilon}{3}\hspace{12pt}\text{and}\hspace{12pt}\int
	_{B^{c}(0,r)}\psi(\mathbf{x})\mu(d\mathbf{x})<\frac{\epsilon}{3},
	\]
	where the latter inequality holds because $\mu\in{\mathcal{P}}_{\psi}$ implies
	that $\psi$ is $\mu$-integrable. Thus, we have
	\begin{align}
	\left\vert \int_{\mathbb{R}^{d}}\psi(\mathbf{x})\mu_{n}(d\mathbf{x}%
	)-\int_{\mathbb{R}^{d}}\psi(\mathbf{x})\mu(d\mathbf{x})\right\vert  &
	\leq\left\vert \int_{B(0,r)}\psi(\mathbf{x})\mu_{n}(d\mathbf{x})-\int
	_{B(0,r)}\psi(\mathbf{x})\mu(d\mathbf{x})\right\vert
	\nonumber\label{psimu-ineq}\\
	&  +\left\vert \int_{B^{c}(0,r)}\psi(\mathbf{x})\mu_{n}(d\mathbf{x}%
	)-\int_{B^{c}(0,r)}\psi(\mathbf{x})\mu(d\mathbf{x})\right\vert \\
	&  \leq\left\vert \int_{B(0,r)}\psi(\mathbf{x})\mu_{n}(d\mathbf{x}%
	)-\int_{B(0,r)}\psi(\mathbf{x})\mu(d\mathbf{x})\right\vert +\frac{2\epsilon
	}{3}.\nonumber
	\end{align}
	Since $d_{w}(\mu_{n},\mu)\rightarrow0$ and $\mu$ puts no mass on the set of
	discontinuities of the bounded function $\psi(\mathbf{x})1_{B(0,r)}%
	(\mathbf{x})$, there exists $n_{0}^{\prime}<\infty$ such that
	\[
	\left\vert \int_{B(0,r)}\psi(\mathbf{x})\mu_{n}(d\mathbf{x})-\int_{B(0,r)}%
	\psi(\mathbf{x})\mu(d\mathbf{x})\right\vert <\frac{\epsilon}{3},\qquad\forall
	n\geq n_{0}^{\prime}.
	\]
	Since $\epsilon$ is arbitrary, when substituted back into \eqref{psimu-ineq},
	this shows that
	\[
	\lim_{n\rightarrow\infty}\left\vert \int_{\mathbb{R}^{d}}\psi(\mathbf{x}%
	)\mu_{n}(d\mathbf{x})-\int_{\mathbb{R}^{d}}\psi(\mathbf{x})\mu(d\mathbf{x}%
	)\right\vert =0.
	\]

	We now turn to the proof that $\mathcal{P}_{\psi}(\mathbb{R}^{d})$ is
	separable. Let $\{\boldsymbol{x}_{n}\}$ be a countable dense subset of
	$\mathbb{R}^{d}$, and define
	\[
	{\mathcal{A}}:=\left\{  \sum_{i=1}^{N}c_{n}\delta_{\boldsymbol{x}_{n}}%
	:c_{n}\in\mathbb{Q}_{+},n=1,\ldots,\mathbb{N},\sum_{n=1}^{N}c_{n}=1,\sum
	_{n=1}^{N}c_{n}\psi(\boldsymbol{x}_{n})<\infty,N\in\mathbb{N}\right\}  ,
	\]
	where $\mathbb{Q}_{+}$ is the set of nonnegative rational numbers, and observe
	that $\mathcal{A}$ is a countable subset of ${\mathcal{P}}_{\psi}$. We now
	show that $\mathcal{A}$ is dense in ${\mathcal{P}}_{\psi}$. Fix $\mu
	\in{\mathcal{P}}_{\psi}$ and $\varepsilon>0$. Also, consider the space
	$\mathbb{F}$ of bounded, Lipschitz continuous functions on $\mathbb{R}^{d}$,
	equipped with the norm
	\[
	||f||_{BL}:=\max\left(  \sup_{\mathbf{x},\mathbf{y}\in\mathbb{R}%
		^{d},\mathbf{x}\neq\mathbf{y}}\frac{|f(\mathbf{x})-f(\mathbf{y})|}%
	{|\mathbf{x}-\mathbf{y}|},2\sup_{\mathbf{x}\in\mathbb{R}^{d}}|f(\mathbf{x}%
	)|\right)  ,
	\]
	and let $\mathbb{F}_{1}$ be the subspace of functions with $||f||_{BL}\leq1$.
	Then consider the metric on ${\mathcal{P}}(\mathbb{R}^{d})$ given by
	\[
	d_{BL}(\mu,\nu):=\sup_{f\in\mathbb{F}_{1}}\left\vert \int_{\mathbb{R}^{d}%
	}f(\mathbf{x})\mu(\mathbf{dx})-\int_{\mathbb{R}^{d}}f(\mathbf{x}%
	)\nu(\mathbf{dx})\right\vert .
	\]
	In view of the definition of $d_{\psi}$ in \eqref{psidef} and the fact that
	there exists a constant $C<\infty$ such that $d_{w}(\mu,\nu)\leq C\sqrt
	{d_{BL}(\mu,\nu)}$ (see \cite[p.\ 396]{Dud02}), it suffices to show that there
	exists $\nu\in{\mathcal{A}}$ such that
	\begin{equation}
	\sup_{f\in\mathbb{F}_{1}\cup\{\psi\}}\left\vert \int_{\mathbb{R}^{d}%
	}f(\mathbf{x})\mu(d\mathbf{x})-\int_{\mathbb{R}^{d}}f(\mathbf{x}%
	)\nu(d\mathbf{x})\right\vert \leq\varepsilon.\label{sep-toshow}%
	\end{equation}

	Recalling that $\psi$ is continuous, for each $n\in\mathbb{N}$, choose
	$r_{n}\in(0,\varepsilon/2)$ such that
	\begin{equation}
	\sup_{\mathbf{x}\in B_{r_{n}}(\mathbf{x}_{n})}|\psi(\mathbf{x}%
	)-\psi(\mathbf{x}_{n})|\leq\frac{\varepsilon}{2},\label{sep-ineq1}%
	\end{equation}
	and note that then we also have
	\begin{equation}
	\sup_{\mathbf{x}\in B_{r_{n}}(\mathbf{x}_{n})}|f(\mathbf{x}%
	)-f(\mathbf{x}_{n})|\leq r_{n}\leq\frac{\varepsilon}{2},\quad
	f\in\mathbb{F}_{1}.\label{sep-ineq1.5}%
	\end{equation}
	Now define $\tilde{B}_{n}:=B_{r_{n}}(\mathbf{x}_{n})\setminus\cup
	_{k=1}^{n-1}B_{r_{k}}(\mathbf{x}_{k})$ and $b_{n}:=\mu(\tilde{B}_{n})$.
	Clearly, $\{\tilde{B}_{n}\}_{n\in\mathbb{N}}$ forms a disjoint partition of
	$\mathbb{R}^{d}$ and hence, $\sum_{n=1}^{\infty}b_{n}=1$. Moreover, by
	\eqref{sep-ineq1} and \eqref{sep-ineq1.5} we have for all $f\in\mathbb{F}%
	_{1}\cup\{\psi\}$,
	\begin{equation}
	\left\vert \sum_{n=1}^{\infty}b_{n}f(\mathbf{x}_{n})-\int_{\mathbb{R}^{d}%
	}f(\mathbf{x})\mu(d\mathbf{x})\right\vert \leq\sum_{n=1}^{\infty}b_{n}%
	\sup_{\mathbf{x}\in\tilde{B}_{n}}|f(\mathbf{x}_{n})-f(\mathbf{x}%
	)|\leq\frac{\varepsilon}{2}.\label{sep-ineq2}%
	\end{equation}
	We can assume without loss of generality that $\psi$ is uniformly bounded from
	below away from zero. Since $\int_{\mathbb{R}^{d}}\psi(\mathbf{x}%
	)\mu(\mathbf{dx})$ is finite, this implies $\sum_{n=1}^{\infty}b_{n}%
	\psi(\mathbf{x}_{n})<\infty$, and hence there exists $N\in\mathbb{N}$ such
	that
	\begin{equation}
	\sum_{n=N+1}^{\infty}b_{n}\leq\frac{\varepsilon}{8(\psi(\mathbf{x}%
		_{1})\vee1)}\quad\mbox{ and  }\quad\sum_{n=N+1}^{\infty}b_{n}\psi
	(\mathbf{x}_{n})\leq\frac{\varepsilon}{8}.\label{sep-ineq3}%
	\end{equation}
	Now, for $n=2,\ldots,N$, choose $c_{n}\in\mathbb{Q}_{+}$ such that
	\begin{equation}
	0\leq b_{n}-c_{n}\leq\left(  \frac{b_{n}}{\max(|\psi(\mathbf{x}_{1}%
		)+\psi(\mathbf{x}_{n})|,|\mathbf{x}_{n}-\mathbf{x}_{1}|)}\right)
	\frac{\varepsilon}{4},\label{sep-ineq4}%
	\end{equation}
	and set
	\[
	c_{1}:=b_{1}+\sum_{n=2}^{N}(b_{n}-c_{n})+\sum_{n=N+1}^{\infty}b_{n}.
	\]
	Observe that $\sum_{n=1}^{N}c_{n}=\sum_{n=1}^{\infty}b_{n}=1$, and hence,
	$c_{1}$ also lies in $\mathbb{Q}_{+}$. Set $\nu:=\sum_{n=1}^{N}c_{n}%
	\delta_{\mathbf{x}_{n}}$.
	Then, for $f\in\mathbb{F}_{1}\cup\{\psi\}$, using \eqref{sep-ineq4} and
	\eqref{sep-ineq3}, we have
	\begin{align*}
	\left\vert \int_{\mathbb{R}^{d}}f(\mathbf{x})\nu(d\mathbf{x}%
	)-\sum_{n=1}^{\infty}b_{n}f(\mathbf{x}_{n})\right\vert  &  =\left\vert
	\sum_{n=1}^{N}c_{n}f(\mathbf{x}_{n})-\sum_{n=1}^{\infty}b_{n}%
	f(\mathbf{x}_{n})\right\vert \\
	&  \leq\sum_{n=2}^{N}(b_{n}-c_{n})|f(\mathbf{x}_{n})-f(\mathbf{x}%
	_{1})|+\sum_{n=N+1}^{\infty}b_{n}|f(\mathbf{x}_{1})-f(\mathbf{x}%
	_{n})|\\
	&  \leq\frac{\varepsilon}{4}+|f(\mathbf{x}_{1})|\sum_{n=N+1}^{\infty}%
	b_{n}+\sum_{n=N+1}^{\infty}b_{n}|f(\mathbf{x}_{n})|\leq\frac{\varepsilon
	}{2}.
	\end{align*}
	When combined with \eqref{sep-ineq2} this establishes the desired inequality \eqref{sep-toshow}.
\end{proof}

\begin{proof}
	[Proof of Lemma \ref{tightnessfunction}]Let $C<\infty$ and let $\{\mu
	_{n}\}\subset\mathcal{P}_{\psi}(\mathbb{R}^{d})$ be a sequence such that
	$\mathcal{T}(\mu_{n})\leq C$ for all $n$. Now $\lim_{c\rightarrow\infty}%
	\inf_{\mathbf{x}:\left\Vert \mathbf{x}\right\Vert =c}\phi(\psi(\mathbf{x}%
	))=\infty$ because $\lim_{c\rightarrow\infty}\inf_{\mathbf{x}:\left\Vert
		\mathbf{x}\right\Vert =c}\psi\left(  \mathbf{x}\right)  =\infty$ and
	$\lim_{s\rightarrow\infty}\frac{\phi\left(  s\right)  }{s}=\infty.$ Hence, by
	Lemma \ref{tightness} with $g=\phi\circ\psi$, the sequence $\{\mu_{n}\}$ is
	tight in the weak topology, and we have
	\begin{align*}
	&  \lim_{r\rightarrow\infty}\sup_{n}\left\{  \int_{B^{c}(0,r)}\psi
	(\mathbf{x})\mu_{n}(d\mathbf{x})\right\}  \\
	&  \quad=\lim_{r\rightarrow\infty}\sup_{n}\left\{  \int_{B^{c}(0,r)}\phi
	(\psi(\mathbf{x}))\frac{\psi(\mathbf{x})}{\phi(\psi(\mathbf{x}))}\mu
	_{n}(d\mathbf{x})\right\}  \\
	&  \quad\leq\lim_{r\rightarrow\infty}\sup_{n}\left\{  \left(  \sup_{x\in
		B^{c}(0,r)}\frac{\psi(\mathbf{x})}{\phi(\psi(\mathbf{x}))}\right)  \int
	_{B^{c}(0,r)}\phi(\psi(\mathbf{x}))\mu_{n}(d\mathbf{x})\right\}  \leq
	C\lim_{r\rightarrow\infty}\sup_{\mathbf{x}\in B^{c}(0,r)}\frac{\psi
		(\mathbf{x})}{\phi(\psi(\mathbf{x}))}=0.
	\end{align*}
	Thus, by the first assertion of Lemma \ref{Polish}, $\{\mu_{n}\}$ is tight in
	$\mathcal{P}_{\psi}(\mathbb{R}^{d}).$
\end{proof}

\section{Tightness Results}

\label{sec-apC}

\begin{proof}
	[Proof of Lemma \ref{lem:weak_limit_Sanov}]Since $\mathbb{R}^{d}$ is a
	Polish\vspace{0pt} space, to verify weak convergence of a sequence of
	measures in $\mathcal{P}%
	(\mathbb{R}^{d})$ it suffices to consider convergence of integrals
	with respect to the measures of functions $f$ that are uniformly
	continuous. We use the fact \cite[Lemma 3.1.4]{Stroock2010} that there is an
	equivalent metric $m$ on $\mathbb{R}^{d}$, such that if $\mathcal{U}%
	_{b}(\mathbb{R}^{d},m)$ is the space of bounded uniformly continuous functions
	with respect to this metric, then there is a countable dense subset $\left\{
	f_{m}\right\}  _{m\in\mathbb{N}}\subset\mathcal{U}_{b}(\mathbb{R}^{d},m)$.
	Define $K_{m}:=\sup_{\mathbf{x}\in \mathbb{R}^d}\left\vert f_{m}\left(  \mathbf{x}\right)  \right\vert $ and
	$\Delta_{m,i}^{n}:= f_{m}\left(  \bar{\mathbf{X}}_{i}^{n}\right)  -%
	{\textstyle\int_{\mathbb{R}^{d}}}
	f_{m}\left(  \mathbf{x}\right)  \bar{\mu}_{i}^{n}\left(  d\mathbf{x}\right)  $. For any
	$\varepsilon>0$, Chebyshev's inequality shows that 
	\begin{equation*}
	\mathbb{P}\left[  \left\vert \frac{1}{n}\sum_{i=1}^{n}
	\int_{\mathbb{R}^{d}}
	f_{m}\left(  \mathbf{x}\right)  \delta_{\bar{\mathbf{X}}_{i}^{n}}\left(  d\mathbf{x}\right)  -\frac{1}{n}\sum_{i=1}^{n}
	\int_{\mathbb{R}^{d}}f_{m}\left(  x\right)  \bar{\mu}_{i}^{n}\left(  dx\right)  \right\vert>\varepsilon\right]  \leq\frac{1}{\varepsilon^{2}}\mathbb{E}\left[  \frac{1}{n^{2}}\sum_{i,j=1}^{n}%
	\Delta_{m,i}^{n}\Delta_{m,j}^{n}\right]  .
	\end{equation*}
	Let
	$\mathcal{F}_{j}^{n}=\sigma(\bar{\mathbf{X}}_{i}^{n},i=1,\ldots,j)$.
	As we show below, by a standard
	conditioning argument, the off-diagonal terms vanish: for $i>j,$\
	\[
	\mathbb{E}\left[  \Delta_{m,i}^{n}\Delta_{m,j}^{n}\right]  =\mathbb{E}\left[  \mathbb{E}\left[  \left.
	\Delta_{m,i}^{n}\Delta_{m,j}^{n}\right\vert \mathcal{F}_{i}^{n}\right]
	\right]  =\mathbb{E}\left[ \mathbb{E}\left[  \left.  \Delta_{m,i}^{n}\right\vert \mathcal{F}%
	_{i}^{n}\right]  \Delta_{m,j}^{n}\right]  =0.
	\]
	Since $|\Delta_{m,i}^{n}|\leq2K_{m}$,
	\[
	\mathbb{P}\left[  \left\vert \frac{1}{n}\sum_{i=1}^{n}%
	{\displaystyle\int_{\mathbb{R}^{d}}}
	f_{m}\left(  x\right)  \delta_{\bar{X}_{i}^{n}}\left(  dx\right)  -\frac{1}%
	{n}\sum_{i=1}^{n}%
	{\displaystyle\int_{\mathbb{R}^{d}}}
	f_{m}\left(  \mathbf{x}\right)  \bar{\mu}_{i}^{n}\left(  d\mathbf{x}\right)  \right\vert
	>\varepsilon\right]  \leq\frac{4K_{m}^{2}}{n\varepsilon^{2}}\text{.}%
	\]
	Since $(\bar{L}^{n},\hat{\mu}^{n})\Rightarrow\left(  \bar{L},\hat{\mu}\right)
	$ and $\varepsilon>0$\ is arbitrary, by Fatou's lemma,  
	\[
	\mathbb{P}\left[
	{\displaystyle\int_{\mathbb{R}^{d}}}
	f_{m}\left(  x\right)  \bar{L}\left(  dx\right)  =%
	{\displaystyle\int_{\mathbb{R}^{d}}}
	f_{m}\left(  x\right)  \hat{\mu}\left(  dx\right)  \right]  =1.
	\]
	Now use the property that $\left\{  f_{m},m\in\mathbb{N}\right\}  $\ is
	countable and dense to conclude that $\bar{L}=\hat{\mu}$ a.s.
\end{proof}

\section{Auxiliary Lemmas}

\begin{lemma}
	\label{approx} For every $\mu\in\mathcal{P}(\mathbb{R}^{d})$ with no atoms,
	there exists a sequence $\{\mathbf{x}^{n}\}_{n \in \mathbb{N}} \subset (%
	\mathbb{R}^{dn})_{\neq}$ such that $\mathcal{J}_{\neq}(L_{n}(\mathbf{x}%
	^{n};\cdot))\rightarrow\mathcal{J}(\mu).$
\end{lemma}

\begin{proof}
	Let $\{\mathbf{X}_{n}\}_{n\in\mathbb{N}}$ be a sequence of independent
	$\mathbb{R}^{d}$-valued random variables with common law $\mu.$ For every
	$n\in\mathbb{N}$ let $\mathbf{X}^{n}:=(\mathbf{X}_{1},\ldots,\mathbf{X}_{n})$,
	and denote $L_{n}(\mathbf{X}^{n},\cdot)$ simply by $L_{n}$. Then we have
	\begin{align*}
	\mathbb{E}[\mathcal{J}_{\neq}(L_{n})]  & =\frac{1}{2}\mathbb{E}\left[
	\int_{(\mathbb{R}^{d}\times\mathbb{R}^{d})_{\neq}}\left(  V\left(
	\mathbf{x}\right)  +V\left(  \mathbf{y}\right)  +W\left(  \mathbf{x}%
	,\mathbf{y}\right)  \right)  L_{n}\left(  d\mathbf{x}\right)  L_{n}\left(
	d\mathbf{y}\right)  \right]  \\
	& =\mathbb{E}\left[  \frac{1}{2n^{2}}\sum_{i=1}^{n}\sum_{j=1,j\neq i}\left(
	V\left(  \mathbf{X}_{i}\right)  +V\left(  \mathbf{X}_{j}\right)  +W\left(
	\mathbf{X}_{i},\mathbf{X}_{j}\right)  \right)  \right]  \\
	& =\mathbb{E}\left[  \frac{1}{2n^{2}}\sum_{i=1}^{n}\sum_{j=1,j\neq i}\frac
	{1}{2}\int_{\mathbb{R}^{d}\times\mathbb{R}^{d}}\left(  V\left(  \mathbf{x}%
	\right)  +V\left(  \mathbf{y}\right)  +W\left(  \mathbf{x},\mathbf{y}\right)
	\right)  \mu\left(  d\mathbf{x}\right)  \mu\left(  d\mathbf{y}\right)
	\right]  \\
	& =\frac{n}{n-1}\mathcal{J}(\mu).
	\end{align*}
	By the Glivenko-Cantelli Lemma (or Lemma \ref{lem:weak_limit_Sanov}), $L_{n}$
	converges in distribution to the deterministic measure $\mu$. Hence using the
	Skorokhod Representation (and by introducing a new probability space if
	needed, but which we still denote as $(\Omega,\mathcal{F},\mathbb{P})$) we can
	assume the almost sure convergence to $\mu.$ By Fatou's Lemma, we have
	\[
	\mathcal{J}(\mu)=\lim_{n\rightarrow\infty}\frac{n}{n-1}\mathcal{J}%
	(\mu)=\liminf_{n\rightarrow\infty}\mathbb{E}[\mathcal{J}_{\neq}(L_{n}%
	)]\geq\mathbb{E}[\liminf_{n\rightarrow\infty}\mathcal{J}_{\neq}(L_{n}%
	)]\geq\mathbb{E}[\mathcal{J}_{\neq}(\mu)]=\mathbb{E}[\mathcal{J}(\mu)].
	\]
	From the above we get that a.s.\  $\liminf_{n\rightarrow\infty}\mathcal{J}%
	_{\neq}(L_{n})=\mathcal{J}(\mu),$ and therefore trivially there exists a
	realization $\omega$ such that by setting $\mathbf{x}_{n}=\mathbf{X}%
	_{n}(\omega)$, $\mathcal{J}_{\neq}(L_{n}(\mathbf{x}_{n};\cdot))\rightarrow
	\mathcal{J}(\mu).$
\end{proof}

\section{Proof of Lemma \protect\ref{existanceex}}

\label{ap-existenceex}

Suppose Assumptions \ref{A} and \ref{C} hold. Then there exists at least one
probability measure $\mu $ such that $J(\mu )<\infty $ (e.g., $\mu :=\ell _{|%
	{B}}/\ell (B)$) where $B\subset A$, with $A$ as defined in Assumption \ref{A}%
, satisfies $0<\ell (B)<\infty $). Also, by Assumption \ref{C}\ref{C2},
there exists a sequence $\{\mu _{n}\}$, with each $\mu _{n}\ll \ell $, such
that $\mu _{n}\overset{w}{\rightarrow }\mu $ and ${\mathcal{J}}(\mu
_{n})\rightarrow {\mathcal{J}}(\mu )$. We now argue that we can assume
without loss of generality that $\rho _{n}:=e^{(1-a)V}d\mu _{n}/d\ell $ is
uniformly bounded. For $M\in \mathbb{N}$, define $\mu
_{n}^{M}(A):=\int_{A}\rho _{n}^{M}(\mathbf{x})e^{-(1-a)V}(\mathbf{x})\ell (d%
\mathbf{x})/\int_{\mathbb{R}}\rho _{n}^{M}(\mathbf{x})e^{-(1-a)V(\mathbf{x}%
	)}\ell (d\mathbf{x}),$ where $\rho _{n}^{M}:=\min (M,\rho _{n})$ is clearly
bounded. Since $\rho _{n}^{M}$ is increasing with respect to $M$ and the map 
$(\mathbf{x},\mathbf{y})$ $\mapsto $ $W(\mathbf{x},\mathbf{y})+V(\mathbf{x}%
)+V(\mathbf{y})$ is bounded from below, by an application of the monotone
convergence theorem, $\mu _{n}^{M}\overset{w}{\rightarrow }\mu _{n}$ and $%
\mathcal{J}(\mu _{n}^{M})\rightarrow \mathcal{J}(\mu _{n})$ as $M\rightarrow
\infty $. The first claim of the lemma then follows from a standard
diagonalization argument.

Next, fix $\psi \in \Psi$ and suppose $\mu \in {\mathcal{P}}_{\psi} (\mathbb{%
	R}^d)$. We now show that the approximating sequence $\{\mu_{n}\}$ can be
taken to satisfy $d_{\psi}(\mu_{n},\mu)\rightarrow 0.$ To see this, first
assume that $\mu$ has compact support $K$, and let $\widetilde{K}$ be the
closure of $N_{\epsilon}(K)$, the $\epsilon$-neighborhood of $K$ for some $%
\epsilon > 0$. let $\{\mu_n\}$ be the approximating sequence obtained in the
first part of the lemma, and set $\tilde{\mu}_{n}(\cdot):= \mu_{n}(\cdot\cap 
\widetilde{K})/\mu_{n}(\widetilde{K})$. Note that $\tilde{\mu}_n$ is well
defined for all sufficiently large $n$ since $\lim_{n\rightarrow \infty}
\mu_n (\widetilde{K}) = \lim_{n\rightarrow \infty} \mu_n (N_{\epsilon}(K)) =
\mu (N_{\epsilon}(K)) = 1$ because $\mu_n \overset{w}{\rightarrow} \mu$, $%
\mu (\partial N_{\epsilon}(K)) = 0$, and $K$ is the support of $\mu$. 
Similarly, for any closed set $F$, $F \cap \widetilde{K}$ is also closed,
and hence by Portmanteau's theorem, $\limsup_{n} \mu_n (F \cap \widetilde{K}%
) \leq \mu(F \cap \widetilde{K}) = \mu (F)$. Since $\tilde{\mu}_n (%
\widetilde{K}) \rightarrow 1$, this implies that $\limsup_{n} \tilde{\mu}_n
(F) \leq \mu(F)$, which (by another application of Portmanteau's theorem)
shows that $\tilde{\mu}_n \overset{w}{\rightarrow} \mu$. Furthermore, since
all $\tilde{\mu}_{n}$ have the same compact support and converge to $\mu$,
and since $\psi$ is continuous, $d_{\psi} (\tilde{\mu}_{n}, \mu) \rightarrow
0$. Moreover, 
\begin{eqnarray*}
	\lim_{n \rightarrow \infty} \mathcal{J}(\tilde{\mu}_n ) \!\!&=&
	\lim_{n\rightarrow \infty }\int_{\mathbb{R}^{d}\times \mathbb{R}^{d}}\left(
	V\left( \mathbf{x}\right) +V\left( \mathbf{y}\right) +W\left( \mathbf{x},%
	\mathbf{y}\right) \right) \tilde{\mu}_{n}\left( d\mathbf{x}\right) \tilde{\mu%
	}_{n}\left( d\mathbf{y}\right) \\
	& = & \lim_{n \rightarrow \infty} \frac{1}{\mu_n^2(K)} \lim_{n\rightarrow
		\infty }\frac{1}{2}\int_{\mathbb{R}^{d}\times \mathbb{R}^{d}}\left( V\left( 
	\mathbf{x}\right) +V\left( \mathbf{y}\right) +W\left( \mathbf{x},\mathbf{y}%
	\right) \right) \mu_{n}\left( d\mathbf{x}\right) \mu_{n}\left( d\mathbf{y}%
	\right) \\
	& = & \lim_{n\rightarrow \infty }{\mathcal{J}}(\mu_n) = {\mathcal{J}}(\mu).
\end{eqnarray*}
Finally, consider an arbitrary $\mu\in\mathcal{P}_{\psi}(\mathbb{R}^{d})$
(not necessarily with compact support) and let $\nu_{n}(\cdot):= \frac{%
	\mu(\cdot\cap B(0,n))}{\mu(B(0,n))}$. By dominated convergence $%
\int_{B(0,n)}\psi \left( \mathbf{x}\right) \nu _{n}\left( d\mathbf{x}\right)
\rightarrow \int_{\mathbb{R}^{d}}\psi \left( \mathbf{x}\right) \mu\left( d%
\mathbf{x}\right) $, which shows that $d_{\psi}(\nu_{n},\mu)\rightarrow 0$.
Since we also have $\mathcal{J}(\nu_{n})\rightarrow\mathcal{J}(\mu)$, the
desired approximating measures can be found by combining the two
approximations and using a diagonalization argument.

\end{document}